\documentclass[webpdf,modern,medium,namedate]{JRSSB_arxiv} 

\onecolumn
\usepackage{setspace}
\usepackage[fontsize=12pt]{fontsize}

\makeatletter
\newcommand\figcaptionfont{\fontsize{10pt}{12pt}\sffamily\selectfont}
\long\def\@figurecaption#1#2{%
  \begingroup
  \hspace*{1pt}\vspace*{-1.5pt}\newline
  \figcaptionfont
  #1:\space{#2\strut\par}%
  \endgroup\vspace{\belowcaptionskip}}

\newcommand\tabcaptionfont{\fontsize{10pt}{12pt}\sffamily\selectfont}
\long\def\@tablecaption#1#2{%
  \begingroup
  \tabcaptionfont
  \textbf{#1.}\space{#2\strut\par}%
  \endgroup\vspace{\belowcaptionskip}}
\makeatother

\usepackage{amsfonts}
\usepackage{amsbsy}
\usepackage{calrsfs}
\usepackage{bbold}
\usepackage{mwe}
\usepackage{soul}
\usepackage{xr}
\usepackage{bbm}
\usepackage{bm}
\usepackage{cancel}
\usepackage{booktabs}
\usepackage{longtable}
\usepackage[normalem]{ulem}
\hypersetup{bookmarksdepth=3}

\newcommand{\parhead}[1]{%
  \par\smallskip
  \noindent{\fontsize{10}{12}\selectfont\sffamily\bfseries #1}%
  \par\nobreak\smallskip
}

\makeatletter
\def\subsecsize{\sffamilyfontbold\fontsize{10}{12}\selectfont}
\def\subsection{\@startsection{subsection}{2}{\z@}%
  {-10\p@}{3\p@}{\reset@font\raggedright\subsecsize}}
\makeatother

\makeatletter
\def\@itemize[#1]{%
  \ifnum \@itemdepth >3 \@toodeep\else
    \advance\@itemdepth\@ne
    \edef\@itemctr{item\romannumeral\the\@itemdepth}%
    \list{\csname label\@itemctr\endcsname}{%
      \itemargs
      \settowidth\labelwidth{\csname label\@itemctr\endcsname}%
      \setlength\itemindent{\z@}%
      \setlength\leftmargin{\labelwidth}%
      \addtolength\leftmargin{\labelsep}%
      \def\makelabel##1{##1\hss}}%
  \fi}
\makeatother

\makeatletter
\def\enumerate{%
  \@ifnextchar[{\@numerate}{\@numerate[0.]}}
\def\@numerate[#1]{%
  \ifnum \@enumdepth >3 \@toodeep\else
    \advance\@enumdepth\@ne
    \edef\@enumctr{enum\romannumeral\the\@enumdepth}%
    \list{\csname label\@enumctr\endcsname}{%
      \enumargs
      \usecounter{\@enumctr}%
      \settowidth\labelwidth{#1}%
      \setlength\itemindent{\z@}%
      \setlength\leftmargin{\labelwidth}%
      \addtolength\leftmargin{\labelsep}%
      \def\makelabel##1{##1\hss}}%
  \fi}
\makeatother

\newcommand\comment[1]{{\color{red} \fbox{Stuff commented out -- see the latex source.}}}

\theoremstyle{plain}
\newtheorem{assumption}{Assumption}[section]
\newtheorem{defn}{Definition}[section]

\newtheorem{thm}{Theorem}[section]

\newtheorem{lemma}[thm]{Lemma}
\newtheorem{prop}[thm]{Proposition}
\newtheorem{remark}{Remark}

\newcommand{\bn}{\begin{enumerate}}
\newcommand{\en}{\end{enumerate}}
\newcommand{\bct}{\begin{center}}
\newcommand{\ect}{\end{center}}
	
\newcommand{\im}{\item}
\newcommand{\bc}{\begin{cases}}
\newcommand{\ec}{\end{cases}}
\newcommand{\var}{{\rm Var}}

\newcommand{\tr}{{\rm tr}}
\DeclareMathOperator*{\argmin}{argmin}

\newcommand{\bbR}{{\mathbb R}}

\newcommand{\boldc}{{\boldsymbol{c}}}

\newcommand{\wh}{\widehat}
\newcommand{\wt}{\widetilde}
\newcommand{\E}{\mathbb{E}}
\newcommand{\pr}{\mathbb{P}}

\newcommand{\CN}{\mathcal{N}}
\newcommand{\CG}{\mathcal{G}}
\newcommand{\bbL}{\mathbb{L}}
\newcommand{\CL}{\mathcal{L}}

\newcommand{\CR}{\mathcal{R}}

\newcommand{\bbH}{\mathbb{H}}

\def\tilde{\widetilde}
\DeclareMathOperator*{\esssup}{ess\,sup}

\renewcommand{\theenumi}{\roman{enumi}}

\newcommand{\biz}{\begin{itemize}}
\newcommand{\eiz}{\end{itemize}}

\def\hat{\widehat}

\newcommand{\abs}[1]{\left|#1\right|}
\newcommand{\norm}[1]{\left\|#1\right\|}
\newcommand{\ind}[1]{\mathbf{1}\!\left(#1\right)}

\allowdisplaybreaks

\graphicspath{{./figures}}

\begin{document}
\journaltitle{Journal of the Royal Statistical Society Series B: Statistical Methodology}
\DOI{}
\copyrightyear{2026}
\pubyear{2026}
\access{}
\appnotes{Original Article}
\firstpage{1}

\title[Function-On-Function Regression Through SNOs]
{Function-On-Function Regression Through Separable Neural Operators}

\author[1$\ast$]{Tailen Hsing}
\author[2]{Su-Yun Huang}
\author[2]{Toshinari Morimoto}
\authormark{Hsing et al.}

\address[1]{\orgdiv{Department of Statistics},
\orgname{University of Michigan},
\orgaddress{Ann Arbor, \state{Michigan}, \country{USA}}
}

\address[2]{\orgdiv{Institute of Statistical Science},
\orgname{Academia Sinica},
\orgaddress{Taipei, \country{Taiwan}}
}

\corresp[$\ast$]{Corresponding author:
\href{mailto:xxxxx}{thsing@umich.edu}
}

\abstract{
This paper investigates the estimation of the regression operator in function-on-function regression models. While traditional research has predominantly focused on linear models or their immediate nonlinear extensions, we propose a neural operator approach to accommodate general regression operators under mild smoothness assumptions. Operator learning has emerged as an active area of machine learning, particularly for solving physical models governed by partial differential equations. Using this paradigm, our methodology introduces the separable neural operator, 
a neural-operator architecture
that represents the regression operator through input-dependent coefficient functions and output-dependent basis functions.
Beyond adapting this architecture to the regression operator estimation problem, we establish the consistency of the estimator under relatively mild smoothness and sampling conditions, allowing functional data to be observed on dense, possibly irregular, discrete grids. We also apply the proposed approach to the BGC Argo data and demonstrate its potential for oceanographic research.
}

\keywords{
function-on-function regression,
separable neural operators,
deep learning,
regression operator estimation,
statistical consistency,
irregular functional data,
BGC-Argo data
}

\maketitle


\section{Introduction}

Functional data analysis (FDA) has been an active research area in statistics for over two decades. Some background of FDA can be found in a number of books including 
\cite{ramsaysilverman}, \cite{Horvath2012},  \cite{hsing2015theoretical}, \cite{Kokoszka2018}, 
as well as review papers \cite{Wang2016}, \cite{Koner2023}, and \cite{Gertheiss2024}, to name a few.

The topic of functional regression is at the center of FDA research. The most widely studied model is the so-called scalar-on-function linear regression. In this model, the functional predictor $X$ resides in some Hilbert space $\bbH$, typically a function space, while the response $Y$ is a scalar in $\bbR$. Their relationship is given by $\E(Y|X)= \langle X,\beta\rangle$ for some regression slope function $\beta\in\bbH$. Combinations of other response and predictor types, as well as non-linear conditional expectation models, have also been considered, albeit to a lesser extent than the scalar-on-function linear regression. These aforementioned publications, particularly the recent reviews, provide comprehensive overviews of functional regression.

In this paper, we consider the function-on-function regression model:
\begin{align} \label{e:model_0}
Y = G_0(X_1,\ldots,X_P) + \varepsilon,
\end{align}
where each predictor $X_p$ resides in some function space $\bbH_{1,p}$, both the response $Y$ and error process $\varepsilon$ belong to another function space $\bbH_2$, and $G_0$ is a possibly nonlinear operator mapping from $\prod_{p=1}^P \bbH_{1,p}$ to $\bbH_2$. We refer to $G_0$ as the regression operator. Recently, \cite{rao2023modern} provided a comprehensive survey of functional regression along with a detailed compilation of specific forms for $G_{0}$ studied in the literature.

The problem that we focus on is the estimation of $G_0$. Given the complexity of $G_0$, it is natural to leverage recent advances in operator learning for this problem. 
In that regard, \cite{rao2023modern} considered densely observed functional data, with $\bbH_{1,p}=\bbH_2=\bbL^2[0,1]$, and introduced an operator neural-network architecture with $L$ hidden layers for the regression problem. Conceptually, the input and output layers consist of $H_{(p)}^{(0)}(s)=X_p(s), p \in [P]$, and $H^{(L)}(t)= \wh Y(t)$, respectively, and the $k$-th neuron in the $l$-th hidden layer is defined by
\begin{align*}
H_{(k)}^{(l)}(s) = \sigma\left(b_{(k)}^{(l)}(s)+\sum_{j=1}^{J_{l-1}} \int_0^1 w_{(j,k)}^{(l)}(s,t)H_{(j)}^{(l-1)} (t)dt\right),
\end{align*}
for some activation function
$\sigma$, bias function  $b_{(k)}^{(l)} \in \bbL^2[0,1]$,  and weight kernel $w_{(j,k)}^{(l)}\in \bbL^2([0,1]\times [0,1])$. This is a natural adaptation  of the feedforward network for scalar data to functional data, and is 
 sometimes referred to as ``neural operator'' \citep[cf.][]{kovachki2023neural} in operator learning.
To minimize $\sum_{i=1}^n \int_0^1(Y_i(t)-\widehat{Y}_i(t))^2 d t$,
this architecture propagates information through nonlinear transformations of integral operators, and can be implemented via, e.g., function-on-function direct neural networks (FFDNN) or function-on-function basis neural networks (FFBNN) introduced by \cite{rao2023modern}. Their work numerically demonstrated the feasibility of the approach for several models, although statistical guarantees for the resulting regression-operator estimator were not developed.

Several other classes of neural operators have recently been explored in the general field of operator learning. The majority of this literature has focused on approximating solution operators of PDEs. 
Examples include Fourier neural operators~\citep{li2020fourier}, DeepONet~\citep{lu2021learning}, transformer-based neural operators~\citep{li2022transformer,hao2023gnot}, and graph neural operators~\citep{li2020neural}; see, also, \cite{subedi2025operator} for a review of operator learning from a statistical perspective.
The present paper differs from much of that literature in its statistical purpose. Our goal is to estimate a regression operator from observed functional data, rather than to approximate a PDE solution operator based on synthetic data. Our formulation is also different from the function-on-function neural-network architecture of \citet{rao2023modern},  
and is based on the natural idea of separable representation, as described in the following.

For ease of exposition, we assume from this point on that the number of predictors is $P=1$ and that the function spaces are $\bbH_1=\bbH_2=\bbL^2[0,1]$ in \eqref{e:model_0}. Generalization to $P>1$ is straightforward, which we do in the Argo data analysis in Section~\ref{s:argo}.
Our approach is motivated by the basic idea that any $G(x)(\cdot)\in\bbL^2[0,1]$ can be represented as
\begin{align} \label{e:separable}
G(x)(t) = \sum_{k\ge 1} c_k(x) \phi_k(t),\quad t\in [0,1],
\end{align}
where the expansion is understood in the $\bbL^2[0,1]$ sense. Here, $c_k: \bbL^2[0,1]\to\bbR$ and $\phi_k\in\bbL^2[0,1]$. If $\{\phi_k, k\ge 1\}$ is a complete orthonormal system of $\bbL^2[0,1]$, then $c_k(x)=\int_0^1 G(x)(t)\phi_k(t)dt$. However, we consider more general classes of $\{\phi_k, k\ge 1\}$ that can be learned from data using neural networks. 
We describe the representation in \eqref{e:separable} as ``separable'', since the coefficient-like terms $c_k(x)$ depend only on the input function $x$, whereas the basis-like terms $\phi_k(t)$ depend only on the output argument $t$. 
Accordingly, an operator expressed in terms of such a separable representation is (for convenience) referred to as a separable operator. 
For the estimation of $G_0$, our candidate estimators are finite-rank separable operators where both the $c_k$ and $\phi_k$ are modeled by neural networks.

The theoretical foundation for operator learning based on separable representations was established by \cite{chen1995universal} through a seminal universal approximation theory. 
This framework later inspired the DeepONet architecture introduced by \citet{lu2021learning}, which integrates the foundational approximation result developed by \cite{chen1995universal} with modern neural network designs. 
Section \ref{s:SNO} provides a more thorough discussion of these topics.

The remainder of this paper is organized as follows. Section~\ref{s:SNO} introduces the separable neural operator and discusses its connection with DeepONet and related operator-learning architectures. Section~\ref{s:theory} establishes the consistency of our neural-network estimator of the regression operator $G_0$ with respect to predictive risk, under general assumptions on both the regression model and the functional data. To the best of our knowledge, this is the first result 
to establish predictive-risk consistency for regression operator estimation using a neural-network approach. 
Section~\ref{s:implementation} addresses practical implementation, with particular attention to functional data observed on irregular grids. Section~\ref{s:simulation} reports the results of a thorough simulation study demonstrating the effectiveness of our approach and makes comparisons with the neural operator approach of \cite{rao2023modern}. Finally, Section~\ref{s:argo} presents a real-data application to the Argo oceanographic dataset. 
All the proofs and technical details are deferred to the Supplement at the end of the paper
so as not to interrupt the flow of the main text. Labels and equation numbers in the Supplement will be prefixed with $S$ for clarity.

\section{Separable Neural Operators}\label{s:SNO}

The purpose of this section is to formalize the separable neural operator
class motivated by the representation in~\eqref{e:separable}. This class provides the architectural foundation for both the theoretical sieve estimator studied in
Section~\ref{s:theory} and the practical implementation, simulation studies, and real-data application described in later sections.
This leads to the following definition.
\begin{defn}[Separable Neural Operator]\label{d:SNO} A \emph{separable neural operator} (SNO) $G_{\theta,\eta}$ is an operator from $\bbH_1$ to $\bbH_2$, two spaces of functions on $[0,1]$, that has the finite-rank separable representation
\[
G_{\theta,\eta}(x)(t)
= \sum_{k=1}^p c_k(x;\theta)\,\phi_k(t;\eta), \qquad t\in[0,1],
\]
where both $c_k(x;\theta)$ and $\,\phi_k(t;\eta)$ are modeled by neural networks, referred to as coefficient and basis networks, respectively. 
\end{defn}

The representation in Definition~\ref{d:SNO} separates the dependence on the input function $x$ from the dependence on the output argument $t$, with the maps $c_k(x;\theta)$ playing the role of input-dependent coefficient functions, while the functions $\phi_k(t;\eta)$ constituting a learnable output basis. In view of \eqref{e:separable}, 
SNO provides a general and flexible separable framework for approximating operators.
DeepONet~\citep{lu2021learning} is a prominent example of this type of architecture, with its branch network encoding the input function and its trunk network encoding the query location. Another closely related separable architecture is the Separable Operator Networks (SepONet) of~\citet{yu2024separable}, which uses a separable branch-trunk representation for physics-informed learning of parametric PDE operators.
Despite these architectural similarities, the present paper uses the separable structure for a different statistical objective, i.e., regression-operator estimation in function-on-function regression. This approach retains the interpretability and dimension-reduction advantages of classical basis-based FDA methods while accommodating a substantially more flexible class of estimators for the regression operator. Subsequent sections establish theoretical guarantees for this framework for the regression problem and demonstrate its empirical performance in simulation studies and a data analysis.

The approximation-theoretic foundation of SNO can be traced back to \cite{chen1995universal}. Theorem~5 therein establishes a universal approximation theorem for nonlinear operators using finite sums of separable
neural-network components. A version of that result tailored to our setting can be stated as follows. Let $G: V \to C[0,1]$ be a continuous mapping, where $V$ is a compact subset of $C[0,1]$. Then, for any Tauber-Wiener (TW) activation function $\sigma$ in 
the sense of \cite{chen1995universal} (discussed below) and any $\epsilon>0$, there exist positive integers $m,p,q$, constants $c_i^k, \xi_{i j}^k, \theta_i^k, w_k, \zeta_k\in\bbR$ and locations $s_1,\ldots,s_m\in [0,1]$ such that
\begin{align} \label{e:chen_chen}
\sup_{x\in V, t\in [0,1]}\left|G(x)(t)-\sum_{k=1}^p c_k(x)\,\phi_k(t)\right| < \epsilon,
\end{align}
where 
\[
c_k(x)
:= \sum_{i=1}^{q} c_i^k \,\sigma\!\left(\sum_{j=1}^m \xi_{ij}^k x(s_j)+\theta_i^k\right),
\qquad
\phi_k(t)
:= \sigma(w_k t+\zeta_k).
\]
Consequently, the coefficient components $c_k(x)$ and the
basis components $\phi_k(t)$ can be combined to yield finite-rank separable 
representations that uniformly approximate continuous nonlinear operators on compact subsets of $C[0,1]$. 

The TW condition imposed on the activation function $\sigma$ ensures that the corresponding single-hidden-layer neural-network class is sufficiently rich for universal approximation of continuous nonlinear operators on compact subsets of $C[0,1]$. \cite{chen1995universal} also proved the fundamental result that a continuous activation function defining a (regular) tempered distribution \citep[cf.\ Chapter 9 of][]{folland1999} belongs to the TW class if and only if it is not a polynomial. Thus, standard activations such as ReLU, sigmoid and $\tanh$ all qualify as TW functions.

The separable approximation in~\eqref{e:chen_chen} also provides a natural
bridge to the richer SNO parameterizations used in this paper. Instead of using the
specific shallow forms appearing in the approximation theorem, one may
parameterize the coefficient maps $c_k$ and basis functions $\phi_k$ by more flexible neural networks, while retaining the same separable structure. This extension was adopted by
DeepONet~\citep{lu2021learning} through their branch-trunk formulation, where $c_k$ and $\phi_k$ are modeled by deep neural networks. DeepONet and related neural operators have been applied successfully to a wide range of operator-learning problems, including fractional Laplacians and solution operators associated with stochastic differential equations~\citep[cf.][]{li2020fourier,li2020neural,kovachki2023neural}.

A limitation of the approximation result~\eqref{e:chen_chen} is that it does not provide explicit information on the complexity of the operator network, including the admissible choices of $p,q,m$ and the locations $s_j$.
This creates challenges for both the implementation and theoretical analysis. In the context of operator learning, progress on the approximation complexity of DeepONet, including how the network size scales with the approximation error, has been made in~\cite{lanthaler2022error}, ~\cite{weihs2025deep}, and~\cite{weihs2026multiple}.

In practice, the locations $s_j$ can be taken to be the points at which $x$ is observed; these are often referred to as sensor locations in operator learning literature. For regression operator estimation, these locations may or may not be identical for different curves. When the input grids vary across curves, the sampled values $x(s_j)$ may be replaced by suitable features, such as the coefficients from a basis expansion. This, along with other implementation considerations, will be discussed in Section~\ref{s:implementation} and illustrated in Sections~\ref{s:simulation} and~\ref{s:argo}. 

\section{Theoretical Results} \label{s:theory}
This section introduces an SNO sieve estimator of the regression operator and establishes its consistency. 
The regression model assumptions and the sampling scheme for the functional predictor and response are described in Section \ref{ss:assumption}. The SNO sieve network class from which the estimator is obtained is described in Section~\ref{ss:neural_network}. Section~\ref{ss:main} then defines the estimator and its predictive risk, and states the main consistency theorem. All the proofs and technical details are deferred to the Supplement.



Before proceeding, we introduce the following function and operator norms that will be
used throughout the theoretical development.
\begin{itemize}
\item 
For a function $x\in C[0,1]$ and a subset
$S\subset[0,1]$, define 
$\|x\|_{\sup,S} := \sup_{s\in S}|x(s)|$ and $
\|x\|_{\sup} := \|x\|_{\sup,[0,1]}$.

\item
For an operator $G:C[0,1]\to C[0,1]$ and $u\ge 1$, its sup norm over the truncated input domain $\{x\in C[0,1]:\|x\|_{\sup}\le u\}$ is defined as
\begin{align} \label{e:infinity_norm}
\|G\|_{u,\infty}:=\sup_{\|x\|_{\sup}\le u,\; t\in[0,1]} |G(x)(t)|.
\end{align}
\end{itemize}

\subsection{Model assumptions and sampling design} \label{ss:assumption}

Consider the function-on-function regression model:
\begin{align} \label{e:model}
Y(t) = G_0(X)(t) + \varepsilon(t), \ t\in [0,1],
\end{align}
where, for the moment, we impose the basic assumption that $\{X(t), t \in [0,1]\}$ and $\{\varepsilon(t), t \in [0,1]\}$ are independent zero-mean Gaussian processes with sample paths in $\bbL^2[0,1]$ \citep[cf.\ Chapter 7 of][]{hsing2015theoretical}.
The mapping $G_0: \bbL^2[0,1]\to \bbL^2[0,1]$ may be nonlinear and will be referred to as the regression operator. Additional assumptions on different aspects of the model will be introduced below. Our goal is to estimate $G_0$. Although we focus on the simplified single-predictor model~\eqref{e:model} rather than the more general model~\eqref{e:model_0} for convenience, the main ideas developed under~\eqref{e:model} extend naturally to settings with multiple predictors. We will demonstrate this in the Argo data analysis in Section~\ref{s:argo}, where functional predictors, temperature and salinity profiles, are combined with scalar covariates, day of the year and spatial location, to predict the dissolved oxygen level profile.

We begin by describing the second-order behavior of $X$ and $\varepsilon$. Define the covariance kernels for the predictor and error processes:
$$
C_X(s,t) = \E[X(s)X(t)]
\quad\mbox{and}\quad C_\varepsilon(s,t) = \E[\varepsilon(s)\varepsilon(t)], 
\quad s, t\in [0,1].
$$
In addition, the canonical metric associated with $X$ is defined by
\begin{align} \label{e:rho}
d_X(s,t) = \E^{1/2}[(X(s)-X(t))^2],\quad s, t\in [0,1].
\end{align}

\begin{assumption}\label{a:process}
\begin{itemize}
\item [(i)] The canonical metric $d_X (s,t)$ satisfies a H\"older-type condition; i.e., there exist finite constants $\lambda>0$ and $0<\beta\le 1$ such that
\begin{align} \label{e:rho_1}
d_X (s,t) \le \lambda |s-t|^\beta, \qquad s,t \in [0,1].
\end{align}
\item [(ii)] The error covariance kernel $C_\varepsilon(s,t)$ is uniformly bounded on $[0,1]\times [0,1]$.
\end{itemize}
\end{assumption}

\begin{remark} \label{r:process}
The canonical-metric condition \eqref{e:rho_1} holds for a large class of models, including, e.g., fractional Brownian motions and Mat\'ern processes, and plays several roles in our theoretical analysis. Some implications of \eqref{e:rho_1} are as follows.
\biz\im[(a)]
First, \eqref{e:rho_1} implies that $X$ admits a modification with continuous sample paths
almost surely \citep[cf. Theorem 1.3.5 of][]{adler2007random}. Thus, we henceforth take the domain of $G_0$ to be $C[0,1]$, the space of continuous functions. 
\im[(b)]
The condition \eqref{e:rho_1} also provides control over the increments of $X$, which will be used to derive a concentration bound (cf. Lemma~\ref{l:GaussianKriging}) for the kriging residual when reconstructing the full trajectory of $X$ from data observed on a 
finite grid.
\im[(c)] Under our Gaussian framework, \eqref{e:rho_1} implies that $\E\|X\|_{\sup}<\infty$ by standard entropy bounds \citep[e.g., Theorem 1.3.3 of][]{adler2007random}. Thus, the Borell--TIS inequality \citep[Theorem 2.1.1 of][]{adler2007random}
applies to $\|X\|_{\sup}$ and yields, in particular, that $\|X\|_{\sup}$ has finite moments of all orders.
\eiz
\end{remark}

The following assumption introduces the required smoothness properties of $G_0$.

\begin{assumption}\label{a:G_0}
The regression operator $G_0$ is a mapping from $C[0,1]$ to $C[0,1]$ and satisfies the following conditions. 
\begin{itemize}
\item [(i)] 
$G_0(0)=0$. There exist constants $\alpha\in (0,1]$, $\nu>0$, and
$C_0<\infty$ such that, for every $u\ge 1$,
\begin{align} \label{e:smooth_G0}
\|G_0(x_1)-G_0(x_2)\|_{\sup}
\le C_0 u^{\nu} \|x_1-x_2\|_{\sup}^{\alpha},
\qquad 
\|x_1\|_{\sup},\|x_2\|_{\sup}\le u .
\end{align}
Moreover, $G_0$ satisfies the polynomial growth condition
\begin{align} \label{e:growth_G0}
\|G_0(x)\|_{\sup}
\le C_0(1\vee\|x\|_{\sup}^{\nu})
\qquad x\in C[0,1].
\end{align}
Without loss of generality, the same constant $C_0$ and exponent $\nu$ are used in the two bounds, since they can be enlarged to dominate the corresponding constants and polynomial orders.

\item[(ii)] For each $x\in C[0,1]$, $G_0(x)(\cdot)$ is absolutely continuous on $[0,1]$ and its derivative, denoted by $G_0(x)'(\cdot)$, satisfies $\E \left[\esssup_{t\in (0,1)} |G_0(X)'(t)|^2\right] < \infty$. 
\end{itemize}
\end{assumption}  

\begin{remark} \label{r:G_0}
The conditions imposed on $G_0$ in Assumption~\ref{a:G_0} are fairly mild. 
\biz
\im [(a)] 
The condition $G_0(0)=0$ corresponds to a zero-intercept assumption $\E(Y|X=0) = 0$, which is imposed mainly for the convenience of the theoretical analysis. By (c) of Remark \ref{r:process}, the growth condition \eqref{e:growth_G0} on $G_0$ implies that
\begin{align} \label{e:G_0_sup}
\E[\|G_0(X)\|_{\sup}^p] \le C_0^p\, \E\left[\{1\vee\|X\|_{\sup}^{\nu}\}^p\right]< \infty \quad\mbox{for all $p>0$.}
\end{align}

\im[(b)]
The smoothness condition \eqref{e:smooth_G0} on $G_0$ holds for many commonly studied regression operators in the literature, such as integral operators of the form $G_0(x)(t) = \int_0^1 w(s,t)x(s)ds$, 
as well as index-type operator models of the form
\[
G_0(x)(t) = h\left(\langle \beta_1(\cdot,t),x(\cdot)\rangle,\ldots,
\langle \beta_d(\cdot,t),x(\cdot)\rangle\right),
\]
provided that the kernel $w(s,t)$, the link function $h$ and the index functions $\beta_j(s,t)$ satisfy appropriate smoothness conditions.
\eiz
\end{remark}

Next, we describe the sampling design for the functional data. Assume that an iid.\ sample $(X_i,Y_i)$, $i\in [n]$, is available for estimating $G_0$. However, the functional data are not continuously observed. Instead, we observe
\begin{align*}
X_{i,j} := X_i(s_j)\quad\mbox{and}\quad Y_{i,\ell} := Y_i(t_\ell) + e_{i,\ell}, \quad i \in [n], j\in [J], \ell\in [L],
\end{align*}
where $e_{i,\ell}$ are measurement errors, $J$ and $L$ are finite positive integers, and the non-random design points $s_j,t_\ell\in(0,1)$ are
ordered so that $s_j< s_{j+1}$ and $t_\ell < t_{\ell+1}$.
In the asymptotic analysis, the grids are allowed to depend on $n$; i.e., $s_j=s_{n,j}$ and $t_\ell=t_{n,\ell}, J=J_n$ and $L=L_n$, with $J_n\to\infty$ and $L_n\to\infty$ as $n\to\infty$. For convenience, denote the input and
output grids as $S_n:=\{s_{n,j}, j \in [J_n]\}$ and $T_n:=\{t_{n,\ell}, \ell \in [L_n]\}$. We impose the following sampling-design assumption.

\begin{assumption}\label{a:sampling}
\begin{itemize}
\im[(i)] The input grids are nested, i.e., $S_n\subset S_{n+1}$ for all $n$. 
\im[(ii)] 
The input-grid size satisfies $\log J_n = o(n)$. Moreover, the maximal spacings $\rho_n:=\max_{0\le j\le J_n}(s_{n,j+1}-s_{n,j})\to 0$ and $\gamma_n:=\max_{0\le \ell\le L_n} (t_{n,\ell+1}-t_{n,\ell}) \to 0$ as $n\to\infty$, where $s_{n,0}=t_{n,0}:=0$ and $s_{n,J_n+1}=t_{n,L_n+1}:=1$. 
\im [(iii)] The measurement errors $e_{i,\ell}$ are independent Gaussian random variables with mean zero and variances uniformly bounded by some $\sigma_e^2 < \infty$. They are also independent of the functional data $X_i, Y_i$, $i \in [n]$. 
\end{itemize}
\end{assumption}

\begin{remark}\label{r:grid}
\begin{itemize}
\im[(a)]
The assumption that the grid points do not depend on $i$ is imposed to simplify the theoretical analysis. 
In the implementation described in Section~\ref{s:implementation} and in the empirical studies in Sections~\ref{s:simulation} and~\ref{s:argo}, the curves are allowed to be observed on irregular, curve-specific grids. 
\im[(b)]
Condition (ii) of Assumption~\ref{a:sampling} describes a ``dense functional data'' regime, where the number of observation points on each curve tends to infinity and the maximal spacing between adjacent sampling locations tends to zero. The condition $\log J_n = o(n)$ is very mild in functional data analysis, as a common assumption is  $J_n\asymp n^c$ for some finite $c>0$.
\end{itemize}
\end{remark}


\subsection{Neural-network assumptions}
\label{ss:neural_network}

We now introduce the sieve class used in the theoretical analysis, together with the corresponding assumptions on its network parameters. 
Let $\sigma$ be a Lipschitz-continuous TW function.
As explained in Section \ref{s:SNO}, this includes standard activations such as the sigmoid, $\tanh$, and ReLU. For each $n\in\mathbb{N}$ and $u\ge 1$, define a class $\CG_{n,u}$ of operator-valued
maps $G$ from $C[0,1]$ to $C[0,1]$ by
\begin{align} \label{e:Gn}
\CG_{n,u}:=\Bigg\{
& G(x)(t) = \ind{\|x\|_{\sup,S_n} \le u} \sum_{k=1}^{p_n} \sum_{i=1}^{q_n} c_i^k \sigma\left(\sum_{j=1}^{|S_n|} \xi_{i j}^k x(s_{n,j})+\theta_i^k\right)
\sigma\left(w_k  t+\zeta_k\right):  \\
& \|G\|_{u,\infty} \le Bu^{\nu}, 
c_i^k, \xi_{i j}^k, \theta_i^k, \zeta_k, w_k \in [-b_n,b_n], \mbox{ and }  \sum_{j=1}^{|S_n|} \ind{\xi_{i j}^k\not=0} \le r_n \Bigg\},
\nonumber
\end{align}
where $\nu$ is the exponent defined in Assumption~\ref{a:G_0}, and $B$ is a
fixed constant chosen at least as large as the constant $C_0$ appearing
there. The sequences $b_n,p_n,q_n,r_n$ are assumed to satisfy $1 \le b_n,p_n,q_n,r_n \uparrow\infty$ slowly, as specified later in Assumption~\ref{a:network}. Also, let $\CN_{n,u,\delta}$ denote the $\delta$-covering number of $\CG_{n,u}$ in $\|\cdot\|_{u,\infty}$, which plays an important role in the asymptotic theory. Some discussion of $\CN_{n,u,\delta}$ can be found in part (c) of Remark \ref{r:cover_number} below.

Note that each $G\in\mathcal{G}_{n,u}$ admits a separable rank-$p_n$ representation of the form
$G(x)(t)=\sum_{k=1}^{p_n} c_k(x)\,\phi_k(t)$,
where
\[
c_k(x)
= \ind{\|x\|_{\sup,S_n}\le u}\sum_{i=1}^{q_n}
c_i^k\,\sigma\!\left(\sum_{j=1}^{|S_n|}\xi_{ij}^k x(s_{n,j})+\theta_i^k\right), \qquad
\phi_k(t)=\sigma(w_k t+\zeta_k).
\]
Thus, the ``coefficient'' component $c_k(x)$ is a single-hidden-layer network 
in the input $x$, while the ``basis'' component $\phi_k(t)$ is a single affine 
unit in $t$ followed by the activation $\sigma$. Hence, the operator class 
$\CG_{n,u}$ has a shallow separable architecture. It is analogous to 
the rank-$p_n$ decomposition representation used in DeepONet and SNO, but is deliberately kept simpler in order to facilitate and streamline the theoretical analysis.

As reviewed in Section~\ref{s:SNO}, this architecture is motivated by Theorem~5 of \cite{chen1995universal}
which provides approximation guarantees for a continuous operator on a compact set of $C[0,1]$.
The restrictions imposed in $\mathcal{G}_{n,u}$, which lead to a relatively small and sparse 
network class, may appear somewhat stringent. Such constraints, however, are standard in the 
literature on sieve-based theoretical 
analysis; see, e.g., the sieve framework for
deep neural networks in~\cite{schmidt2020nonparametric}.
In contrast, the practical implementation considered later in this paper allows both the coefficient and basis components of $G(x)(t)$ to be modeled by more general deep neural networks; see also Remark~\ref{r:main_theorem}(b) below.

Before stating the assumptions on the network parameters of $\CG_{n,u}$, we record a few observations in the following remark.

\begin{remark} \label{r:cover_number}
\begin{itemize}
\item [(a)] 
Consistent with the sampling design, each $ G(x)\in \CG_{n,u}$ depends on $x$ only through the sampled values $x(s)$, $s \in S_n$.
Moreover, the support of any operator $G\in\CG_{n,u}$ is contained in $\{x\in C[0,1]:\|x\|_{\sup,S_n}\le u\}$, which suggests that $G_0$ will be estimated on that truncated domain. This point will be elaborated on further in Section \ref{ss:main}.

\item[(b)] The growth condition \eqref{e:growth_G0} in~Assumption~\ref{a:G_0}(i) implies that
\begin{align} \label{e:G0u_sup}
\|G_0\|_{u,\infty} \le C_0\sup_{\|x\|_{\sup} \le u}\left(1\vee\|x\|_{\sup}^{\nu} \right) \le C_0 u^{\nu}.
\end{align}
This motivates the bound
$\|G\|_{u,\infty}\le B u^{\nu}$ in~\eqref{e:Gn}.
Without such a bound, the sieve class would be unnecessarily large for the estimation problem, which would complicate the proofs.
Since the true operator $G_0$ is unknown, the constants $C_0$ and $\nu$ are likewise unknown. However, as will be seen in the proofs, the particular choice of the bound $B u^{\nu}$ is not essential. It may be replaced by another deterministic bound
that dominates the growth of $G_0$ on the truncated domain.
This makes the restriction less stringent.
\item[(c)]
In the proofs, we require $(\log \CN_{n,u,\delta})/n$ to tend to zero at a suitable rate for some choice of $\delta=\delta_n$.
The computation of $\CN_{n,u,\delta}$ is given in Lemma \ref{l:covering} of the Supplement. Focusing on the dominant terms of~\eqref{e:metricentropy} in Lemma~\ref{l:covering}, and using the crude bound $\binom{J}{r} \leq \frac{J^r}{r!} \leq \frac{J^r}{(r / e)^r}=\left(\frac{e J}{r}\right)^r$ (by Stirling's formula), we obtain, 
\begin{align}\label{e:metric_entropy}
\log \CN_{n,u,\delta}
\lesssim 
p_nq_nr_n\log\left(\frac{b_n J_n p_nq_nu\, \sigma_{b_n(r_n u +1)}}{\delta}\right), 
\end{align}
where $\sigma_s = \sup_{|t|\le s}|\sigma(t)|$.
Here $a_n\lesssim b_n$ means that $a_n\le Cb_n$ for some constant $C<\infty$ independent of $n$.
\end{itemize}
\end{remark}

Based on~\eqref{e:metric_entropy}, we impose
the following growth conditions on the network parameters in order to control
the complexity of the sieve class $\CG_{n,u}$. 

\begin{assumption}
\label{a:network}
For each fixed $u\ge1$, the network parameters $b_n,p_n,q_n,r_n\uparrow\infty$ and satisfy
\begin{align} \label{e:complexity}
p_nq_nb_n^2\gamma_n \sigma_{b_n(r_n u +1)}\to 0, ~~\mbox{and}~~~
\frac{p_nq_nr_n\log\left(b_n J_n p_nq_n \sigma_{b_n(r_n u +1)}\right)}{n} \to 0, 
\end{align}
as $n\to\infty$, where $\gamma_n$ is as defined in Assumption~\ref{a:sampling}. 
\end{assumption}

\begin{remark} \label{r:complexity}
If $\sigma$ is bounded, as in the case of the sigmoid or $\tanh$ activation, then the factor $\sigma_{b_n(r_n u +1)}$ can be absorbed into the constant in both conditions of~\eqref{e:complexity}. If $\sigma$ is unbounded, such as ReLU, then the Lipschitz-continuity assumption of $\sigma$ implies that
$
\sigma_{b_n(r_n u +1)}\le |\sigma(0)|+C_\sigma b_n(r_n u +1)
$,
where $C_\sigma$ is the Lipschitz constant. Consequently, for fixed $u\ge1$, the factor $\sigma_{b_n(r_nu+1)}$ may be replaced by $b_n r_n$ in~\eqref{e:complexity}.
\end{remark}

\subsection{Main results} \label{ss:main}

We now define the empirical risk minimizer over the sieve class $\CG_{n,u}$. Recall the notation of the input and output grids in  Assumption \ref{a:sampling}. In particular, $\{t_\ell\}_{\ell=1}^L$ denotes the common output grid, which is allowed to be irregular. Note that, for convenience, we suppress the dependence of $t_\ell$ and $L$ on $n$. Define the cell-length quadrature weights:
\begin{equation}\label{e:w_1}
w_1=\frac{t_1+t_2}{2},\quad
w_\ell=\frac{t_{\ell+1}-t_{\ell-1}}{2}~~ ({\rm for}~1<\ell<L),\quad
w_L=1-\frac{t_{L-1}+t_L}{2}.
\end{equation}
Also, recall that $\CG_{n,u}$ contains operators with supports in $\{x:\|x\|_{\sup,S_n}\le u\}$. 
For fixed $u\ge 1$, define
\[V_{i,n,u}:=\ind{\|X_i\|_{\sup,S_n}\le u}.\]
The estimator studied in our theoretical analysis is 
\begin{align} \label{e:G_hat_G_tilde}
\hat G_{n,u} = \argmin_{G\in\CG_{n,u}} \frac{1}{n}\sum_{i=1}^n V_{i,n,u} \sum_{\ell=1}^{L} w_{\ell}\left\{Y_{i,\ell}-G(X_{i})(t_{\ell})\right\}^2.
\end{align}
Thus, if $V_{i,n,u}=0$ or $\|X_i\|_{\sup,S_n}> u$, then the pair $(X_i,Y_i)$ does not contribute to the estimation of $\hat G_{n,u}$. 
Restricting estimation on compact domains is a common practice in classical nonparametric regression theory \citep[see, e.g.,][]{stonerates}. Here, $\hat G_{n,u}$ focuses on the estimation of $G_0$ on the truncated domain $\{x:\|x\|_{\sup,S_n}\le u\}$ for fixed $u\ge 1$. This domain is not relatively compact in $C[0,1]$, but the restriction is crucial in allowing us to have tight control of the properties of the estimator.
The extension from fixed $u$ to a sequence
$u=u_n\uparrow\infty$ is addressed later in Remark~\ref{r:main_theorem} below.
Define, for $u\ge 1$, the truncated operators
\begin{align} \label{e:G0u}
G_{0,u}(x)  = \ind{\|x\|_{\sup} \le u} G_0(x), \qquad
G_{0,n,u}(x)  = \ind{\|x\|_{\sup, S_n} \le u} G_0(x),
\end{align}
where the $G_{0,u}$ can be viewed as the ideal 
estimation target for $\hat G_{n,u}$, while $G_{0,n,u}$ is its grid-based counterpart, using the same discrete-grid truncation as in the sieve class $\CG_{n,u}$. Lemma \ref{l:G0_moment} of 
the Supplement shows that $\limsup_{n\to\infty}\E\left[\left\|G_{0,u}(X)-G_{0,n,u}(X)\right\|_{\bbL^2}^2\right]\to 0$ exponentially fast in $u$.
Thus, our main goal is to study the sieve estimator $\widehat G_{n,u}$ as an estimator of  $G_{0,n,u}$ for any fixed $u\ge 1$. To assess the estimation error, we define the predictive risk:
\begin{align} \label{e:pred_risk}
\CR\left(\wh G_{n,u},G_{0,n,u}\right) := \E\left[\left\|\wh G_{n,u}(X)-G_{0,n,u}(X)\right\|_{\bbL^2}^2\right],
\end{align}
where $\|\cdot\|_{\bbL^2}$ is the standard $L^2$-norm on $\bbL^2[0,1]$, and the expectation is taken with respect to both the training sample and an independent copy $X$ of the predictor.

The main theoretical result of this paper is stated below.

\begin{thm} \label{t:rate}
Assume that Assumptions \ref{a:process}--\ref{a:network} hold.
Then, for every fixed $u\ge 1$,
\begin{align} \label{e:goal}
\lim_{n\to\infty} \CR\left(\wh G_{n,u},G_{0,n,u}\right) = 0.
\end{align}

\end{thm}

Some remarks are in order.
\begin{remark} \label{r:main_theorem}

\biz
\im [(a)]
Note that 
\begin{align*}
\CR\left(\wh G_{n,u},G_{0}\right)&:= \E\left[\left\|\wh G_{n,u}(X)-G_{0}(X)\right\|_{\bbL^2}^2\right] \\
& \le 2\CR\left(\wh G_{n,u},G_{0,n,u}\right) + 2 \E\left[\left\|G_{0,n,u}(X)-G_{0}(X)\right\|_{\bbL^2}^2\right].
\end{align*}
Consequently, by Theorem \ref{t:rate},
\begin{align*} 
\limsup_{n\to\infty} \CR\left(\wh G_{n,u},G_0\right) \le 2 \E\left[\ind{\|X\|_{\sup} > u}\cdot \|G_0(X)\|_{\sup}^2\right],
\end{align*}
where the right-hand side is bounded by $2\Psi^2(u)$ defined in Lemma~\ref{l:G0_moment} of Supplement,
which tends to zero exponentially fast as $u\uparrow\infty$.
Thus, Theorem~\ref{t:rate} has the consequence that one may choose a sequence $u_n\uparrow\infty$ such that
$\lim_{n\to\infty}\CR\left(\wh G_{n,u_n},G_{0}\right) = 0$. 
\im[(b)] Theorem \ref{t:rate} does not provide a rate for the predictive risk for $\hat G_{n,u}$. The reason is essentially that our proof closely aligns with the universal approximation theory of \cite{chen1995universal}, which does not offer information on the necessary network complexity to approximate an operator with a given degree of accuracy (cf.\ Section \ref{s:SNO}). Developing a more sophisticated approximation theory to remedy this is beyond the scope of this paper.
\im[(c)]
The shallow sieve class $\CG_{n,u}$ was chosen to mitigate the theoretical challenges in establishing Theorem~\ref{t:rate}. The practical SNO architecture described in Sections \ref{s:implementation}-\ref{s:argo} can be viewed as an extension of this shallow class, in which the coefficient and basis components are represented by more flexible neural networks. 
Standard universal approximation results suggest that such richer architectures can approximate the shallow class well. As such, Theorem~\ref{t:rate} also provides theoretical support for the broader SNO methodology considered in this paper. However, a rigorous consistency analysis in that regard is left for future work.

\eiz
\end{remark}

\section{Implementation}\label{s:implementation}

In this section, we describe the practical implementation of the SNO for function-on-function regression with irregular input and output grids.

\subsection{Irregular observations and spline representation of inputs}

Assume that both the number of observations and the observation locations may vary across samples for the predictor function $X_i(\cdot)$ and the response function $Y_i(\cdot)$. 
Since $X_i(\cdot)$ is used as the input to the SNO, we refer to it as the input function in what follows.
This sampling scheme differs from the theoretical setting in Section~\ref{s:theory}.
Specifically, for each training sample $i=1,\dots,n$, the input function $X_i(\cdot)$ is observed on an irregular grid $\{s_{i,j}\}_{j=1}^{J_i}$, and the response function $Y_i(\cdot)$ is observed with measurement errors on another irregular grid $\{t_{i,\ell}\}_{\ell=1}^{L_i}$.
Thus, for sample $i$, the available data are given by $\{(s_{i,j}, X_i(s_{i,j}))\}_{j=1}^{J_i}$ and $\{(t_{i,\ell}, Y_{i,\ell})\}_{\ell=1}^{L_i}$, where $Y_{i,\ell} = Y_i(t_{i,\ell}) + e_{i,\ell}$, with $Y_i(t) = G_0(X_i)(t) + \varepsilon_i(t)$.

The SNO is trained using the training data observed as described above.
For the response observations, the SNO can be evaluated at their irregular observation locations, allowing the resulting predictions to be directly compared with the observations during training.
In contrast, it is not straightforward to handle the raw input observations in the coefficient mapping, since the mapping depends on the entire trajectory of the predictor function.
To that end, we project each input function onto a fixed cubic B-spline basis on $[0,1]$, which provides a common fixed-dimensional coefficient representation of the irregularly observed input functions across samples. 
Let $\{B_k\}_{k=1}^{n_B}$ denote the chosen B-spline basis. 
Specifically, we approximate
$X_i(s)\approx {\widetilde X}_i(s) = \sum_{k=1}^{n_B} c^{(x)}_{i,k}\,B_k(s)$,
where the spline coefficients are estimated by penalized least squares; see, e.g., \cite{ramsaysilverman}. 
The resulting coefficient vector
$x^{\text{feat}}_i = (c^{(x)}_{i,1},\dots,c^{(x)}_{i,n_B})$
is used as the feature representation of $X_i$ and serves as the input to the coefficient networks of the SNO. 
 
The resulting spline feature vectors may optionally be standardized coordinate-wise using the training-set mean and standard deviation, and the same transformation is then applied to the validation and test sets. 
This spline-based encoding yields a stable and compact representation of input curves while naturally accommodating subject-specific irregular grids.

\subsection{Output representation of the SNO}

Let $x^{\mathrm{feat}}$ denote the B-spline coefficient vector of the
input function $X(\cdot)$. For $t\in[0,1]$, the SNO represents the
predicted response as
$\widehat Y(t)=\sum_{j=1}^{p}
c_j(x^{\mathrm{feat}};\theta)\phi_j(t;\eta)$,
where $\{c_j(x^{\mathrm{feat}};\theta)\}_{j=1}^p$ and
$\{\phi_j(t;\eta)\}_{j=1}^p$ are the outputs of CoefficientNet and
BasisNet, respectively. The parameters $\theta$ and $\eta$ are trained
jointly. We describe the two network structures below.

\parhead{CoefficientNet}
This network maps the B-spline coefficient vector
$x^{\mathrm{feat}}\in\mathbb{R}^{n_B}$ to a $p$-dimensional coefficient vector:
\vspace{-0.75em}
\[
(c_1(x^{\mathrm{feat}};\theta),\ldots,c_p(x^{\mathrm{feat}};\theta))^\top
=
\mathrm{CoefNet}(x^{\mathrm{feat}};\theta)\in\mathbb{R}^p.
\]
In general, the CoefficientNet may be implemented as a fully connected neural network with nonlinear hidden layers and a linear output layer.
In particular, we use three ReLU hidden layers:
$h_\ell=\mathrm{ReLU}(W_\ell h_{\ell-1}+b_\ell)$,
$\ell=1,2,3$, with $h_0=x^{\mathrm{feat}}$, followed by
$\mathrm{CoefNet}(x^{\mathrm{feat}};\theta)=W_4h_3+b_4$,
where $\theta=\{(W_\ell,b_\ell)\}_{\ell=1}^4$.
The specific network configurations used in the simulation study and the Argo data analysis are given in Sections~\ref{ss:architecture} and~\ref{ss:argo_model_and_fitting}, respectively.

\parhead{BasisNet}
The BasisNet maps an evaluation point $t \in [0,1]$ to a $p$-dimensional vector:
\vspace{-0.75em}
\[
(\phi_1(t;\eta),\ldots,\phi_p(t;\eta))^\top
=
\mathrm{BasisNet}(t;\eta)
\in \mathbb{R}^p.
\]
In our implementation, Fourier features are used as an optional positional encoding for the evaluation point $t$. When this option is adopted, the evaluation point $t$ is first transformed into
\vspace{-0.75em}
\[
\gamma(t)
=
\left(
t,
\sin(2\pi t),\ldots,\sin(2\pi n_{\mathrm{freq}} t),
\cos(2\pi t),\ldots,\cos(2\pi n_{\mathrm{freq}} t)
\right)^\top,
\]
where $n_{\mathrm{freq}}$ is the number of frequencies. The map $\gamma(t)$ augments the scalar location $t$ with sine and cosine features at multiple frequencies, thereby providing the network with a richer representation of the position at which the output function is evaluated. This is  analogous in spirit to the positional encoding used in Transformer models, where location information is embedded into a higher-dimensional feature representation; see~\cite{vaswani2017attention}. Such Fourier-feature representations can help neural networks learn functions with high-frequency variation in low-dimensional inputs; see~\cite{tancik2020fourier}. When $n_{\mathrm{freq}}=0$, no Fourier features are used and the transformation reduces to the identity encoding $\gamma(t)=t$.

As with the CoefficientNet, the BasisNet may be implemented as a fully connected neural network with nonlinear hidden layers and a linear output layer.
In particular, we use two fully connected hidden layers with componentwise hyperbolic tangent activations:
$r_\ell=\tanh(V_\ell r_{\ell-1}+d_\ell)$, $\ell=1,2$, with
$r_0=\gamma(t)$, followed by
$(\phi_1(t;\eta),\ldots,\phi_p(t;\eta))^\top=V_3r_2+d_3$,
where $\eta=\{(V_\ell,d_\ell)\}_{\ell=1}^3$.
The specific network configurations used in the simulation study and the Argo data analysis are given in Sections~\ref{ss:architecture} and~\ref{ss:argo_model_and_fitting}, respectively.

\subsection{Training loss}

The SNO is trained by minimizing the empirical integrated squared loss described below. 
Since the SNO can be evaluated at arbitrary output locations, predictions can be computed directly at the observed output locations $\{t_{i,\ell}\}_{\ell=1}^{L_i}$.
For the $i$-th training function and an evaluation point $t_{i,\ell}$, the network output is
\[
\wh Y_{i,\ell} := G_{\theta,\eta}(X_i)(t_{i,\ell})
= \sum_{j=1}^{p} c_j(x^{\mathrm{feat}}_i;\theta)\,\phi_j(t_{i,\ell};\eta). 
\]
Thus, our training loss is chosen as
\begin{equation}\label{e:trainLoss}
\CL_n(\theta,\eta)
= \frac{1}{n} \sum_{i=1}^n \sum_{\ell=1}^{L_i} w_{i,\ell}\,
  \left( Y_{i,\ell} - \widehat Y_{i,\ell} \right)^2,
\end{equation}
where the quadrature weights $w_{i,\ell}$ are defined analogously to \eqref{e:w_1}, with $L$, $t_\ell$, and $w_\ell$ replaced by $L_i$, $t_{i,\ell}$, and $w_{i,\ell}$, respectively, for the $i$-th response grid. 
The parameters $(\theta,\eta)$ are estimated by minimizing $\CL_n(\theta,\eta)$ via stochastic gradient-based optimization.


\section{Simulation Studies} \label{s:simulation}

We conducted simulation studies to evaluate the prediction performance of the proposed SNO. Specifically, we considered eight function-on-function regression models and compared the SNO with the FFDNN benchmark of \cite{rao2023modern} in terms of response-function prediction accuracy.

\subsection{Data-generating models}

\parhead{Generation of predictor functions}
Each $X_i(\cdot)$ is sampled from a zero-mean Gaussian process with Mat\'ern covariance kernel
\[
C(s,s')
= \sigma^2 \, \frac{2^{1-\nu}}{\Gamma(\nu)}
\left(\frac{\sqrt{2\nu}\, \|s-s'\|}{\vartheta}\right)^{\nu}
K_{\nu}\!\left(\frac{\sqrt{2\nu}\, \|s-s'\|}{\vartheta}\right),
\]
where $\Gamma$ is the gamma function, $K_\nu$ is the modified Bessel function of the second kind, and $\vartheta$ and $\nu$ are positive covariance parameters. 
Here we set the smoothness parameter $\nu=5/2$, the range parameter $\vartheta=0.5$, and the variance $\sigma^2=1$. 

For each replication, the predictor functions $X_i(\cdot)$, $i=1,\ldots,n$, are partially observed on subject-specific grids. Specifically, for each $i$, the number of observation points $m_{x,i}$ is first drawn uniformly from $\{80,81,\ldots,120\}$. The observation locations are then sampled independently from $\mathrm{Unif}(0,1)$, yielding an irregular grid. This differs from \cite{rao2023modern}, where predictor functions are observed on an equally spaced grid. Thus, although the first six data-generating models are adapted from Rao and Reimherr, our simulation settings are not identical because we use subject-specific irregular grids. This design better reflects irregular sampling in practice and is naturally accommodated by the proposed SNO.

\vskip.2cm 
\parhead{Generation of response functions}
The response functions $Y_i(\cdot)$ are generated from the nonlinear function-on-function regression model
$Y_i(t) = G_0(X_i)(t)+\varepsilon_i(t)$ for a list of $G_0$ 
described below. For each subject $i$, the response function is observed on a subject-specific grid
$\{t_{i,\ell}\}_{\ell=1}^{m_{y,i}}\subset[0,1]$, where the number of observation points $m_{y,i}$ is drawn uniformly from $\{60,61,\ldots,90\}$. The observation locations are sampled independently from $\mathrm{Unif}(0,1)$, yielding an irregular response grid.
At each observed location $t_{i,\ell}$, we add independent Gaussian measurement error,
\begin{equation}\label{e:measurement_error}
Y_{i,\ell} = Y_i(t_{i,\ell}) +e_{i,\ell},
\qquad e_{i,\ell}\overset{\mathrm{iid}}{\sim}N(0,1).
\end{equation}
Following \cite{rao2023modern}, we set $\varepsilon_i(t)\equiv 0$ in the reported simulations, so that the response noise arises only from the measurement errors in~\eqref{e:measurement_error}.
Additional experiments with a nonzero error process of moderate magnitude led to qualitatively similar conclusions.


\vskip.2cm
\parhead{Signal-to-noise calibration}
For an unscaled operator $\widetilde G_0$, define the integrated signal variance
\[
\mathcal S(\widetilde G_0)
=
\E\!\left[
\int_0^1
\bigl\{\widetilde G_0(X)(t)
-\E[\widetilde G_0(X)(t)]\bigr\}^2\,dt
\right].
\]
Since the measurement errors are independent $N(0,1)$ variables, we set
$G_0=a\widetilde G_0$, with
$a=\{2/\mathcal S(\widetilde G_0)\}^{1/2}$,
so that the resulting SNR, defined as the integrated signal variance divided by the measurement-error variance, is approximately $2$. The resulting values of $a$ are
reported with the model definitions below.

We consider eight function-on-function regression settings, where 
the first six are adapted from \cite{rao2023modern}. 

\begin{enumerate}
\item {Linear model:}
$\widetilde G_0(X_i)(t)
=
\left\{\int_0^1 5\sin(2\pi s)X_i(s)\,ds\right\}
3\sin(3\pi t)$, with $a=0.47$.

\item {Convolution additive model (CAM):}
$\widetilde G_0(X_i)(t)
=
t\int_0^1 X_i(s)^2s\,ds$, with $a=4.26$.

\item {Single-index model:}
$\widetilde G_0(X_i)(t)
=
\left\{
3\sin(3\pi t)
\int_0^1 5\sin(2\pi s)X_i(s)\,ds
\right\}^{2}$, with $a=0.090$.

\item {Multiple-index model:}
Let $\widetilde G_0(X_i)(t)=I_1(X_i,t)^2I_2(X_i,t)^2$, where $I_1(X_i,t)=3\sin(3\pi t)\int_0^1 5\sin(2\pi s)X_i(s)\,ds$
and $I_2(X_i,t)=2\sin(3\pi t)\int_0^1 4\sin(5\pi s)X_i(s)\,ds$, with $a=0.083$.

\item {Quadratic model:} With $\beta_a(q)=5\sin(3\pi q)$,
$\beta_b(s)=5\sin(\pi s)$, and
$\beta_c(t)=5\sin(\pi t)$, define
\[
\begin{aligned}
\widetilde G_0(X_i)(t)
={}&
15\sin(3\pi t)\int_0^1\sin(2\pi s)X_i(s)\,ds\\
&+
\beta_c(t)\int_0^1\!\!\int_0^1
\beta_a(q)\beta_b(s)X_i(q)X_i(s)\,dq\,ds,
\end{aligned}
\]
and set $G_0=a\widetilde G_0$ with $a=0.122$.

\item {Complex quadratic model:}
Using the functions $\beta_a,\beta_b,\beta_c$ defined in the quadratic model, let
\[
\widetilde G_0(X_i)(t)
=
t\int_0^1 X_i(s)^2s\,ds
+
\left\{
\beta_c(t)\int_0^1\!\!\int_0^1
\beta_a(q)\beta_b(s)X_i(q)X_i(s)\,dq\,ds
\right\}^{2},
\]
and set $G_0=a\widetilde G_0$ with $a=0.00212$.

\item {Nonlinear dynamical system with analytic solution:}
For $\lambda=0.1$, let
\[
\widetilde G_0(x)(t)
=
\tanh\!\left\{
\int_0^t e^{-\lambda(t-s)}x(s)\,ds
\right\},
\qquad t\in[0,1].
\]
Equivalently, $G_0(x)=a\widetilde G_0(x)$ satisfies
\[
\frac{d}{dt}G_0(x)(t)
=
\left[
1-\left\{\frac{G_0(x)(t)}{a}\right\}^{2}
\right]
\left[
-\lambda a\,
\operatorname{arctanh}\!\left\{\frac{G_0(x)(t)}{a}\right\}
+a\,x(t)
\right],
\]
where $G_0(x)(0)=0$ and $a=3.63$.

\item {Nonlinear dynamical system without analytic solution:}
Let $\widetilde G_0(x)(t)=\widetilde y_0(t)$, where
$\widetilde y_0'(t)
=
-\widetilde y_0(t)
-\frac12\widetilde y_0(t)^3
+x(t)+0.3\sin(2\pi t)$ and $\widetilde y_0(0)=0$.
The trajectory $\widetilde y_0$ is generated numerically using an ODE solver, and we set $G_0=a\widetilde G_0$ with $a=4.00$.
\end{enumerate}

\vskip.2cm
\parhead{Sample sizes and replications}
For each of the models (i)-(viii), we generate $50$ independent replications using different random seeds. 
In each replication, $500$ training samples, $200$ validation samples, and $500$ test samples are generated.
The validation set is used to select the training epoch. 

\subsection{Model architectures}\label{ss:architecture}

We consider SNO as the proposed method and FFDNN of~\cite{rao2023modern} as the main benchmark method.
Below, we describe the specific architectures of these two models used in the simulation studies. For an additional comparison, we also consider FFBNN of~\cite{rao2023modern}; its architecture and simulation results are reported in the Supplement.

\parhead{SNO architecture}
The input function $X_i(\cdot)$ is represented by $n_B=30$ B-spline coefficients.
The CoefficientNet has three hidden layers with widths $128$, $64$, and $32$, respectively, followed by a linear output layer with output dimension $p=32$.
For the BasisNet, we use Fourier features with $n_{\mathrm{freq}}=1$, i.e., $\gamma(t) = (t,\sin(2\pi t),\cos(2\pi t))^\top$.
The BasisNet has two hidden layers with widths $64$ and $32$, respectively, followed by a linear output layer with output dimension $p=32$.

\parhead{FFDNN architecture}
The FFDNN benchmark maps a regular-grid representation of the predictor
curve to a regular-grid representation of the response curve. Because
the simulated data are observed on irregular grids, both curves are first
linearly interpolated onto regular grids, yielding
$x\in\mathbb R^{m_x}$ and $y\in\mathbb R^{m_y}$ 
with
$m_x=100$ and $m_y=75$.
The network has two hidden layers, each consisting of four parallel
channels of width $S=30$. For $k=1,\ldots,4$,
$h_1^{(k)}=\mathrm{ReLU}\!\left(m_x^{-1}W_1^{(k)}x+b_1^{(k)}\right)$,
and
$h_2^{(k)} = \mathrm{ReLU}\!\left(
S^{-1}\sum_{j=1}^4 W_2^{(j,k)}h_1^{(j)}+b_2^{(k)}
\right)$.
The output is $\widehat y
=b_3+S^{-1}\sum_{k=1}^4W_3^{(k)}h_2^{(k)}$.
Here
$W_1^{(k)}\in\mathbb R^{S\times m_x}$,
$W_2^{(j,k)}\in\mathbb R^{S\times S}$,
$W_3^{(k)}\in\mathbb R^{m_y\times S}$,
$b_1^{(k)},b_2^{(k)}\in\mathbb R^S$, and
$b_3\in\mathbb R^{m_y}$.
The model is trained by minimizing the mean squared error between
$\widehat y$ and the interpolated response vector $y$.

\parhead{Training settings}
To save space, detailed training configurations, including the maximum number of epochs and learning rate, are reported in Table~\ref{tab:simulation_training_settings} in the Supplement.

\subsection{Evaluation metric}


The training loss 
is used for model fitting, and the validation loss is used for selecting the training epoch.
For test evaluation, however, we evaluate prediction accuracy against the noiseless true response function $G_0(X_i)(\cdot)$ rather than the noisy observed response $Y_i(\cdot)$.
Specifically, for each simulation replicate, we use the fitted $\wh G$ and the test sample, $X_{\mathrm{test},i}, i=1,\ldots,n_{\rm test}=500$, to compute
\[
\left\{
\frac{1}{n_{\rm test}}
\sum_{i=1}^{n_{\rm test}}
\int_0^1 \bigl(\widehat{G}(X_{\mathrm{test},i})(t)-G_0(X_{\mathrm{test},i})(t)\bigr)^2\,dt
\right\}^{1/2},
\]
where the integral over $t$ is computed numerically.
This quantity is an estimate of the root mean integrated squared error:
\[
\mathrm{iRMSE} =
\left\{
\E_X
\int_0^1 \bigl(\widehat{G}(X)(t)-G_0(X)(t)\bigr)^2\,dt
\right\}^{1/2},
\]
where the expectation is taken with respect to an independent copy of the predictor $X$,  conditional on the fitted operator $\wh G$. 

\subsection{Results}

Table~\ref{tab:sim_irmse_comparison} compares the iRMSE of FFDNN and SNO across the eight data-generating models.
The table reports the mean and standard deviation of the iRMSE values over the 50 simulation runs.
Figures~\ref{fig:simulation_histogram_vs_ffdnn} and \ref{fig:simulation_boxplot_vs_ffdnn} display, respectively, the histograms and boxplots of the run-wise iRMSE values for the eight data-generating models.
As can be seen from these figures as well as Table~\ref{tab:sim_irmse_comparison}, except for the linear model, for which FFDNN is particularly well suited, SNO consistently yields smaller and less variable iRMSE values than FFDNN.
This suggests that SNO provides more accurate predictions than the FFDNN baseline, particularly in nonlinear settings.

Table~\ref{tab:sim_computational_cost} reports the computational costs of FFDNN and SNO for each data-generating model in our computing environment, with all experiments run using a GPU.
The reported values are averages over the 50 simulation runs.
Overall, FFDNN tends to require more training epochs than SNO and, in several settings, approaches the maximum of 500 epochs without triggering early stopping.
Consequently, FFDNN generally requires a longer total training time.
In contrast, early stopping is triggered after substantially fewer epochs for SNO in most settings, and SNO also requires less time per epoch than FFDNN in our experiments.
The computation times should be interpreted as approximate; see the caption of Table~\ref{tab:sim_computational_cost} for details.


\begin{table}[h!]
\centering
\caption{Comparison of prediction performance for FFDNN and SNO across eight data-generating models. Mean and Std denote the mean and standard deviation of the iRMSE over 50 simulation runs, respectively.}
\label{tab:sim_irmse_comparison}

{\footnotesize
\begin{tabular}{lcccc}
\toprule
\multirow{2}{*}{Data-generating model}
& \multicolumn{2}{c}{FFDNN}
& \multicolumn{2}{c}{SNO} \\
\cmidrule(lr){2-3} \cmidrule(lr){4-5}
& Mean & Std & Mean & Std \\
\midrule
Linear            & 0.0763 & 0.0123 & 0.0899 & 0.0090 \\
CAM               & 0.9612 & 0.3914 & 0.1393 & 0.0281 \\
Single-index      & 1.4218 & 0.1288 & 0.1574 & 0.0351 \\
Multiple-index    & 1.3694 & 0.3611 & 0.6164 & 0.2675 \\
Quadratic         & 1.0409 & 0.3429 & 0.1769 & 0.0380 \\
Complex quadratic & 1.1499 & 0.3652 & 0.4782 & 0.2857 \\
Dynamical 1       & 0.3115 & 0.0091 & 0.1156 & 0.0074 \\
Dynamical 2       & 0.3610 & 0.0140 & 0.0984 & 0.0052 \\
\bottomrule
\end{tabular}
}
\end{table}

\begin{table}[h!]
\centering
\caption{Comparison of computational costs for FFDNN and SNO across eight data-generating models.
The reported values are averages over 50 simulation runs.
Computation times are approximate and may vary with the computational environment and system load.
For example, although the architecture and computational procedure of FFDNN are unchanged across the data-generating models, its average time per epoch varies across models.
The timing results are therefore intended primarily as a rough comparison under the present experimental setting.}
\label{tab:sim_computational_cost}
{\footnotesize
\begin{tabular}{lcccccc}
\toprule
\multirow{2}{*}{Data-generating model}
& \multicolumn{2}{c}{Epochs}
& \multicolumn{2}{c}{Total time}
& \multicolumn{2}{c}{Time/epoch} \\
\cmidrule(lr){2-3} \cmidrule(lr){4-5} \cmidrule(lr){6-7}
& FFDNN & SNO & FFDNN & SNO & FFDNN & SNO \\
\midrule
Linear            & 495.0 & 110.3 & 60.9 & 10.3 & 0.123 & 0.094 \\
CAM               & 417.5 & 128.0 & 51.7 & 11.9 & 0.124 & 0.093 \\
Single-index      & 152.4 & 150.2 & 19.1 & 13.8 & 0.126 & 0.092 \\
Multiple-index    & 233.1 & 188.1 & 30.8 & 17.2 & 0.136 & 0.092 \\
Quadratic         & 400.9 & 170.1 & 67.9 & 15.7 & 0.170 & 0.092 \\
Complex quadratic & 350.1 & 140.5 & 59.6 & 13.0 & 0.171 & 0.093 \\
Dynamical 1       & 441.1 & 112.4 & 61.3 & 10.4 & 0.139 & 0.093 \\
Dynamical 2       & 497.0 & 97.8  & 61.2 & 9.2  & 0.123 & 0.094 \\
\bottomrule
\end{tabular}
}
\end{table}

\begin{figure}[!t]
\centering
\includegraphics[width=0.95\textwidth]{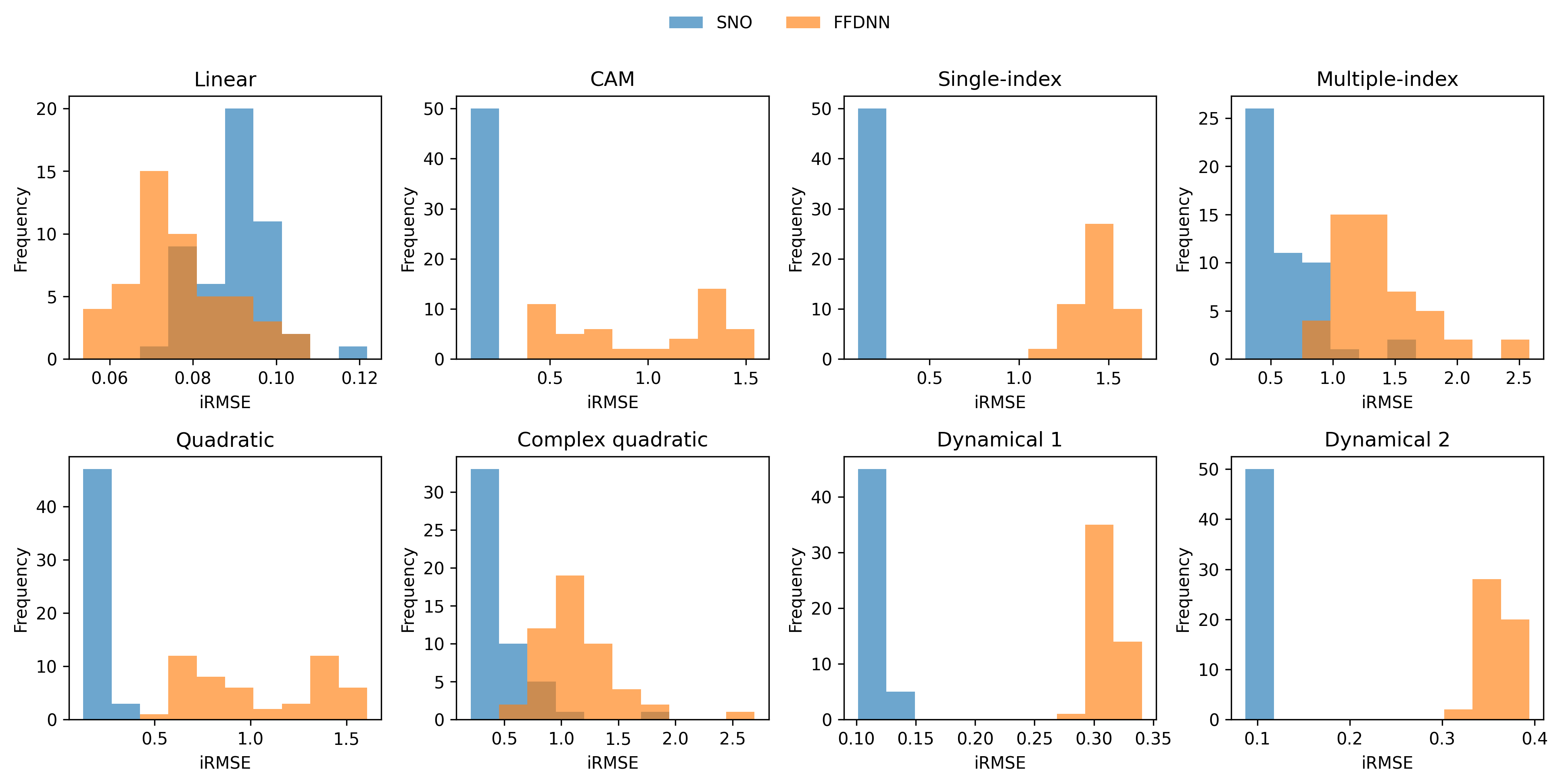}
\caption{
Histograms comparing the run-wise iRMSE values of SNO and FFDNN over the 50 replicate runs.
}
\label{fig:simulation_histogram_vs_ffdnn}
\end{figure}

\begin{figure}[!t]
\centering
\includegraphics[width=0.95\textwidth]{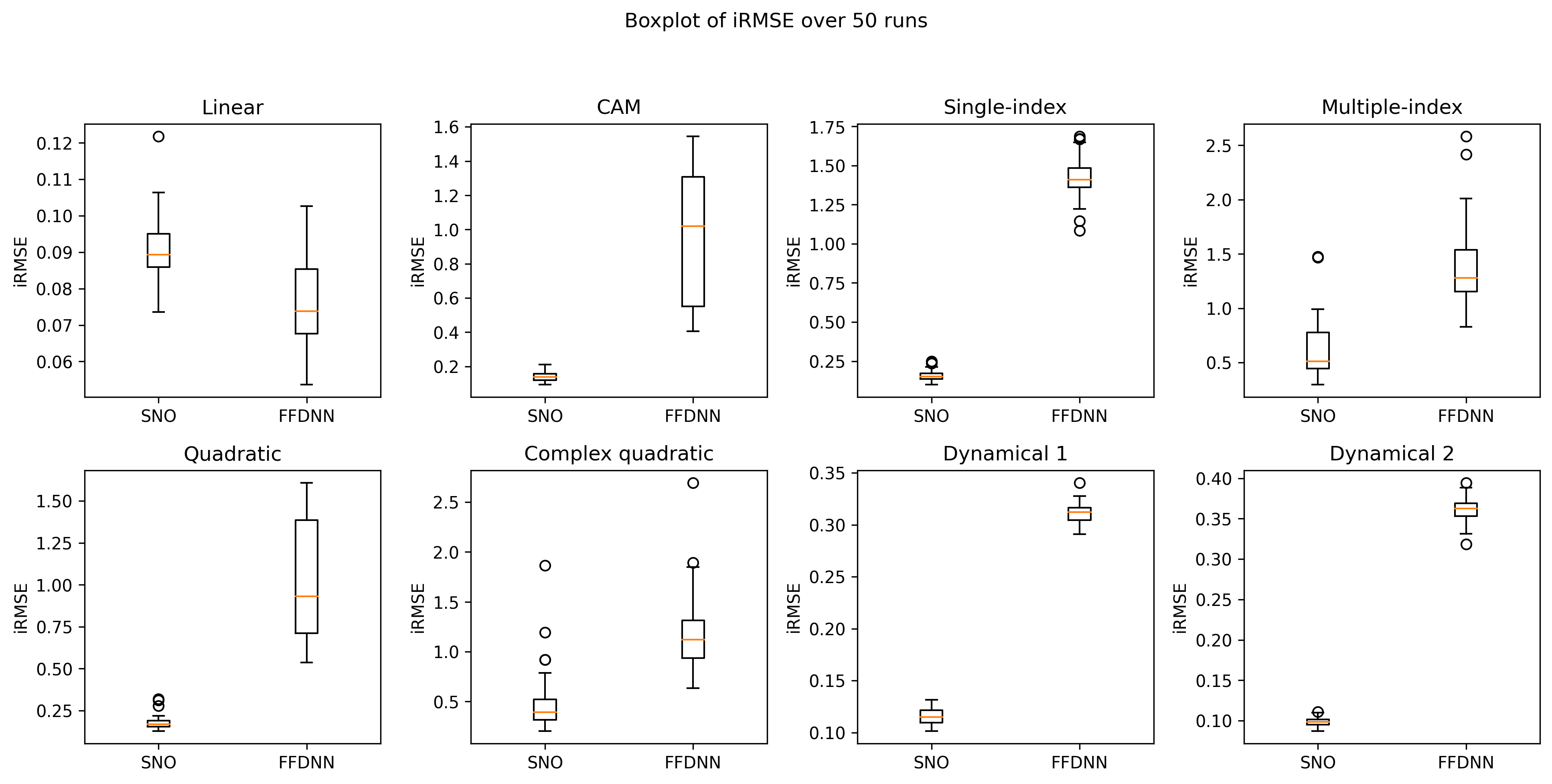}
\caption{
Boxplots comparing the run-wise iRMSE values of SNO and FFDNN over the 50 replicate runs.
}
\label{fig:simulation_boxplot_vs_ffdnn}
\end{figure}


\section{Analysis of the Argo data} \label{s:argo}



In this section, we present a real-data application of SNO using BGC-Argo float data. 
Argo is an international program \citep[cf.][]{argo2020} of autonomous robotic floats that drift at depths of the ocean to monitor ocean conditions. 
Each float periodically ascends from 2000 meters, recording measurements as a function of depth or pressure, producing what is called a \emph{profile}. 
Core Argo floats collect temperature and salinity profiles, whereas the more advanced BioGeoChemical Argo (BGC-Argo) floats additionally measure biogeochemical variables, including dissolved oxygen, pH, and nitrate concentration levels. 
Since BGC-Argo floats are considerably more expensive, they remain substantially outnumbered by Core floats.
This disparity motivates our application.
We use the SNO regression framework to learn the relationship between the commonly available variables, temperature and salinity, and the BGC-specific variable dissolved oxygen. Such a model provides a framework for potential imputation of dissolved oxygen profiles for the much larger Core Argo fleet. Specifically, once trained on BGC-Argo data, the model could be used to predict a dissolved oxygen profile from the temperature and salinity profiles of a Core float together with auxiliary spatio-temporal information (location, time). 
In the present analysis, however, both model fitting and out-of-sample evaluation are conducted using BGC-Argo profiles.
We refer the reader to \cite{giglio2018estimating} and \cite{korte2025functional} and the references therein for further background on imputation of BGC variables for Core Argo profiles.

\subsection{Data description and preprocessing}

We analyze BGC-Argo profiles observed during 2016--2024 from the SEANOE \texttt{BGC Sprof} NetCDF database using the snapshot dated April 8, 2026.
Each profile contains temperature, salinity, and dissolved oxygen curves observed at irregular pressure levels. Within a profile, the three variables share the same pressure grid, whereas the grids vary across profiles. 
Each profile is also accompanied by metadata, including observation date and sampling location.

\parhead{Preprocessing and data filtering}

We focus on observations in the pressure range
$[100,300]$ dbar and remove any pressure level at which temperature,
salinity, or dissolved oxygen is non-finite. Because the three variables
share the same within-profile pressure grid, observations at a removed
pressure level are deleted from all three curves, while the remaining
profile is retained.
We further exclude profiles that fail any of the following criteria:
(i) at least one valid observation below $100$ dbar and one above $300$ dbar, so that boundary information is retained when the profiles are restricted to $[100,300]$ dbar;
(ii) all dissolved oxygen values over the extended interval, from the nearest valid observation below $100$ dbar to the nearest valid observation above $300$ dbar, lie within $[0,600]$;
(iii) the difference between the maximum and minimum dissolved oxygen values over the extended interval lies within $[5,300]$; 
and (iv) at least $30$ valid observations are in the range
$[100,300]$ dbar. These criteria ensure adequate boundary coverage,
remove physically implausible or nearly flat profiles, and provide
sufficient within-profile resolution for model fitting.
Because date and location are used as covariates, profiles with missing
or invalid metadata are also removed. The Argo Julian day, latitude,
and longitude are used to construct year, month, day-of-year, and
geographical covariates.

\parhead{Training and test periods}

For each prediction target year from 2021 to 2024, profiles from the preceding five years are randomly divided into training and validation sets, with the latter used for early stopping, while profiles from the target year are used only for out-of-sample prediction and evaluation. 
For example, the predictions for 2021 use data from 2016--2020 for training and validation.

\subsection{Prediction model and fitting}\label{ss:argo_model_and_fitting}

Recall that our goal is to predict the dissolved oxygen curve $X^{\rm doxy}_i(s)$ from the temperature and salinity curves $X^{\rm temp}_i(s)$ and $X^{\rm psal}_i(s)$, together with auxiliary metadata, namely the geographical sampling location and observation date.
Thus, $X^{\rm temp}_i$ and $X^{\rm psal}_i$ are used as functional predictors, whereas $X^{\rm doxy}_i$ is used as the functional response.
In this subsection, we describe how these variables are encoded and combined to form the input to the SNO model, following the implementation details described in Section~\ref{s:implementation}.

\parhead{Encoding temperature and salinity curves by B-spline coefficients}
First, the functional predictors are represented by B-spline coefficient vectors.
The pressure values in $[100,300]$ dbar are first rescaled to $[0,1]$.
For each profile, we then separately fit B-spline expansions to the irregularly observed temperature and salinity curves on the rescaled pressure domain.
The resulting B-spline coefficient vectors for $X^{\rm temp}_i(\cdot)$ and $X^{\rm psal}_i(\cdot)$ are denoted by
$c^{\rm temp}_i \in \mathbb{R}^{K}$ and
$c^{\rm psal}_i \in \mathbb{R}^{K}$, respectively.
In our implementation, we use $K=32$ basis functions for each variable.
The B-spline coefficients are standardized using the mean and standard deviation computed from the training set.

\parhead{Encoding spatial and seasonal information}
In addition to the two B-spline coefficient vectors, we incorporate the observation date and geographical sampling location as auxiliary predictor covariates.
The sampling location is represented by latitude and longitude.
To avoid the discontinuity in longitude, we map latitude and longitude to three-dimensional Cartesian coordinates on the unit sphere as
$z_{i,1}=\cos(\mathrm{lat}_i)\cos(\mathrm{lon}_i)$,
$z_{i,2}=\cos(\mathrm{lat}_i)\sin(\mathrm{lon}_i)$, and
$z_{i,3}=\sin(\mathrm{lat}_i)$,
where latitude and longitude are measured in radians.
The observation date is encoded by the day of year $d_i$ as
$z_{i,4}=\sin(2\pi d_i/365.25)$ and $z_{i,5}=\cos(2\pi d_i/365.25)$.
Thus, the metadata vector is defined as
$z_i=(z_{i,1},z_{i,2},z_{i,3},z_{i,4},z_{i,5})^\top\in\mathbb{R}^{5}$.
These auxiliary covariates are bounded and are therefore used without standardization.

\parhead{Prediction model}
Combining the two B-spline coefficient vectors with the auxiliary metadata vector, we define $x_i=\bigl((c^{\rm temp}_i)^\top,(c^{\rm psal}_i)^\top,z_i^\top\bigr)^\top\in\mathbb{R}^{2K+5}$ with $K=32$.
The CoefficientNet takes $x_i$ as input and has three hidden layers with widths $256$, $128$, and $64$, respectively, followed by a linear output layer with output dimension $p=32$.
For the BasisNet, we use Fourier features with $n_{\rm freq}=1$, i.e.,
$\gamma(s)=(s,\sin(2\pi s),\cos(2\pi s))^\top$, where $s\in[0,1]$ denotes the rescaled pressure.
The BasisNet has two hidden layers with widths $64$ and $32$, respectively, followed by a linear output layer with output dimension $p=32$. 

\parhead{Model fitting}
The model is trained by minimizing the weighted integrated mean squared error over the observed response locations, as described in Section~\ref{s:implementation}.
For early stopping, $20\%$ of the profiles in each training period are randomly held out as a validation set, and training is terminated if the validation loss does not improve for $30$ consecutive epochs.
Further details on the training settings are provided in Table~\ref{tab:argo_training_settings} in the Supplement.

\subsection{Results}

For each observed dissolved oxygen profile in 2021-2024, we produce a predicted profile based on the data in the preceding five years. We then compare the observed and predicted profiles in two ways.

First, for each prediction year and hemisphere--season combination among the four summer/winter and north/south combinations, we compute the pointwise means of the observed and predicted dissolved oxygen profiles on a common pressure grid and compare the resulting mean curves. Thus, the spatial information other than hemisphere is not taken into account.
Figure~\ref{fig:argo_mean_curves_hemi_season_2021_2024} presents the observed and predicted mean curves for the four hemisphere--season groups in each year from 2021 to 2024.
Overall, the predicted mean curves and the interquartile ranges follow closely those of the observed ones.
The noticeable discrepancies occur in 2021, particularly during June--August, corresponding to Northern Hemisphere (NH) summer and Southern Hemisphere (SH) winter, when the predicted mean dissolved oxygen levels are slightly higher than the observed means in both hemispheres.
This may be due to the fact that 2021–2022 were truly anomalous years in the southern hemisphere, driven by an interacting group of factors (triple-dip La Ni\~{n}a, record Southern Ocean heat content, unprecedented sea ice loss, extreme atmospheric events). Any of these, or their combination, could plausibly shift the temperature–doxy relationship from what we expect based on previous years. For the remaining groups and prediction years, the discrepancies are generally small, indicating that SNO captures the main seasonal and hemispheric patterns
in the oxygen profiles.

\begin{figure}[htbp]
\centering
\begin{minipage}[t]{0.48\textwidth}
\centering
\includegraphics[width=\textwidth]{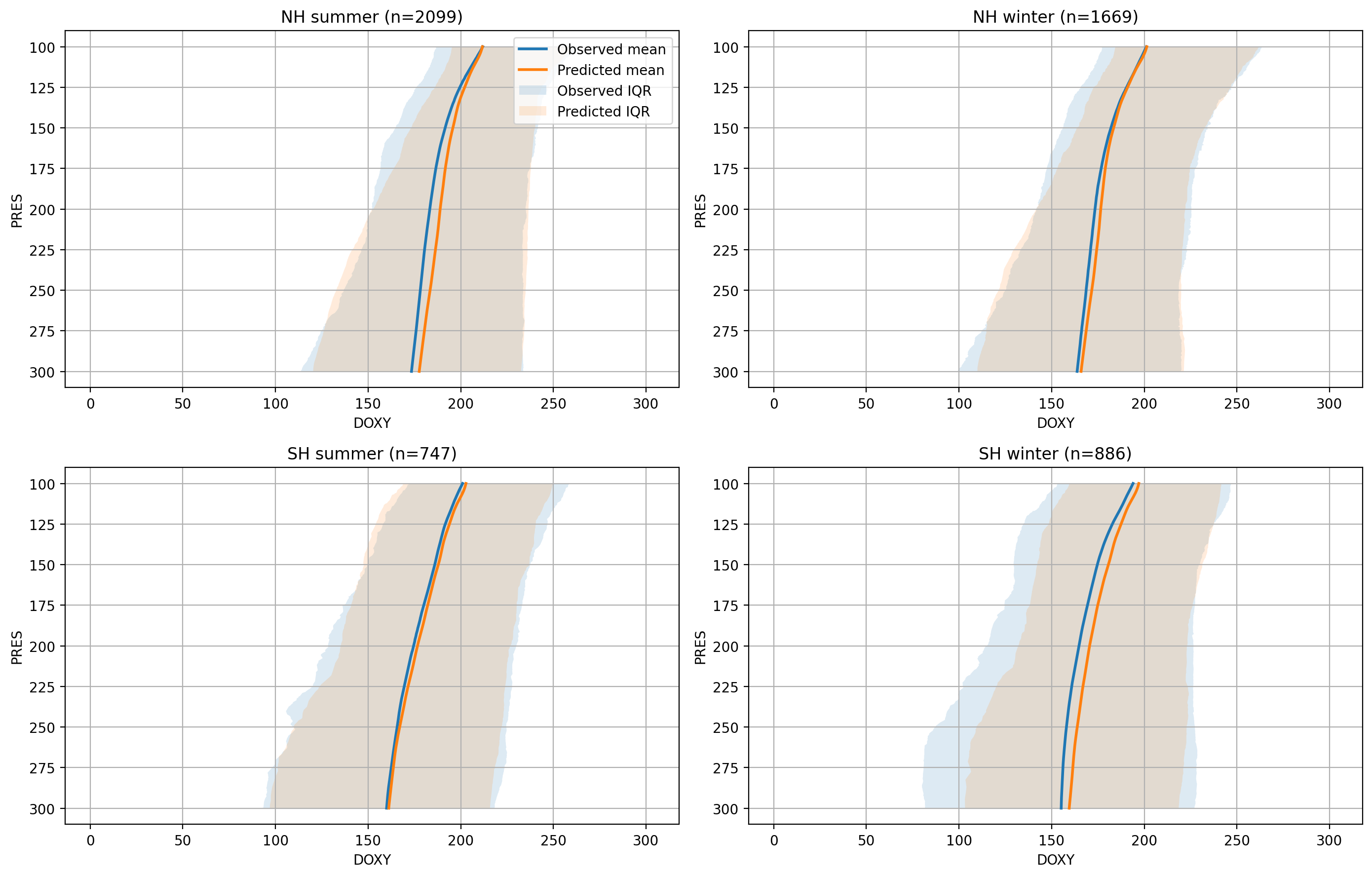}
\footnotesize Observed and predicted mean profiles for 2021
\end{minipage}
\hfill
\begin{minipage}[t]{0.48\textwidth}
\centering
\includegraphics[width=\textwidth]{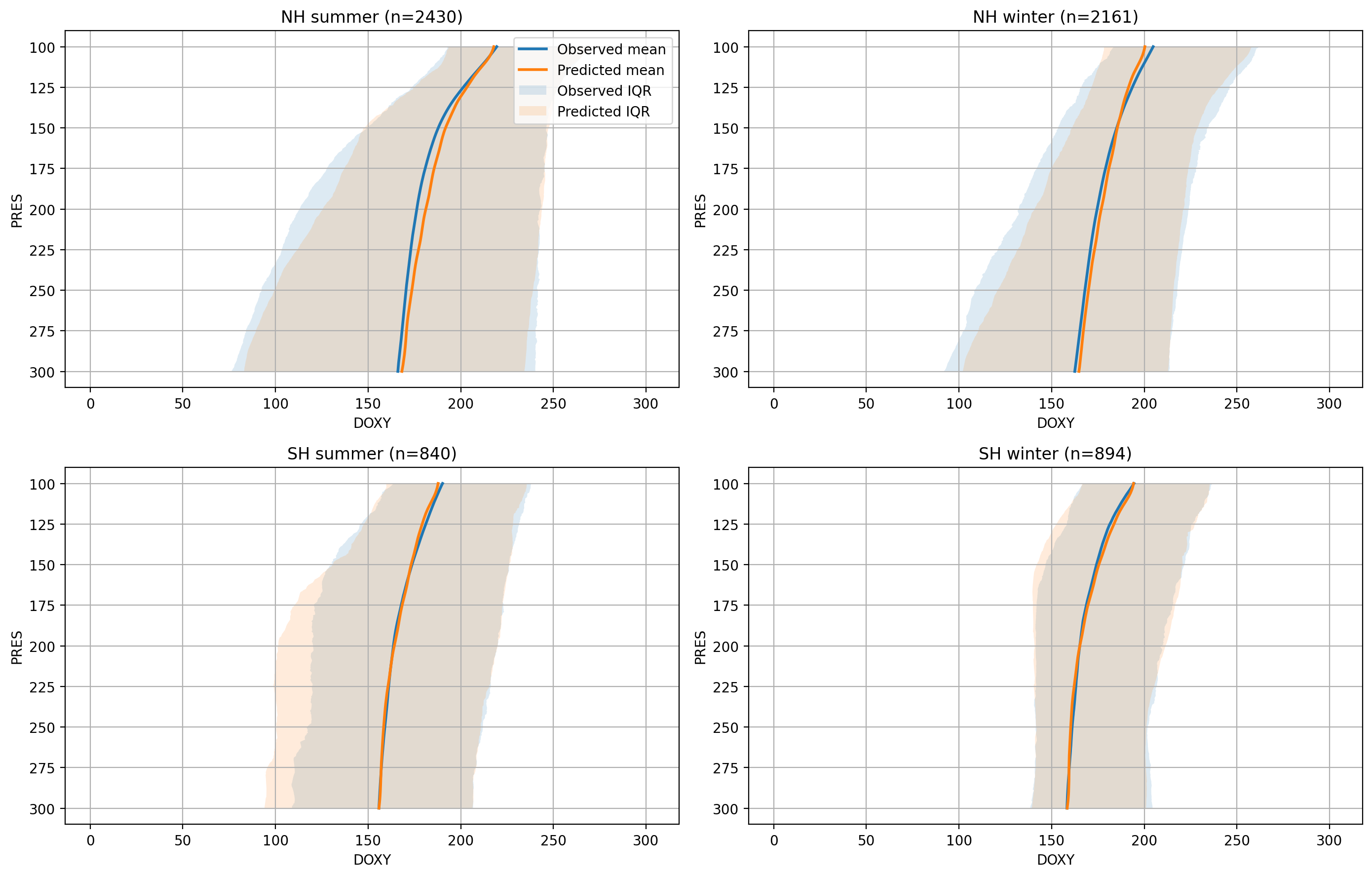}
\footnotesize Observed and predicted mean profiles for 2022
\end{minipage}

\vspace{0.5em}

\begin{minipage}[t]{0.48\textwidth}
\centering
\includegraphics[width=\textwidth]{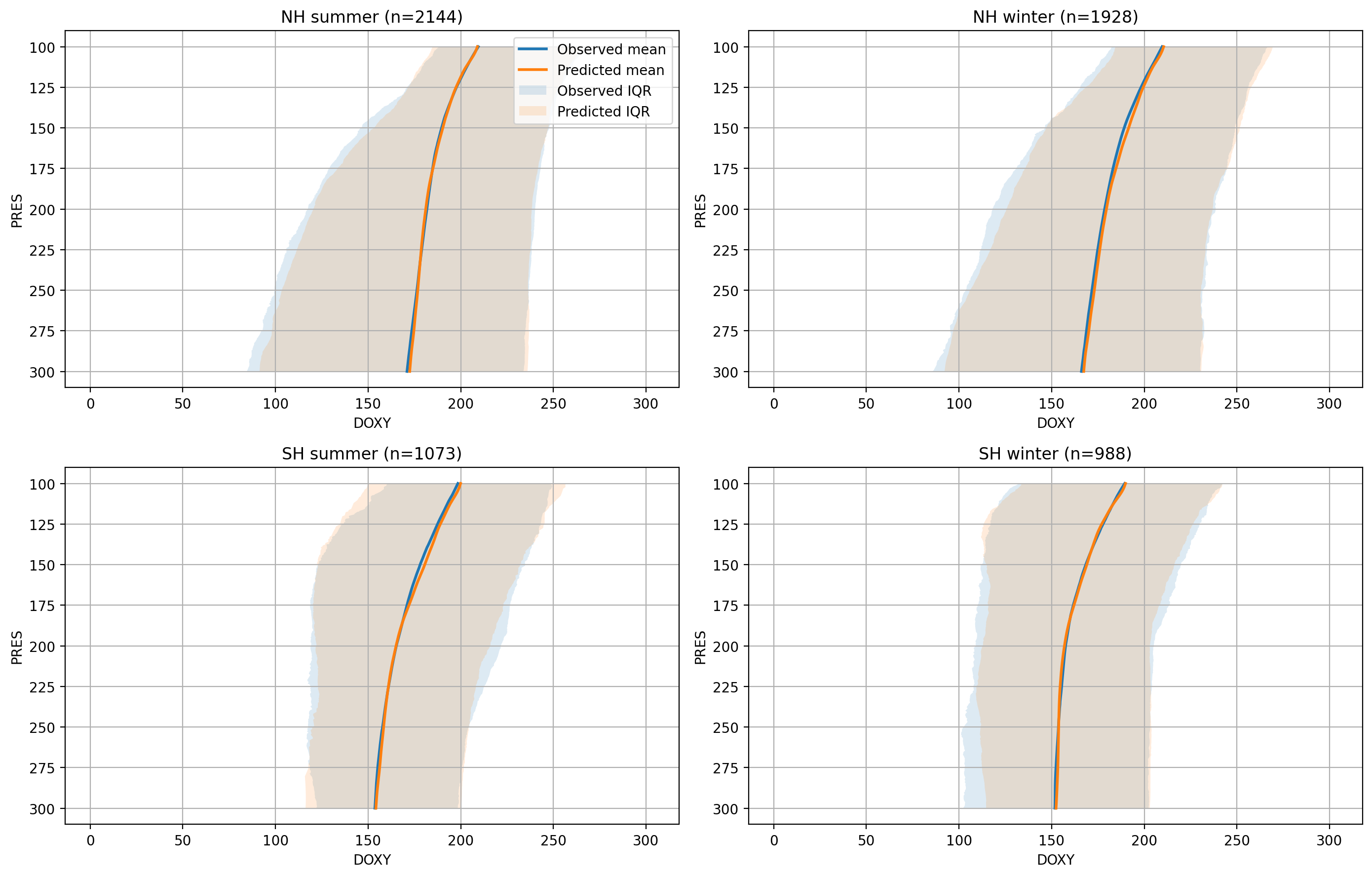}
\footnotesize Observed and predicted mean profiles for 2023
\end{minipage}
\hfill
\begin{minipage}[t]{0.48\textwidth}
\centering
\includegraphics[width=\textwidth]{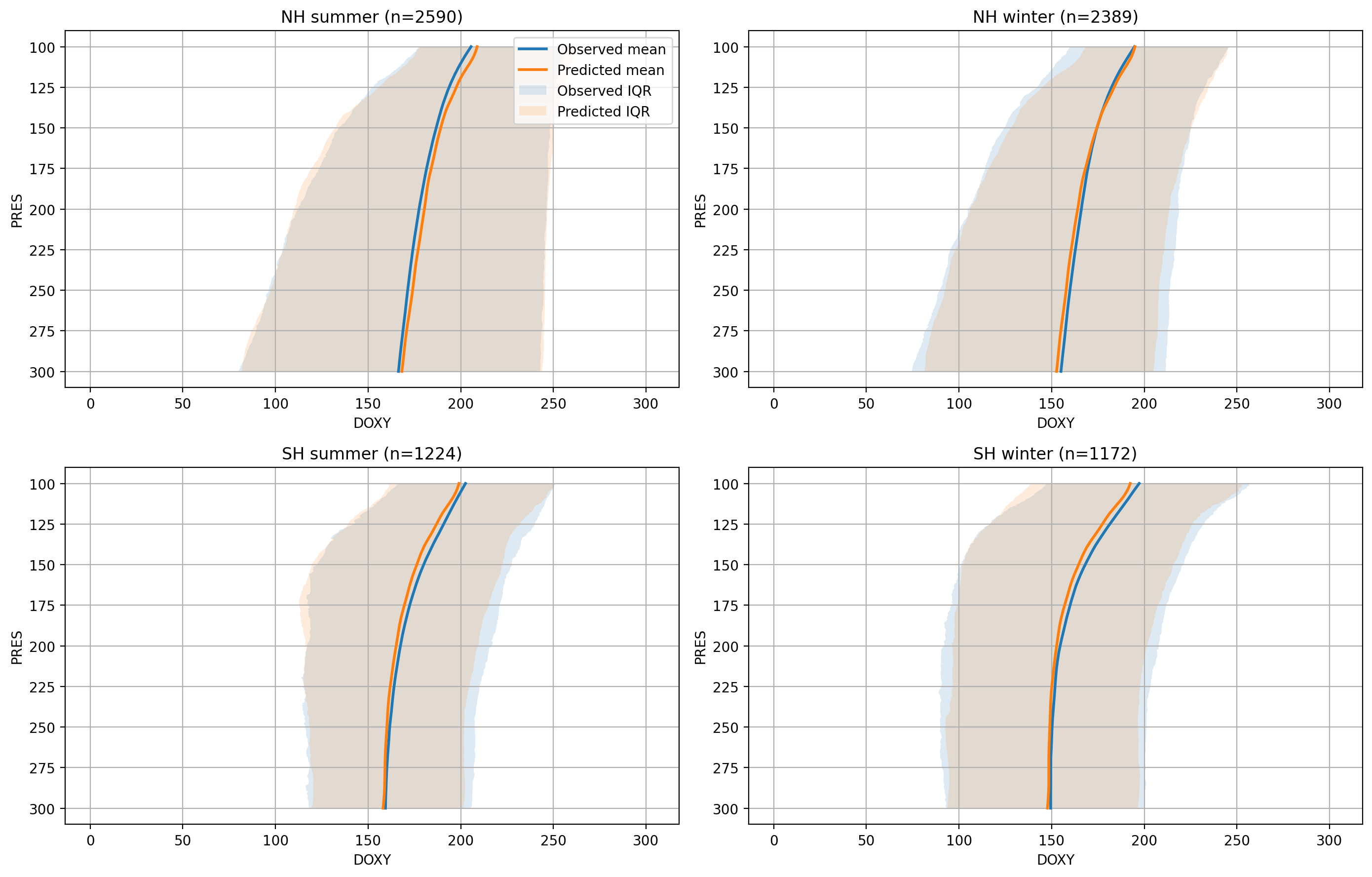}
\footnotesize Observed and predicted mean profiles for 2024 
\end{minipage}
\vskip.3cm
\caption{Observed and predicted mean dissolved oxygen profiles in 2021-2024 by hemisphere and season for the one-year-ahead predictions.
The horizontal and vertical axes represent pressure (dbar) and dissolved oxygen concentration ($\mu$mol/kg), respectively.
The solid lines represent the observed and predicted means, and the shaded regions represent the corresponding interquartile ranges.
Summer and winter correspond to June--August and December--February, respectively, in the Northern Hemisphere, with the seasonal labels reversed in the Southern Hemisphere.
The observed dissolved oxygen profiles are obtained by linear interpolation onto the common pressure grid.}
\label{fig:argo_mean_curves_hemi_season_2021_2024}
\end{figure}

In addition to the pressure-wise comparison based on the mean curves, we compare the observed and predicted spatial distributions of dissolved oxygen level at $200$ dbar, chosen as a representative pressure level in the range $[100,300]$ dbar.
Figure~\ref{fig:argo_spatial_comparison_200dbar}   displays the observed and predicted dissolved oxygen concentrations at $200$ dbar for the four hemisphere--season groups, with color representing the oxygen concentration; see the caption for computational details.
The SNO predictions broadly capture the observed large-scale spatial patterns of high- and low-oxygen regions in both hemispheres and in both seasons.
Although some local discrepancies remain, the overall spatial structures of the observed and predicted distributions are similar.

\begin{figure}[t]
\centering
\begin{minipage}[t]{0.49\textwidth}
\centering
\includegraphics[width=\textwidth]{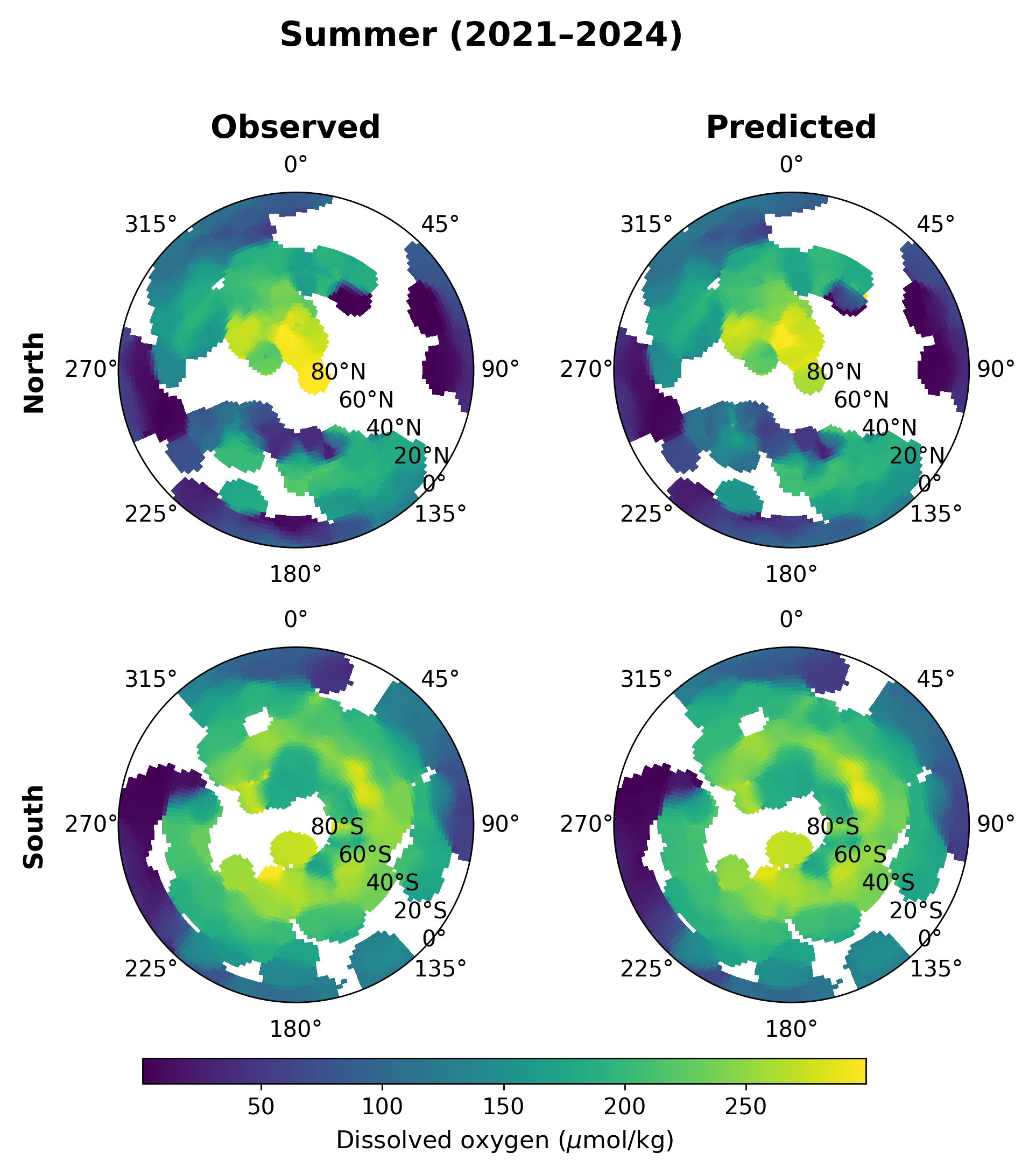}
\end{minipage}
\hfill
\begin{minipage}[t]{0.49\textwidth}
\centering
\includegraphics[width=\textwidth]{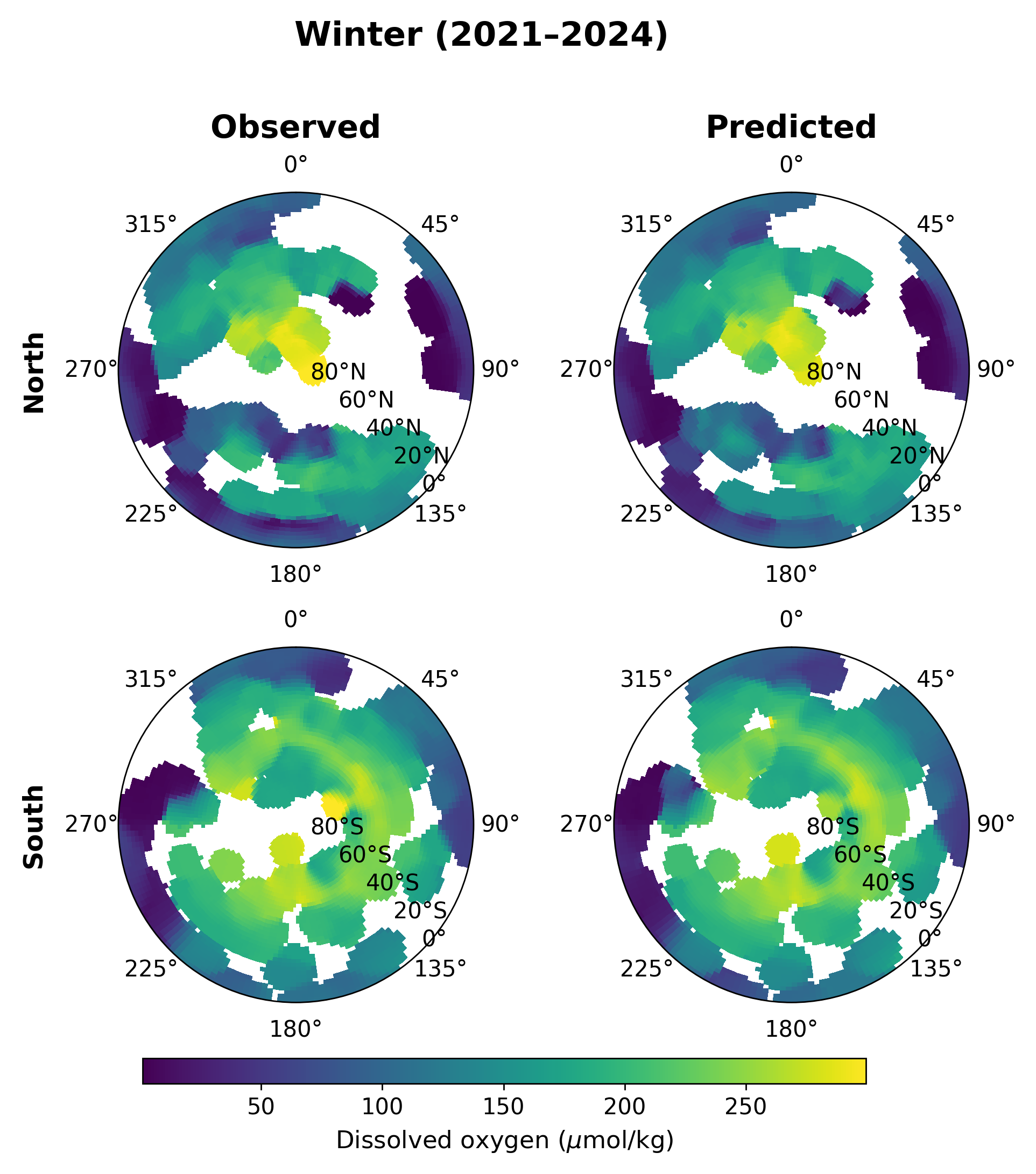}
\end{minipage}
\caption{Spatial comparison of observed and predicted dissolved oxygen concentrations at $200$ dbar in summer and winter.
Colors represent dissolved oxygen concentration.
Evaluation points are placed at $2^\circ$ intervals in latitude and longitude.
At each evaluation point, a Gaussian-kernel weighted average is computed from up to the $60$ nearest profiles within an angular distance of $8^\circ$, using spherical distance and bandwidth $4^\circ$.
The same profile locations, neighboring profiles, and weights are shared between the observed and predicted maps.
No estimate is shown when fewer than three profiles are available within $8^\circ$ of an evaluation point; the corresponding regions are shown in white.
Each estimate is displayed in the cell centered at the corresponding evaluation point, so the adjacent colored cells do not represent interpolation between evaluation points.}
\label{fig:argo_spatial_comparison_200dbar}
\end{figure}

Similar prediction results were obtained in \citet{giglio2018estimating}  and \citet{korte2025functional} based on different approaches. Direct comparisons with their results are not feasible since they considered different (training and testing) data. We are currently planning a study that will systematically compare all the approaches that deal with this problem.

\clearpage

\begingroup
\fontsize{10.2}{12.75}\selectfont  
\let\suppoldtabular\tabular
\renewcommand\tabular{\fontsize{10pt}{12.5pt}\selectfont\suppoldtabular}

\renewcommand{\thethm}{S.\arabic{thm}}
\renewcommand{\thelemma}{S.\arabic{lemma}}
\renewcommand{\theprop}{S.\arabic{proposition}}
\renewcommand{\theassumption}{S.\arabic{assumption}}
\renewcommand{\thedefn}{S.\arabic{defn}}
\renewcommand{\theexample}{S.\arabic{example}}
\renewcommand{\theremark}{S.\arabic{remark}}
\renewcommand{\thetable}{S.\arabic{table}}
\renewcommand{\thefigure}{S.\arabic{figure}} %
\renewcommand{\thesection}{S.\arabic{section}}
\renewcommand{\theequation}{S.\arabic{equation}}

\setcounter{thm}{0}
\setcounter{assumption}{0}
\setcounter{defn}{0}
\setcounter{example}{0}
\setcounter{remark}{0}
\setcounter{table}{0}
\setcounter{figure}{0}
\setcounter{section}{0}
\setcounter{equation}{0}

\noindent
\begin{center}\large\bfseries Supplementary Material for Function-On-Function Regression Through Separable Neural Operators
\end{center}
\vskip.2cm
This Supplement contains the proof of Theorem~\ref{t:rate} as well as certain technical and computational details for the previous sections. For clarity, all the labels and equation numbers in the Supplement will be prefixed by S. A notation summary is provided in Section \ref{s:notation} for easy reference.

\section{Proof of Theorem \ref{t:rate}}

The proof of Theorem~\ref{t:rate} relies on Propositions~\ref{p:rate_3}--\ref{p:hatRn} below. The auxiliary lemmas needed for these propositions are collected in Section \ref{s:auxiliary}, and the propositions themselves are proved in Sections \ref{ss:p:rate_3}-\ref{ss:p:hatRn}. These propositions compare the predictive risk defined by \eqref{e:pred_risk}:
$$
\CR\left(\wh G_{n,u},G_{0,n,u}\right) = \E\left[\left\|\wh G_{n,u}(X)-G_{0,n,u}(X)\right\|_{\bbL^2}^2\right]
$$
with two intermediate quantities. To describe those quantities, 
we first define the following empirical norms:
\begin{itemize}
\item 
The vector norm of $G(x)$ based on the output grid $\{t_\ell\}$: $\|G(x)\|_L = \left(\sum_{\ell=1}^{L} w_\ell \left(G(x)(t_\ell)\right)^2\right)^{1/2}$.
\item 
The empirical norm of $G$ over the training sample $\{X_i\}_{i=1}^n$ and output grid $\{t_\ell\}_{\ell=1}^L$: \[\|G\|_n = \left(\frac1n\sum_{i=1}^n \sum_{\ell=1}^{L} w_\ell \big(G(X_i)(t_{\ell})\big)^2\right)^{1/2}
= \left(\frac1n\sum_{i=1}^n \|G(X_i)\|_L^2
\right)^{1/2}.\]
\end{itemize}
Here $\{w_\ell\}_{\ell=1}^L$ denotes the quadrature weights associated with
the common output grid $\{t_\ell\}_{\ell=1}^L$, as defined in~\eqref{e:w_1}.
For a dense grid $\{t_\ell\}$, $\|G(x)\|_L$ approximates $\|G(x)\|_{\bbL^2}$, and if, additionally, $n$ is large, then
$\|G\|_n$ approximates $\left(\E\|G(X)\|_{\bbL^2}^2\right)^{1/2}$.

The first intermediate quantity in the analysis of $\CR\left(\wh G_{n,u},G_{0,n,u}\right)$ is based on the $\|\cdot\|_n$-norm :
\begin{align} \label{e:hat_R_n}
\wh R_n(\wh G_{n,u},G_{0,n,u}) := \E\left[ \left\|\wh G_{n,u}-G_{0,n,u}\right\|_n^2\right],
\end{align}
and the second is the corresponding population quantity based on the discrete $\|\cdot\|_L$-norm :
\begin{align} \label{e:hat_R_L}
R\left(\wh G_{n,u},G_{0,n,u}\right) := \E\left[\left\|\wh G_{n,u}(X)-G_{0,n,u}(X)\right\|_L^2\right],
\end{align}
where $X$ is an independent copy of the predictor. In~\eqref{e:hat_R_n}, the expectation is taken with respect to the training sample, while in~\eqref{e:hat_R_L} it is taken with respect to both the training sample and the independent copy $X$.

The proof proceeds by comparing
$\CR(\wh G_{n,u},G_{0,n,u})$ with $R(\wh G_{n,u},G_{0,n,u})$ (Proposition \ref{p:rate_3}),
then comparing $R(\wh G_{n,u},G_{0,n,u})$ with
$\widehat R_n(\wh G_{n,u},G_{0,n,u})$ (Proposition \ref{p:rate_2}), and finally bounding
$\widehat R_n(\wh G_{n,u},G_{0,n,u})$ (Proposition \ref{p:hatRn}).
To state the propositions, define the following quantities. Recall that $\|G\|_{u,\infty}\le Bu^{\nu}$ for
$G\in\CG_{n,u}$, while Assumption~\ref{a:G_0}(i) gives the growth bound
$\|G_0\|_{u,\infty}\le C_0u^{\nu}$ for $u\ge1$. Set
\begin{align} \label{e:Delta}
\begin{split}
B_u &:= 2 (B+C_0) u^{\nu}, \\
\Delta_{n,u,\epsilon}
&:= 4 C_0^2\,(u+\epsilon)^{2\nu}\epsilon^{2\alpha} + 4 \E\left[\ind{\|X\|_{\sup,S_{m_n}} \le u < \|X\|_{\sup}} \|G_0(X)\|_{\sup}^2\right] \\
&\quad + 4(B+2C_0)^2 u^{2\nu}
\exp\left(-\frac{\epsilon^2}{8\lambda^2 \rho^{2\beta}_{m_n}} \right), \\
\Omega_{n,u,\delta,\epsilon} &:= \sqrt{n^{-1} \log \CN_{n,u,\delta}} \vee  \sqrt{\Delta_{n,u,\epsilon}}.
\end{split}
\end{align}
Here $C_0$ and $B$ are the constants defined in Assumption~\ref{a:G_0} and~\eqref{e:Gn}, respectively, $\CN_{n,u,\delta}$ is the $\delta$-covering number from Remark~\ref{r:cover_number} of Section \ref{s:theory}, and $m_n$ is the sequence of positive integers, with $m_n\uparrow\infty$, defined in~\eqref{e:m_n_1} below. 

\begin{prop} \label{p:rate_3}
Under Assumptions \ref{a:process}--\ref{a:network}, for each fixed $u\ge 1$, there exists a constant $C_u<\infty$, independent of $n$, such that for all $n$,
\begin{align*}
& \Big| \CR\left(\wh G_{n,u},G_{0,n,u}\right) - R\left(\wh G_{n,u},G_{0,n,u}\right) \Big|
\le C_u\, p_nq_nb_n^2\gamma_n\sigma_{b_n( r_n u +1)},
\end{align*}
where $\sigma_s := \sup_{|t|\le s}|\sigma(t)|$. 
\end{prop}

\begin{prop} \label{p:rate_2}
Under Assumptions~\ref{a:process}--\ref{a:network}, for each fixed $u\ge1$
and for any fixed $0<\epsilon\le 1$ and $\delta>0$, the following bound holds for
all sufficiently large $n$:
\begin{align} \label{e:R-Rhat_1}
\begin{split}
R\left(\wh G_{n,u}, G_{0,n,u}\right) & \le 2 \widehat{R}_n\left(\wh G_{n,u}, G_{0,n,u}\right) + 2 C_1B_u^2\Omega_{n,u,\delta,\epsilon}\left(\sqrt{\frac{\log \CN_{n,u,\delta}}{n}}+\delta\right)  \\
& \quad + C_2B_u^4 \left(\Omega_{n,u,\delta,\epsilon}^2 +n^{-1}\right) + 4(B_u\delta + \delta^2)
\end{split}
\end{align}
for some universal constants $C_1,C_2 < \infty$, where $B_u$ and $\Omega_{n,u,\delta,\epsilon}$ are as defined in \eqref{e:Delta}.
\end{prop}

\begin{prop} \label{p:hatRn}
Under Assumptions~\ref{a:process}--\ref{a:network}, for each fixed $u\ge1$
and for any fixed $0<\epsilon\le 1$ and $\delta>0$, the following bound holds for
all sufficiently large $n$:
\begin{align} \label{e:step_1_bound_1}
\begin{split}
&\widehat{R}_n\left(\wh G_{n,u}, G_{0,n,u}\right) \\
& \le 2 \Delta_{n,u,\epsilon}+ 4\delta \sigma_{\varepsilon,e} + 4 B_{\varepsilon,e} \delta \sqrt{\frac{3 \log \CN_{n,u,\delta}+1}{n}} + \frac{4 B_{\varepsilon,e}^2 (3 \log \CN_{n,u,\delta}+1)}{n},
\end{split}
\end{align}
where $\Delta_{n,u,\epsilon}$ is defined in~\eqref{e:Delta}, 
\begin{align} \label{e:B_varepsilon_e}
\sigma_{\varepsilon,e} := \sqrt{\sup_{t\in [0,1]} C_\varepsilon(t,t)+\sigma_e^2}, \quad
\mbox{and} \quad B_{\varepsilon,e}:=\sqrt{B_{\varepsilon}^2+\sigma_e^2},
\end{align}
with $B_{\varepsilon}$ and $\sigma_e^2$ defined in Lemma \ref{l:eigen} and Assumption \ref{a:sampling}, respectively.
\end{prop}

The proofs of Propositions~\ref{p:rate_3}--\ref{p:hatRn} are deferred to Sections~\ref{ss:p:rate_3}--\ref{ss:p:hatRn}. 
We now apply these propositions to prove Theorem~\ref{t:rate}.

\begin{proof}[Proof of Theorem~\ref{t:rate}]
Fix $u\ge 1$ and let $0<\epsilon\le 1$ be fixed but arbitrary. Let $\delta=\delta_n\downarrow 0$ be a sequence such that
$n^{-1}\log \CN_{n,u,\delta}\to 0$,
whose existence is guaranteed by \eqref{e:metric_entropy} and Assumption~\ref{a:network}. Since
\begin{align*}
\CR\!\left(\wh G_{n,u},G_{0,n,u}\right)
&\le
\Big|
\CR\!\left(\wh G_{n,u},G_{0,n,u}\right)
-
R\!\left(\wh G_{n,u},G_{0,n,u}\right)
\Big|
+
R\!\left(\wh G_{n,u},G_{0,n,u}\right),
\end{align*}
Propositions~\ref{p:rate_3}--\ref{p:hatRn} yield
\begin{align*}
\CR\!\left(\wh G_{n,u},G_{0,n,u}\right)
&\le
4\!\left(\Delta_{n,u,\epsilon}+2\delta\,
\sigma_{\varepsilon,e}\right)
+ 8 B_{\varepsilon,e} \delta 
  \sqrt{\frac{3\log \CN_{n,u,\delta}+1}{n}}
+ \frac{8 B_{\varepsilon,e}^2\big(3\log \CN_{n,u,\delta}+1\big)}{n}   \\
&\quad
+ 2 C_1 B_u^{2}\,
   \Omega_{n,u,\delta,\epsilon}
   \!\left(
      \sqrt{\frac{\log \CN_{n,u,\delta}}{n}}
      +\delta
   \right)
+ C_2 B_u^{4}\!\left(
     \Omega_{n,u,\delta,\epsilon}^{2}
     + n^{-1}
   \right) \\
&\quad + 4(B_u\delta+\delta^{2})                                  
+ C_u \, p_nq_nb_n^2\gamma_n\sigma_{b_n( r_n u +1)}
\end{align*}
for all $n$, $0<\epsilon\le1$, and $\delta>0$. Absorbing constants independent of $n,u,\epsilon,\delta$ into a generic constant $C<\infty$, we obtain
\begin{align*} 
\CR\!\left(\wh G_{n,u},G_{0,n,u}\right)
\le C \Bigg[
& \Delta_{n,u,\epsilon}+\delta + \delta \sqrt{\log \CN_{n,u,\delta}\over n} + {\log \CN_{n,u,\delta}\over n} + B_u^2\Omega_{n,u,\delta,\epsilon}\left(\sqrt{\frac{\log \CN_{n,u,\delta}}{n}}+\delta\right)  \\
& \quad + n^{-1} + B_u^4 \left(\Omega_{n,u,\delta,\epsilon}^2 +n^{-1}\right) + (B_u\delta + \delta^2)\Bigg] + C_u \, p_nq_nb_n^2\gamma_n\sigma_{b_n( r_n u +1)}.
\end{align*}
By the definition of $\Delta_{n,u,\epsilon}$ in \eqref{e:Delta} and the fact that $\frac{\log \CN_{n,u,\delta}}{n}\to 0$ (see part (c) of Remark \ref{r:cover_number}),
we further obtain
\begin{align}\label{e:rate_bound}
\CR\!\left(\wh G_{n,u},G_{0,n,u}\right)
\le D_u\left[ \Delta_{n,u,\epsilon} +\delta
+\frac{\log \CN_{n,u,\delta}}{n} +n^{-1}
+ p_nq_nb_n^2\gamma_n\sigma_{b_n(r_nu+1)} \right],
\end{align}
where $D_u$ is a generic constant independent of $n,\delta,\epsilon$ and absorbs constants involving $B_u$ and $C_u$.
It remains to analyze $\Delta_{n,u,\epsilon}$.
Since $X$ has 
continuous sample paths (cf.~Remark \ref{r:process}), $\ind{\|X\|_{\sup,S_{m_n}} \le u < \|X\|_{\sup}}$ converges to $0$ a.s.\ as the grid $S_{m_n}$ becomes dense. 
By~\eqref{e:G_0_sup} and the Lebesgue Dominated Convergence Theorem,
$$
\lim_{n\to\infty}\E\left[\ind{\|X\|_{\sup,S_{m_n}} \le u < \|X\|_{\sup}} \|G_0(X)\|_{\sup}^2\right] = 0,
$$
which, along with $\rho^{2\beta}_{m_n}\to 0$, implies that
$$
\limsup_{n\to\infty} \Delta_{n,u,\epsilon} \le  4\, C_0^2\, (u+\epsilon)^{2\nu}\epsilon^{2\alpha} \quad \mbox{for every fixed $u\ge 1$ and $0<\epsilon\le 1$}.
$$
Combining this bound with \eqref{e:rate_bound}, and using
$\delta=\delta_n\downarrow0$,
$n^{-1}\log\CN_{n,u,\delta_n}\to0$, and
\[
p_nq_nb_n^2\gamma_n\sigma_{b_n( r_n u +1)}\to0,
\]
we obtain
$$
\limsup_{n\to\infty} \CR\left(\wh G_{n,u},G_{0,n,u}\right)
\le D_u \epsilon^{2\alpha}.
$$
Letting $\epsilon\downarrow 0$ now yields the desired result in Theorem~\ref{t:rate}.
\end{proof}

\section{Auxiliary lemmas} \label{s:auxiliary}

In this subsection, we collect a few results that facilitate the proofs of Propositions~\ref{p:rate_3}--\ref{p:hatRn}. The first result below gives a bound on the covering number of the network class $\CG_{n,u}$, which plays a critical role in the proofs. Recall that $\CN_{n,u,\delta}$ denotes the $\delta$-covering number of $\CG_{n,u}$ under the metric $\|\cdot\|_{u,\infty}$.

\begin{lemma}[Covering number of $\CG_{n,u}$] \label{l:covering}
Let $\|\cdot\|_{u,\infty}$ be defined by \eqref{e:infinity_norm}, and, for simplicity, we suppress $n$ in $p_n,q_n, r_n, b_n, J_n$.
Then
\begin{equation*}   
\CN_{n,u,\delta}
  \leq
  \left(\frac{C_1 bpq\sigma_{b(ru+1)}^2}{\delta}\right)^{pq}
  \cdot
  \left[\binom{J}{r}
    \left(\frac{C_2 b^2 pq (2ru+1) \sigma_{b(ru+1)}}{\delta}\right)^{r+1}
  \right]^{pq}
  \cdot
  \left(\frac{C_3 b^2 pq\sigma_{b(ru+1)}}{\delta}\right)^{2p},
\end{equation*}
where $C_1$, $C_2$, $C_3$ are constants independent of $n,u,\delta$, and where $\sigma_s = \sup_{|t|\le s}|\sigma(t)|$. Taking logarithms, the metric
entropy satisfies
\begin{align}   \label{e:metricentropy}
\begin{split}
\log \CN_{n,u,\delta}
  &\leq pq\log\frac{C_1 bpq\sigma_{b(ru+1)}^2}{\delta}
  + pq(r+1)\log\frac{C_2 b^2 pq(2ru+1)\sigma_{b(ru+1)}}{\delta} \\
  &\quad + pq\log\binom{J}{r}
  + 2p\log\frac{C_3 b^2 pq\sigma_{b(ru+1)}}{\delta}.
  \end{split}
\end{align}
\end{lemma}

If $\sigma$ is bounded, for example when $\sigma$ equals sigmoid or tanh, then $\sigma_{b(ru+1)}=O(1)$.
If $\sigma$ is ReLU, then $\sigma_{b(ru+1)}=O\{b(ru+1)\}$, and this growth is accounted for in the
corresponding entropy bound and the sieve complexity conditions.

\begin{proof}
Let $G\in\CG_{n,u}$. On the domain $\{\|x\|_{\sup}\le u\}$, the truncation
indicator $\ind{\|x\|_{\sup,S_n}\le u}$ is equal to one, and hence $G$ can be written as
$$
G(x)(t) = \sum_{k=1}^p\sum_{i=1}^q c_i^k\, f_i^k(x)\, h_k(t),
$$
where, for convenience, $x_j:=x(s_j)$, and
\begin{align*}
  f_i^k(x):= \sigma\!\left(\sum_{j=1}^{J}\xi_{ij}^k x_j + \theta_i^k\right), \qquad
  h_k(t):= \sigma(w_k t + \zeta_k).
\end{align*}
Let $G$ and $\tilde G$ be defined by the two sets of parameters $(c_i^k, \xi_{ij}^k, \theta_i^k, w_k, \zeta_k)$ and 
$(\tilde c_i^k, \tilde\xi_{ij}^k, \tilde\theta_i^k, \tilde w_k, \tilde\zeta_k)$, respectively, 
both satisfying the sparsity constraint 
\begin{equation}
  \sum_{j=1}^{J}\ind{\xi_{ij}^k\neq 0} \leq r,
  \quad \sum_{j=1}^{J}\ind{\tilde\xi_{ij}^k\neq 0} \leq r
  \quad\text{for all } i=1,\ldots,q,\; k=1,\ldots,p.
  \label{eq:sparsity}
\end{equation}
Fix $x$ with $\norm{x}_{\sup,S_n}\leq u$. Let
\begin{equation}
  \delta_1 := \max_{i,j,k}\max\!\left(
    \abs{\xi_{ij}^k - \tilde\xi_{ij}^k},\;
    \abs{\theta_i^k - \tilde\theta_i^k}
  \right).
  \label{eq:delta_xi}
\end{equation}
Define  $C_\sigma:=\esssup_{t\in\mathbb{R}} |\sigma'(t)|$.
Then
\begin{align}
 & \abs{f_i^k(x) - \tilde f_i^k(x)}
  \leq C_\sigma\left(
    \sum_{j=1}^{J}\abs{\xi_{ij}^k - \tilde\xi_{ij}^k}\cdot\abs{x_j}
    + \abs{\theta_i^k - \tilde\theta_i^k}
  \right) \notag\\
  &\leq C_\sigma\left(
    \sum_{j:\xi_{ij}^k\neq 0\text{ or }\tilde\xi_{ij}^k\neq 0}
    \delta_1 \cdot u + \delta_1
  \right)  
  \leq C_\sigma(2 ru\,\delta_1 + \delta_1)
  = C_\sigma(2ru+1)\,\delta_1,
  \label{eq:f_perturb}
\end{align}
where we used $\abs{x_j}\leq u$ and that the union of the two sparse
supports has at most $2r$ elements. 

Next, define
\begin{equation*}
  \delta_2 := \max_k\max\!\left(\abs{w_k - \tilde w_k},\;\abs{\zeta_k - \tilde\zeta_k}\right).
\end{equation*}
For $t\in[0,1]$,
\begin{equation}
  \abs{h_k(t) - \tilde h_k(t)}
  \leq C_\sigma\left(\abs{w_k - \tilde w_k}\cdot t + \abs{\zeta_k - \tilde\zeta_k}\right)
  \leq C_\sigma\left(\delta_2 + \delta_2\right)
  = 2C_\sigma \delta_2.
  \label{eq:h_perturb}
\end{equation}
Define also 
\begin{equation}
  \delta_3 := \max_{i,k}\abs{c_i^k - \tilde c_i^k}.
  \label{eq:delta_c}
\end{equation}
By the three-term telescoping decomposition,
\begin{equation}
  c_i^k f_i^k h_k - \tilde c_i^k\tilde f_i^k\tilde h_k
  = (c_i^k - \tilde c_i^k) f_i^k h_k
    + \tilde c_i^k(f_i^k - \tilde f_i^k) h_k
    + \tilde c_i^k\,\tilde f_i^k(h_k - \tilde h_k).
  \label{eq:telescope}
\end{equation}
Then, on $\{\|x\|_{\sup,S_n}\le u\}$ and for $t\in[0,1]$,
\[
|f_i^k(x)|\vee |\widetilde f_i^k(x)|\vee |h_k(t)| \vee |\widetilde h_k(t)| \le\sigma_{b(ru+1)}.
\]
Thus,
\begin{align*}
  \abs{c_i^k f_i^k h_k - \tilde c_i^k\tilde f_i^k\tilde h_k}
  &\leq \delta_3\sigma_{b(ru+1)}^2 + b\cdot C_\sigma(2 ru+1)\delta_1\cdot\sigma_{b(ru+1)} + b\cdot\sigma_{b(ru+1)} \cdot 2C_\sigma\delta_2.
\end{align*}
Summing over all $pq$ pairs $(i,k)$:
\begin{equation}
  \norm{G - \tilde G}_{u,\infty}
  \leq pq\left[\delta_3\sigma_{b(ru+1)}^2
    + C_\sigma b(2ru+1)\,\delta_1\,\sigma_{b(ru+1)}
    + 2 C_\sigma b\,\delta_2\,\sigma_{b(ru+1)}\right].
  \label{eq:lip_bound}
\end{equation}
To achieve $\norm{G-\tilde G}_{u,\infty}\leq\delta$, it suffices by
\eqref{eq:lip_bound} to choose $\delta_1$, $\delta_2$, $\delta_3$ such that
\begin{equation}
  pq\left[\delta_3\sigma_{b(ru+1)}^2
    + C_\sigma b(2ru+1)\,\delta_1\,\sigma_{b(ru+1)}
    + 2 C_\sigma b\,\delta_2\,\sigma_{b(ru+1)}\right]
  \leq \delta.
  \label{eq:eps_condition}
\end{equation}
Allocating the budget equally among the three terms by setting
\begin{equation*}
  \delta_3 = \frac{\delta}{3pq\cdot \sigma_{b(ru+1)}^2}, \qquad
  \delta_1 = \frac{\delta}{3pq\cdot C_\sigma b(2ru+1)\sigma_{b(ru+1)}}, \qquad
  \delta_2 = \frac{\delta}{3pq\cdot 2C_\sigma b\sigma_{b(ru+1)}}.
\end{equation*}
It is straightforward to verify that these choices satisfy \eqref{eq:eps_condition}.

The $c$-parameters $\{c_i^k : i=1,\ldots,q,\; k=1,\ldots,p\}$ consist of $pq$
values in $[-b,b]$.  A $\delta_3$-net of $[-b,b]$ in absolute-value metric requires at
most $\lceil 2b/\delta_3\rceil +1$ points.  The number of elements in a $\delta_3$-net
of $[-b,b]^{pq}$ is therefore:
\begin{equation*}
  N_{\delta_3} \le \left(\left\lceil\frac{2b}{\delta_3}\right\rceil +1 \right)^{pq}
  \leq \left(\frac{6b pq\sigma_{b(ru+1)}^2}{\delta} + 2\right)^{pq}
  \le \left(\frac{C_1 b pq\sigma_{b(ru+1)}^2}{\delta}\right)^{pq}
\end{equation*}
for a sufficiently large constant $C_1<\infty$.

For each neuron $(i,k)$, the weight vector $(\xi_{ij}^k)_{j=1}^J$ has at most
$r$ nonzero entries (by the sparsity constraint \eqref{eq:sparsity}) and one bias
$\theta_i^k$.  We cover these parameters in two steps:
\biz
\im [1.] Cover the sparsity pattern. Choose a candidate active set $S_{ik}\subseteq\{1,\ldots,J\}$ with $|S_{ik}|=r$. There are $\binom{J}{r}$ possible choices of such $S_{ik}$. This also covers sparsity patterns of size less than $r$, since some coefficients within the chosen active set may be zero.
\im[2.] Cover the active weights and bias.
Given a fixed active set $S_{ik}$, the free parameters are
$(\xi_{ij}^k)_{j\in S_{ik}}$ and $\theta_i^k$, giving
$r+1$ parameters in $[-b,b]$.  A $\delta_1$-net of $[-b,b]^{r+1}$ has
at most $(\lceil 2b/\delta_1\rceil +1)^{r+1}$ points.
\eiz
Combining over all $pq$ neurons:
\begin{align*}
  N_{\delta_1}
  &\le \left[\binom{J}{r}\left(\left\lceil\frac{2b}{\delta_1}\right\rceil +1 \right)^{r+1}\right]^{pq} \\
  &\leq \left[\binom{J}{r}
    \left(\frac{6C_\sigma b^2 pq(2ru+1)\sigma_{b(ru+1)}}{\delta} + 2\right)^{r+1}
  \right]^{pq} \\
  &\le \left[\binom{J}{r}
    \left(\frac{C_2 b^2 pq(2ru+1)\sigma_{b(ru+1)}}{\delta}\right)^{r+1}
  \right]^{pq}
\end{align*}
for a sufficiently large constant $C_2<\infty$.
The parameters $\{w_k,\zeta_k : k=1,\ldots,p\}$ consist of $2p$ values in
$[-b,b]$.  A $\delta_2$-net of $[-b,b]^{2p}$ has cardinality bounded by
\begin{equation*}
  N_{\delta_2} \le \left(\left\lceil\frac{2b}{\delta_2}\right\rceil +1\right)^{2p}
  \leq \left(\frac{12C_\sigma b^2 pq\sigma_{b(ru+1)}}{\delta} + 2\right)^{2p}
  \le \left(\frac{C_3 b^2 pq\sigma_{b(ru+1)}}{\delta}\right)^{2p}
\end{equation*}
for a sufficiently large constant $C_3<\infty$.
It follows that $\CN_{n,u,\delta}\le N_{\delta_1} N_{\delta_2} N_{\delta_3}$. This proves the lemma.
\end{proof}

\begin{lemma}[Tail bounds for the operator $G_0$] \label{l:G0_moment}
Under Assumptions \ref{a:process} and \ref{a:G_0}, for each $u\ge1$,
the following quantities,
\[\E\left[\ind{\|X\|_{\sup} > u}\cdot \|G_0(X)\|_{\sup}^2\right]\qquad 
{\rm and}\qquad \E\left[\|G_{0,n,u}(X)-G_0(X)\|_{\sup}^2\right],\]
are bounded by $\Psi^2(u)$, where
\begin{align} \label{e:Psi}
\begin{split}
\Psi^2(u)
:=
C_0^2
\Bigg[
&u^{2\nu}
\exp\left\{
-\frac{(u-\E\|X\|_{\sup})_+^2}{2\sigma_X^2}
\right\} \\
&\quad
+
\int_{u^{2\nu}}^\infty
\exp\left\{
-\frac{(t^{1/(2\nu)}-\E\|X\|_{\sup})_+^2}{2\sigma_X^2}
\right\}dt
\Bigg].
\end{split}
\end{align}
with $(a)_+ :=\max\{a,0\}$ and $\sigma_X^2 = \sup_{t\in [0,1]} \var(X(t))$. 
In particular, $\Psi^2(u)\to0$ as $u\to\infty$. 
\end{lemma}

\begin{proof}
First, observe that
\begin{align*} 
&\|G_{0,n,u}(x)-G_0(x)\|_{\sup} 
= \left\|G_0(x) \ind{\|x\|_{\sup, S_n}\le u}-G_0(x)\right\|_{\sup} \\
&=
\ind{\|x\|_{\sup,S_n}>u}\|G_0(x)\|_{\sup}
\le \ind{\|x\|_{\sup} > u}\cdot \|G_0(x)\|_{\sup}.
\end{align*}
Therefore, it suffices to bound
\[\E\left[\ind{\|X\|_{\sup} > u}\cdot \|G_0(X)\|_{\sup}^2\right].\]
By Assumption~\ref{a:G_0}, for $u\ge1$,
\begin{align} \label{e:truncation}
\begin{split}
& \E\left[\ind{\|X\|_{\sup}>u}\cdot\|G_0(X)\|_{\sup}^2\right] \\
&\le 
C_0^2\,\E\left[\ind{\|X\|_{\sup}>u}\|X\|_{\sup}^{2\nu} \right] \\
&= C_0^2\left[ u^{2\nu}\pr(\|X\|_{\sup}>u) 
+ \int_{u^{2\nu}}^\infty \pr\left(\|X\|_{\sup}^{2\nu}>t\right)dt\right].
\end{split}
\end{align}
By the Borell-TIS inequality 
\citep[cf.][]{adler2007random}, 
\begin{align} \label{e:Borel}
\pr\left(\|X\|_{\sup} > u\right) 
\le \exp\left(-\frac{(u -\E\|X\|_{\sup})_+^2}{2\sigma_X^2}\right).
\end{align}
Combining \eqref{e:Borel} with \eqref{e:truncation}, we obtain
\begin{align*}
& \E\left[\ind{\|X\|_{\sup}>u}\cdot\|G_0(X)\|_{\sup}^2\right] \\
&\le
C_0^2
\Bigg[
u^{2\nu}
\exp\left\{
-\frac{(u-\E\|X\|_{\sup})_+^2}{2\sigma_X^2}
\right\} \\
&\hspace{0.5cm}
+
\int_{u^{2\nu}}^\infty
\exp\left\{
-\frac{(t^{1/(2\nu)}-\E\|X\|_{\sup})_+^2}{2\sigma_X^2}
\right\}dt
\Bigg]
=:\Psi^2(u).
\end{align*}
Finally, the first term in $\Psi^2(u)$ converges to zero as $u\to\infty$,
because the Gaussian tail dominates any polynomial factor. For the integral
term, after the change of variables $s=t^{1/(2\nu)}$, it is bounded by a
Gaussian tail integral multiplied by the polynomial factor $s^{2\nu-1}$,
which also converges to zero as $u\to\infty$. Hence $\Psi^2(u)\to0$.
\end{proof}

We next introduce the smoother
\begin{align} \label{e:smoother}
\wt X_n(s) = \E\!\left[X(s)\mid X(s_1),\dots,X(s_{J_n})\right],
\end{align}
which is the conditional expectation of $X(s)$ given the observations on the grid $S_n=\{s_1,\dots,s_{J_n}\}$.
It is the optimal predictor of $X(s)$ under mean squared prediction error. For $\epsilon>0$, define the truncated neighborhood class
\begin{align} \label{e:X_n_epsilon}
\mathfrak{X}_{n,\epsilon} := \left\{x\in C[0,1]:\|x-\wt x_n\|_{\sup}\le \epsilon,~~ \|x\|_{\sup}\le u\right\},
\end{align}
where $\wt x_n$ denotes the realization of the smoother in~\eqref{e:smoother}. Note that $u$ is left out in $\mathfrak{X}_{n,\epsilon}$ for simplicity of notation.


\begin{lemma}[Concentration bounds for the kriging residual and $G_0$ stability]
\label{l:GaussianKriging}
Under Assumptions~\ref{a:process}, \ref{a:G_0}, and~\ref{a:sampling}, the following hold. 
\begin{itemize}
\item[(i)]
Let
$R_n(s):=X(s)-\wt X_n(s)$ denote the kriging residual process, and define\\
$\kappa_n:=\sup_{s\in[0,1]}\Bigl(\var(R_n(s))\Bigr)^{1/2}$.
Then, 
\[
\kappa_n\le \lambda \rho_n^\beta,
\qquad \mbox{where}\qquad 
\rho_n:=\max_{0\le j\le J_n}(s_{j+1}-s_j)~ {\rm with}~s_0:=0 ~{\rm and}~ s_{J_n+1}:=1.
\]
Moreover, there exists a constant $C_\beta>0$ such that, for all sufficiently large $n$,
\[
\E\left[\|R_n\|_{\sup}\right]
\le
C_\beta \kappa_n\sqrt{\log(\lambda/\kappa_n)},
\]
and, for every fixed $\epsilon>0$, there exists $n_\epsilon$ such that for all
$n\ge n_\epsilon$,
\[
\pr\left(\|R_n\|_{\sup}>\epsilon\right)
\le
\exp\left(-\frac{\epsilon^2}{8\kappa_n^2}\right)
\le
\exp\left(-\frac{\epsilon^2}{8\lambda^2\rho_n^{2\beta}}\right).
\]

\item[(ii)]
$\sup_{x\in\mathfrak{X}_{n,\epsilon}}
\|G_0(\wt x_n)-G_0(x)\|_{\sup}
\le C_0 (u+\epsilon)^{\nu} \epsilon^\alpha$. In particular, if $u\ge1$ and $0<\epsilon\le1$, then $\sup_{x\in\mathfrak{X}_{n,\epsilon}}
\|G_0(\wt x_n)-G_0(x)\|_{\sup}\le C\,u^{\nu}\epsilon^\alpha$ for some constant $C>0$ independent of $u$, $\epsilon$, and $n$.
\end{itemize}
\end{lemma}

\begin{proof}
We first prove part~(i). 
Note that $R_n(s)$ is a centered Gaussian process. 
Fix $s\in[0,1]$. Since $\wt X_n(s)$ is the optimal $L^2$ predictor of $X(s)$
based on $\{X(s_1),\dots,X(s_{J_n})\}$, its mean squared prediction error is no larger than that of any linear predictor based on the grid values. 
Let $s_{j(s)}$ be the grid point closest to $s$, that is,
\[
|s-s_{j(s)}|
=
\min_{1\le j\le J_n}|s-s_j|.
\]
Then,
\[
\var\bigl(X(s)-\wt X_n(s)\bigr)
\le
\E\bigl(X(s)-X(s_{j(s)})\bigr)^2  
=
d_X^2(s,s_{j(s)}) 
\le
\lambda^2 |s-s_{j(s)}|^{2\beta}
\le
\lambda^2\rho_n^{2\beta}.
\]
Taking the supremum over $s\in[0,1]$ gives
$\kappa_n\le \lambda\rho_n^\beta$.
Next we bound $\E\|R_n\|_{\sup}$ using Dudley's entropy integral. Let
\[
d_{R_n}(s,t):=\sqrt{\var(R_n(s)-R_n(t))}, \qquad s,t\in[0,1].
\]
Since
$R_n(s)-R_n(t)=\{X(s)-X(t)\} -\E\big[X(s)-X(t)\mid X(s^*),\,s^*\in S_n\big]$,
we have, for any $s,t\in[0,1]$,
\begin{align*}
d_{R_n}^2(s,t) & = \var\{R_n(s)-R_n(t)\} 
=
\E\Bigl[\var\bigl(X(s)-X(t)\mid X(s^*),\,s^*\in S_n\bigr)\Bigr] \\
&\quad\quad \le
\var\bigl(X(s)-X(t)\bigr)
=
d_X^2(s,t)
\le
\lambda^2|s-t|^{2\beta},
\end{align*}
where the last inequality follows from Assumption~\ref{a:process}(i). 
Let $N(\eta,[0,1],d_{R_n})$ be the covering number of the interval $[0,1]$
with respect to the metric $d_{R_n}$, using balls of radius $\eta$.
Since $d_{R_n}(s,t)\le \lambda |s-t|^\beta$, 
any Euclidean interval of radius $(\eta/\lambda)^{1/\beta}$ is contained in a
$d_{R_n}$-ball of radius $\eta$. 
Thus, the number of $d_{R_n}$-ball of radius $\eta$ needed to cover $[0,1]$ is no more than $(2(\eta/\lambda)^{1/\beta})^{-1}+1$, i.e., 
$$
N(\eta, [0,1],d_{R_n})\le 1+
\frac{1}{2}\left(\frac{\lambda}{\eta}\right)^{1/\beta},
$$
where the additive $1$ accounts for possible endpoint effect. 
For $0<\eta\le\lambda$, we have
$\left(\lambda/\eta\right)^{1/\beta}\ge 1$. Therefore,
\[
N(\eta, [0,1],d_{R_n})\le 1+\frac12\left(\frac{\lambda}{\eta}\right)^{1/\beta}
\le
\left(1+\frac12\right)
\left(\frac{\lambda}{\eta}\right)^{1/\beta}
\le
2\left(\frac{\lambda}{\eta}\right)^{1/\beta}.
\]
By Dudley's entropy bound~\citep[Theorem 1.3.3]{adler2007random}, there exists a universal constant $C>0$ such that
\begin{align} \label{e:entropy}
\E\Big[\sup_{s\in[0,1]}R_n(s)\Big]
\le C\int_0^{d_n/2}
\sqrt{\log N(\eta)}\,d\eta 
\le C\int_0^{\kappa_n}
\sqrt{
\log 2+\frac{1}{\beta}\log\left(\frac{\lambda}{\eta}\right)}\,d\eta,
\end{align}   
where $d_n:= \sup_{s,t\in[0,1]} d_{R_n}(s,t) \le 2\kappa_n$.
Using the change of variable $\eta=\kappa_n t$, we have
\[\int_0^{\kappa_n}
\sqrt{\log 2+\frac{1}{\beta}\log\left(\frac{\lambda}{\eta}\right)
}\,d\eta   
=\kappa_n \int_0^1
\sqrt{
\log 2+\frac{1}{\beta}\log\left(\frac{\lambda}{\kappa_n}\right)
+\frac{1}{\beta}\log(t^{-1})
}\,dt .
\]
For all sufficiently large $n$, $\log(\lambda/\kappa_n)\ge 1$. Note that 
$\int_0^1\sqrt{\log(t^{-1})}\,dt
=\Gamma(3/2)=\frac{\sqrt{\pi}}{2}$. Hence,
using $\sqrt{x+y}\le \sqrt{x}+\sqrt{y}$,
\[
\int_0^{\kappa_n}
\sqrt{
\log 2+\frac{1}{\beta}\log\left(\frac{\lambda}{\eta}\right)
}\,d\eta
\le
C_\beta\, \kappa_n
\sqrt{\log\left(\frac{\lambda}{\kappa_n}\right)},
\]
where the constant $C_\beta$ can be chosen to be of order $\beta^{-1/2}$. For example, one may take
\[
C_\beta
=
C\left[
\sqrt{\log 2+\frac{1}{\beta}}
+
\frac{\sqrt{\pi}}{2\sqrt{\beta}}
\right],
\]
after enlarging $C$ if necessary. 
Therefore,
\[
\E\Big[\sup_{s\in[0,1]} R_n(s)\Big]
\le
C_\beta\,\kappa_n\sqrt{\log(\lambda/\kappa_n)}.
\]
Since $R_n$ is centered Gaussian, $R_n$ and $-R_n$ have the same distribution. Then, 
\[
\E\big[\|R_n\|_{\sup}\big]   
\le 2\E\Big[\sup_{s\in[0,1]}R_n(s)\Big]\le  2C_\beta\,\kappa_n\sqrt{\log(\lambda/\kappa_n)} \to 0~~{\rm as}~ n\to\infty.
\]
Now fix $\epsilon>0$. Then, for all sufficiently large $n$,
\[
\E\big[\|R_n\|_{\sup}\big]\le \frac{\epsilon}{2}.
\]
By the Borell-TIS inequality \citep[cf.][]{adler2007random},
\begin{align} \label{e:Borell-TIS}
\begin{split}
&\pr\left(\|R_n\|_{\sup}>\epsilon\right)
 = \pr\left(\|R_n\|_{\sup} -\E\big[\|R_n\|_{\sup} \big]>\epsilon -\E\big[\|R_n\|_{\sup}\big]\right) \\
&\le \exp\Big(-(\epsilon -\E\|R_n\|_{\sup})^2/(2\kappa_n^2)\Big)
\le  \exp\Big(-\epsilon^2/(8\kappa_n^2)\Big)
\le\exp\left(-\frac{\epsilon^2}{8\lambda^2\rho_n^{2\beta}}\right)
\end{split}
\end{align}
for sufficiently large $n$. 

We next prove part~(ii). For any $x\in\mathfrak{X}_{n,\epsilon}$, by definition we have
$\|x\|_{\sup}\le u$ and $\|\wt x_n-x\|_{\sup}\le\epsilon$. Hence,
\[
\|\wt x_n\|_{\sup}
\le \|x\|_{\sup}+\|\wt x_n-x\|_{\sup}
\le u+\epsilon.
\]
Thus, both $x$ and $\wt x_n$ belong to the sup-norm ball of radius
$u+\epsilon$. By Assumption~\ref{a:G_0}(i),
\[
\|G_0(\wt x_n)-G_0(x)\|_{\sup}
\le C_0(u+\epsilon)^{\nu}\|\wt x_n-x\|_{\sup}^{\alpha}
\le C_0(u+\epsilon)^{\nu}\epsilon^\alpha.
\]
If $u\ge1$ and $0<\epsilon\le1$, then $u+\epsilon\le 2u$, and hence $C_0(u+\epsilon)^{\nu}\epsilon^\alpha
\le C_0 2^{\nu}u^{\nu}\epsilon^\alpha$. Thus the simplified bound holds with $C=C_0 2^{\nu}$.
\end{proof}

We next introduce an auxiliary smoothed and truncated version of the true
operator:
$$
\wt G_{0,n,u}(x) := \ind{\|x\|_{\sup,S_n} \le u}\cdot G_0(\wt x_n),
$$
where $\wt x_n$ is the smoother defined in~\eqref{e:smoother}. This operator
uses the conditional-mean reconstruction $\wt x_n$ in place of $x$, and is
restricted to the truncated region $\{\|x\|_{\sup,S_n}\le u\}$. It will be used
as an intermediate target for constructing a neural-network approximation. Fix $m$ and consider $\wt G_{0,m,u}$. On the set
$\{\|x\|_{\sup,S_m}\le u\}$, we have
$\wt G_{0,m,u}(x)=G_0(\wt x_m)$.
The map $x\mapsto G_0(\wt x_m)$ depends on $x$ only through the finite vector
\[
(x(s):s\in S_m)\in V_{m,u}:=[-u,u]^{|S_m|}.
\]
Since $\wt x_m$ is determined by these grid values and $G_0$ is continuous, this
induces a continuous map from the compact set $V_{m,u}$ into $C[0,1]$.
Therefore, by Theorem~5 of \cite{chen1995universal} (cf.~Section \ref{s:SNO}), for any $\epsilon>0$,
there exist positive integers $p,q$ and constants
$c_i^k,\xi_{ij}^k,\theta_i^k,\zeta_k,w_k\in\mathbb{R}$ such that
\begin{align}\label{e:phin0}
\begin{split}
\sup_{t\in [0,1]} \left|G_0(\widetilde x_m)(t) - G(x)(t)\right|
&< C_0\,(u+\epsilon)^{\nu} \epsilon^{\alpha}\quad
\text{for all $x$ with } (x(s_1),\ldots,x(s_{|S_m|}))\in V_{m,u},\\[2pt]
\text{where}\qquad
G(x)(t)
&= \sum_{k=1}^{p}\sum_{i=1}^{q}
    c_i^k\,\sigma\!\left(\sum_{j=1}^{|S_m|}\xi_{ij}^k\,x(s_j)+\theta_i^k\right)
    \sigma\!\left(w_k t+\zeta_k\right).
\end{split}
\end{align}
With this particular $G$, define $G_{m,u,\epsilon}(x):=\ind{\|x\|_{\sup,S_m}\le u}\,G(x)$. Then, by~\eqref{e:phin0},
\begin{align} \label{e:Gn0_approx}
\left\|\wt G_{0,m,u}- G_{m,u,\epsilon}\right\|_{u,\infty}< C_0 \,(u+\epsilon)^{\nu} \epsilon^{\alpha}.
\end{align}
Note that $G_{m,u,\epsilon}$ need not belong to $\CG_{m,u}$. The
universal approximation theorem guarantees the existence of a finite network
approximating $\wt G_{0,m,u}$, but it does not control whether the required
widths $p,q$ and the number of active input coordinates are within the
architectural bounds $p_m,q_m,r_m$ defining $\CG_{m,u}$. By Assumption \ref{a:sampling}, the classes $\CG_{n,u}$ become increasingly rich as $n$ grows.
Thus, there exists a large enough $n_{m,u,\epsilon}$ such that $G_{m,u,\epsilon}\in\CG_{n,u}$ for all $n\ge n_{m,u,\epsilon}$. 
Define
\begin{align} \label{e:m_n_1}
m_n=m_{n,u,\epsilon} := \max\left\{m: n_{m,u,\epsilon}\le n\right\}. 
\end{align}
Then $G_{m_n,u,\epsilon}\in\CG_{n,u}$ for every $n$, and
$m_{n,u,\epsilon}\to\infty$ as $n\to\infty$ for every fixed $u$ and
$\epsilon$.

\begin{lemma}\label{l:G*G0}
Under Assumption~\ref{a:G_0}(i), for every fixed $u\ge1$ and
$0<\epsilon\le1$, the following bound holds:
\begin{align} \label{e:G*G0}
&\sup_{x\in \mathfrak{X}_{m_n,\epsilon}\cap\mathfrak{X}_{n,\epsilon}}
\|G_{m_n,u,\epsilon}(x)-G_{0,n,u}(x)\|_{\sup} 
 \le   2C_0\,(u+\epsilon)^{\nu} \epsilon^{\alpha}.
\end{align}
\end{lemma}


\begin{proof}
Fix $x\in\mathfrak{X}_{m_n,\epsilon}\cap\mathfrak{X}_{n,\epsilon}$.
By definition of $\mathfrak{X}_{m_n,\epsilon}$, we have
$\|x\|_{\sup}\le u$ and $\|x-\wt x_{m_n}\|_{\sup}\le\epsilon$.
In particular,
$\|x\|_{\sup,S_{m_n}}\le u$,
$\|x\|_{\sup,S_n}\le u$.
Hence, 
\[
G_{0,n,u}(x)=G_0(x),
\qquad
G_{m_n,u,\epsilon}(x)=G(x),
\qquad
\wt G_{0,m_n,u}(x)=G_0(\wt x_{m_n}).
\]
By the triangle inequality,
\begin{align*}
&\|G_{m_n,u,\epsilon}(x)-G_{0,n,u}(x)\|_{\sup}
 =
\|G_{m_n,u,\epsilon}(x)-G_0(x)\|_{\sup} \\
&\le
\|G_{m_n,u,\epsilon}(x)-G_0(\wt x_{m_n})\|_{\sup} 
+
\|G_0(\wt x_{m_n})-G_0(x)\|_{\sup}.
\end{align*}
The construction of $G_{m_n,u,\epsilon}$ and~\eqref{e:Gn0_approx} imply that
\[
\|G_{m_n,u,\epsilon}(x)-G_0(\wt x_{m_n})\|_{\sup}
\le
C_0 (u+\epsilon)^{\nu}\epsilon^\alpha .
\]
Moreover, by Lemma~\ref{l:GaussianKriging}(ii),
\[
\|G_0(\wt x_{m_n})-G_0(x)\|_{\sup}
\le
C_0 (u+\epsilon)^{\nu}\epsilon^\alpha .
\]
Combining the preceding two bounds gives
\[
\|G_{m_n,u,\epsilon}(x)-G_{0,n,u}(x)\|_{\sup}
\le
2C_0 (u+\epsilon)^{\nu}\epsilon^\alpha .
\]
Taking the supremum over
$x\in\mathfrak{X}_{m_n,\epsilon}\cap\mathfrak{X}_{n,\epsilon}$ completes the proof.
\end{proof}

To state the next lemma, first define
$$
r(\ell_1,\ell_2) = \sqrt{w_{\ell_1}w_{\ell_2}}\; 
C_\varepsilon(t_{\ell_1}, t_{\ell_2}), \qquad \ell_1, \ell_2\in [L],
$$
where $C_\varepsilon(s,t)$ is the covariance kernel of the error process $\{\varepsilon(t), t\in [0,1]\}$.

\begin{lemma} \label{l:eigen} Assume that $C_\varepsilon(s,t)$ is bounded on $[0,1]^2$, and define 
\[
B^2_\varepsilon:=\sup_{t\in[0,1]} C_\varepsilon(t,t)<\infty .
\]
Then, the largest eigenvalue of the $L\times L$ weighted covariance matrix $\{r(\ell_1,\ell_2)\}_{\ell_1,\ell_2\in [L]}$ is bounded by $B_\varepsilon^2$ for all $n$.
\end{lemma}

\begin{proof}
Let
\[
\mathbf{R}_{\varepsilon,n}
=
\{r(\ell_1,\ell_2)\}_{\ell_1,\ell_2\in[L_n]},
\qquad {\rm where}~~
r(\ell_1,\ell_2)
=
\sqrt{w_{\ell_1}w_{\ell_2}}\,
C_\varepsilon(t_{\ell_1},t_{\ell_2}).
\]
Then $\mathbf{R}_{\varepsilon,n}$ is the covariance matrix of the random vector
$(\sqrt{w_\ell}\,\varepsilon(t_\ell))_{\ell=1}^{L_n}$, and hence is positive
semidefinite. Therefore,
\[
\lambda_{\max}(\mathbf{R}_{\varepsilon,n})
\le
\tr(\mathbf{R}_{\varepsilon,n})
=
\sum_{\ell=1}^{L_n} r(\ell,\ell)
=
\sum_{\ell=1}^{L_n} w_\ell C_\varepsilon(t_\ell,t_\ell)
\le
B^2_\varepsilon \sum_{\ell=1}^{L_n} w_\ell
=
B^2_\varepsilon.
\]
The last inequality follows from the definition of $B^2_\varepsilon$, and the
last equality follows from $\sum_{\ell=1}^{L_n}w_\ell=1$. Thus the largest
eigenvalue of $\mathbf{R}_{\varepsilon,n}$ is uniformly bounded in $n$.
\end{proof}

\begin{lemma} \label{l:basic_ineq}
Let $Z_{i,\ell}=\varepsilon_i(t_\ell)+e_{i,\ell}$. Then, the following inequality holds:
$$
\|\wh G_{n,u}-G_{0,n,u}\|_n^2 \le S_{n,1} + S_{n,2},
$$
where 
\begin{align*}
S_{n,1} &:= \frac{1}{n} \sum_{i=1}^n \|G_{m_n,u,\epsilon}(X_i) -G_{0,n,u}(X_i)\|_{\sup}^2, \\
S_{n,2} &:= {2\over n}\sum_{i=1}^n \sum_{\ell=1}^{L_n} w_\ell (\wh G_{n,u}(X_i)(t_{\ell})
- G_{m_n,u,\epsilon}(X_i)(t_{\ell})) Z_{i,\ell}.
\end{align*}
\end{lemma}

\begin{proof}
Recall that
\[
\|\widehat G_{n,u}-G_{0,n,u}\|_n^2
= \frac1n\sum_{i=1}^n V_{i,n,u} \sum_{\ell=1}^L 
w_\ell\bigl(\widehat G_{n,u}(X_i)(t_\ell)-G_{0,n,u}(X_i)(t_\ell)\bigr)^2,
\]
and that $\widehat G_{n,u}$ minimizes the truncated empirical loss:
\[
\widehat G_{n,u}
= \argmin_{G\in\mathcal G_{n,u}}
\frac{1}{n}\sum_{i=1}^n V_{i,n,u}\sum_{\ell=1}^{L}
w_\ell \bigl(Y_{i,\ell}-G(X_i)(t_\ell)\bigr)^2 ,
\]
where $V_{i,n,u}=\ind{\|X_i\|_{\sup,S_n}\le u}$.
Hence, for any $G\in\mathcal G_{n,u}$,
\begin{equation}
\label{e:ERM-ineq}
\frac{1}{n}\sum_{i=1}^n V_{i,n,u}\sum_{\ell=1}^L
w_\ell \bigl(Y_{i,\ell}-\widehat G_{n,u}(X_i)(t_\ell)\bigr)^2
\le
\frac{1}{n}\sum_{i=1}^n V_{i,n,u}\sum_{\ell=1}^L
w_\ell \bigl(Y_{i,\ell}-G(X_i)(t_\ell)\bigr)^2 .
\end{equation}
Write
\[
A_{i,\ell}
= \widehat G_{n,u}(X_i)(t_\ell)-G_{0,n,u}(X_i)(t_\ell),\qquad
B_{i,\ell}
= G_{m_n,u,\epsilon}(X_i)(t_\ell)-G_{0,n,u}(X_i)(t_\ell).
\]
On the event $\{V_{i,n,u}=1\}$, we have 
$Y_{i,\ell}=G_{0,n,u}(X_i)(t_\ell)+Z_{i,\ell}$.
Therefore,
\[
Y_{i,\ell}-\widehat G_{n,u}(X_i)(t_\ell)
= Z_{i,\ell}-A_{i,\ell},\qquad
Y_{i,\ell}-
G_{m_n,u,\epsilon}(X_i)(t_\ell)
=Z_{i,\ell}-B_{i,\ell}.
\]
Substituting these into \eqref{e:ERM-ineq} and expanding the squares yields
\[
\frac1n\sum_i V_{i,n,u}\sum_\ell w_\ell
\bigl(Z^2_{i,\ell} -2 Z_{i,\ell} A_{i,\ell}+A_{i,\ell}^2\bigr)
\le
\frac1n\sum_i V_{i,n,u}\sum_\ell w_\ell
\bigl(Z^2_{i,\ell} -2 Z_{i,\ell} B_{i,\ell}+B_{i,\ell}^2\bigr).
\]
The terms $Z^2_{i,\ell}$ are identical on both sides and cancel, giving
\[
\frac1n\sum_i V_{i,n,u} \sum_\ell w_\ell A_{i,\ell}^2
\le
\frac1n\sum_i V_{i,n,u} \sum_\ell w_\ell B_{i,\ell}^2
+
\frac{2}{n}\sum_i V_{i,n,u}\sum_\ell
w_\ell (A_{i,\ell}-B_{i,\ell})Z_{i,\ell}.
\]
Since $V_{i,n,u}\cdot G_0(X_i)=G_{0,n,u}(X_i)$ and 
$V_{i,n,u}\cdot G(X_i)=G(X_i)$ for all $G\in\mathcal G_{n,u}$, we may drop
$V_{i,n,u}$ from both sides. Moreover,
\[
\sum_{\ell=1}^L w_\ell B_{i,\ell}^2
=
\|G_{m_n,u,\epsilon}(X_i)-G_{0,n,u}(X_i)\|_L^2
\le
\|G_{m_n,u,\epsilon}(X_i)-G_{0,n,u}(X_i)\|_{\sup}^2.
\]
Therefore,
\begin{align*}
\|\widehat G_{n,u}-G_{0,n,u}\|_n^2
& \le
\frac1n\sum_{i=1}^n\| 
G_{m_n,u,\epsilon}(X_i)-G_{0,n,u}(X_i)\|_{\sup}^2 \\
& \quad +
\frac{2}{n}\sum_{i=1}^n\sum_{\ell=1}^L
w_\ell (\widehat G_{n,u} 
- G_{m_n,u,\epsilon}
)(X_i)(t_\ell)Z_{i,\ell}.
\end{align*}
This is exactly $S_{n,1}+S_{n,2}$, and the proof is complete.
\end{proof}

\begin{lemma} \label{l:rate_ESn}
Under Assumptions~\ref{a:process}--\ref{a:network},
for each fixed $u\ge1$ and
$0<\epsilon\le1$, and for any $\delta>0$, the quantities $S_{n,1}$ and $S_{n,2}$ satisfy, for all sufficiently large $n$,
\begin{align}\label{e:ESn1}
\begin{split}
\E S_{n,1} 
&\le \Delta_{n,u,\epsilon},
\end{split} \\
\label{e:ESn2}
|\E S_{n,2}|
&\le
2\delta\,\sigma_{\varepsilon,e}
\;+\;
2 B_{\varepsilon,e}
\left( \widehat R_n(\widehat G_{n,u},G_{0,n,u})^{1/2} + \delta \right)
\sqrt{\frac{3\log \CN_{n,u,\delta}+1}{\,n\,}},
\end{align}
where $\sigma_{\varepsilon,e}$ and $B_{\varepsilon,e}$ are defined in \eqref{e:B_varepsilon_e},  
$\Delta_{n,u,\epsilon}$ is defined in~\eqref{e:Delta} and $m_n$ is defined in \eqref{e:m_n_1}. 
\end{lemma}

\begin{proof}
{\sc The bound for $\E S_{n,1}$:} Since the $X_i$ are iid., we have 
\begin{align} \label{e:ESn1_1} 
\begin{split}
\E S_{n,1} & = \E \left[\|G_{m_n,u,\epsilon}(X)-G_{0,n,u}(X)\|_{\sup}^2\right] \\
&= \E \left[\left\|G_{m_n,u,\epsilon}(X)-G_{0,n,u}(X)\right\|_{\sup}^2 \ind{X\in \mathfrak{X}_{m_n,\epsilon}\cap\mathfrak{X}_{n,\epsilon}}\right] \\
&\quad + \E \left[\left\|G_{m_n,u,\epsilon}(X)-G_{0,n,u}(X)\right\|_{\sup}^2 \ind{X\not\in\mathfrak{X}_{m_n,\epsilon}\cap\mathfrak{X}_{n,\epsilon}}\right].
\end{split}
\end{align} 
By Lemma~\ref{l:G*G0}, on the event $X\in \mathfrak{X}_{m_n,\epsilon}\cap\mathfrak{X}_{n,\epsilon}$, we have
\[
\|G_{m_n,u,\epsilon}(X)-G_{0,n,u}(X)\|_{\sup}
\le
2C_0\,(u+\epsilon)^{\nu} \epsilon^\alpha.
\]
Hence, the first term on the right-hand side of~\eqref{e:ESn1_1} is bounded by
\begin{align} \label{e:ESn1_2}
\E \left[\left\|G_{m_n,u,\epsilon}(X)-G_{0,n,u}(X)\right\|_{\sup}^2
\ind{X\in \mathfrak{X}_{m_n,\epsilon}\cap\mathfrak{X}_{n,\epsilon}}\right] 
\le
4C_0^2 (u+\epsilon)^{2\nu}\epsilon^{2\alpha}.
\end{align}
We next bound the second term on the right-hand side of~\eqref{e:ESn1_1} as
\begin{align*}
& \E \left[\left\|G_{m_n,u,\epsilon}(X)-G_{0,n,u}(X)\right\|_{\sup}^2 \ind{X\not\in\mathfrak{X}_{m_n,\epsilon}\cap\mathfrak{X}_{n,\epsilon}}\right] \\
& \le 2 \E \left[\left(\left\|G_{m_n,u,\epsilon}(X)\|_{\sup}^2 + \|G_{0,n,u}(X)\right\|_{\sup}^2\right)\ind{X\not\in\mathfrak{X}_{m_n,\epsilon}\cap\mathfrak{X}_{n,\epsilon}}\right] \\
& \le 2 \E \left[\left(\left\|G_{m_n,u,\epsilon}(X)\|_{\sup}^2
+ 2\|G_{0,u}(X)\right\|_{\sup}^2 + 2\left\|G_{0,n,u}(X)-G_{0,u}(X)\right\|_{\sup}^2 
\right)\ind{X\not\in\mathfrak{X}_{m_n,\epsilon}\cap\mathfrak{X}_{n,\epsilon}}\right] \\
& \le 2\left(\|G_{m_n,u,\epsilon}\|_{u,\infty}^2 + 2 \|G_{0,u}\|_{u,\infty}^2\right) \pr(X\not\in\mathfrak{X}_{m_n,\epsilon}\cap\mathfrak{X}_{n,\epsilon}) + 4 \E \left[\left\|G_{0,n,u}(X)-G_{0,u}(X)\right\|_{\sup}^2\right].
\end{align*}
Since $G_{m_n,u,\epsilon}\in\CG_{n,u}$, we have that $\|G_{m_n,u,\epsilon}\|_{u,\infty} \le B u^{\nu}$. Also, $\|G_{0,u}\|_{u,\infty}\le C_0\, u^{\nu}$ by~\eqref{e:G0u_sup}.
Because $\rho_n\le\rho_{m_n}$, by Lemma~\ref{l:GaussianKriging}(i), 
$$
\pr(X\not\in\mathfrak{X}_{m_n,\epsilon}\cap\mathfrak{X}_{n,\epsilon}) \le \pr(X\not\in\mathfrak{X}_{m_n,\epsilon}) + \pr(X\not\in\mathfrak{X}_{n,\epsilon}) \le 2\exp\left(-\frac{\epsilon^2}{8 \lambda^2\rho_{m_n}^{2\beta}}\right).
$$
Moreover, 
$\E \left[\left\|G_{0,n,u}(X)-G_{0,u}(X)\right\|_{\sup}^2\right] =\E\left[\ind{\|X\|_{\sup,S_{n}} \le u < \|X\|_{\sup}} \|G_0(X)\|_{\sup}^2\right]$.
Combining these bounds gives
\begin{align} \label{e:ESn1_3}
\begin{split}
& \E \left[\left\|G_{m_n,u,\epsilon}(X)-G_{0,n,u}(X)\right\|_{\sup}^2 \ind{X\not\in\mathfrak{X}_{m_n,\epsilon}\cap\mathfrak{X}_{n,\epsilon}}\right] \\
& \le 4 \left(B^2 u^{2 \nu}+2C_0^2u^{2\nu}\right)
 \exp\left(-\frac{\epsilon^2}{8 \lambda^2\rho_{m_n}^{2\beta}}\right)
+ 4\E\left[\ind{\|X\|_{\sup,S_{n}} \le u < \|X\|_{\sup}} \|G_0(X)\|_{\sup}^2\right].
\end{split}
\end{align}
Combining \eqref{e:ESn1_1}--\eqref{e:ESn1_3}, we obtain the asserted inequality~\eqref{e:ESn1}.

\medskip
{\sc The bound for $\E S_{n,2}$:} 
Let $\CN_{n,u,\delta}$ be the $\delta$-covering number of $\CG_{n,u}$ in the operator sup norm $\|\cdot\|_{u,\infty}$, and suppose $\CG_{n,u}\subset \cup_{j=1}^{\CN_{n,u,\delta}} B(G_j,\delta)$ for an appropriate set $\{G_j\}\subset\CG_{n,u}$. Since $\wh G_{n,u}\in\CG_{n,u}$, there exists a random index $j^{\circ}$ such that $\|\wh G_{n,u}-G_{j^{\circ}} \|_{u,\infty} \leq \delta$. 
It follows that
\begin{align*}
\frac{2}{n}\sum_{i=1}^n \sum_{\ell=1}^{L} w_\ell \left|\E \left(\wh G_{n,u}(X_i)(t_{\ell})
- G_{j^{\circ}}(X_i)(t_{\ell})\right)Z_{i,\ell}\right| 
\le \frac{2\delta}{n}\sum_{i=1}^n \sum_{\ell=1}^{L} w_\ell \E |Z_{i,\ell}| \le 2\delta  \sigma_{\varepsilon,e}.
\end{align*}
Using the fact that $\E(G(X_i)(t_\ell)Z_{i,\ell})=0$ for any deterministic, sample-independent operator $G$ that may depend on $n,u,\epsilon$, in particular for $G=G_{m_n,u,\epsilon}$ and
$G=G_{0,n,u}$, we have
\begin{align} \label{e:S2}
\begin{split}
|\E S_{n,2}| 
&= \left|{2\over n}\sum_{i=1}^n \sum_{\ell=1}^{L} w_\ell \E \left(\wh G_{n,u}(X_i)(t_{\ell})- G_{0,n,u}(X_i) (t_{\ell})\right)Z_{i,\ell}\right| \\
& \le 2\delta \sigma_{\varepsilon,e} + \left|{2\over n}\sum_{i=1}^n \sum_{\ell=1}^{L} w_\ell \E \left(G_{j^{\circ}}(X_i)(t_{\ell})-G_{0,n,u}(X_i)(t_{\ell})\right) Z_{i,\ell}\right| \\
& = 2\delta \sigma_{\varepsilon,e} + \frac{2}{\sqrt{n}} \left|\E\left[\frac{\sum_{i=1}^n \sum_{\ell=1}^{L} w_\ell \left(G_{j^{\circ}}(X_i)(t_{\ell})-G_{0,n,u}(X_i)(t_{\ell})\right) Z_{i,\ell}}
{\sqrt{n} \|G_{j^{\circ}}-G_{0,n,u}\|_n} \cdot
\|G_{j^{\circ}}-G_{0,n,u}\|_n \right]\right|\\
&= 2\delta \sigma_{\varepsilon,e} + \frac{2}{\sqrt{n}}|\E\left[\xi_{j^{\circ}}
\|G_{j^{\circ}}-G_{0,n,u}\|_n \right]|\\
& \leq 2 \delta \sigma_{\varepsilon,e} +\frac{2}{\sqrt{n}} \E\left[\left(\left\|\wh G_{n,u} -G_{0,n,u}\right\|_n+\delta\right)\left|\xi_{j^{\circ}}\right|\right],
\end{split}
\end{align}
where 
$$
\xi_j= \frac{\sum_{i=1}^n  \sum_{\ell=1}^{L} \sqrt{w_\ell}\left(G_j(X_i)(t_{\ell})-G_{0,n,u}(X_i)(t_{\ell})\right)\sqrt{w_\ell}\, Z_{i,\ell}}
{\sqrt{n} \|G_j-G_{0,n,u}\|_n}
$$
with the convention that $\xi_j=0$ whenever $\|G_j-G_{0,n,u}\|_n=0$.
Let 
$$
c_{i,\ell} := \sqrt{w_\ell}\left(G_j(X_i)(t_{\ell})-G_{0,n,u}(X_i)(t_{\ell})\right) \quad\mbox{and}\quad \boldc_i = (c_{i,1},\ldots, c_{i,L})^\top.
$$
Then, conditional on $X_1,\dots,X_n$,
\begin{align*}
\var(\xi_j|X_1,\ldots,X_n) & =
\frac{\sum_{i=1}^n \var\left(\sum_{\ell=1}^{L} c_{i,\ell}\sqrt{w_\ell} Z_{i,\ell}\right)} {\sum_{i=1}^n \sum_{\ell=1}^{L} c_{i,\ell}^2} \\
& \le \frac{\sum_{i=1}^n \boldc_i^\top \
\{r(\ell,\ell') +\sigma_e^2 \ind{\ell=\ell'}\}_{\ell,\ell'} \boldc_i
} {\sum_{i=1}^n \boldc_i^\top\boldc_i} \le B_\varepsilon^2 +\sigma_e^2 =: B_{\varepsilon,e}^2,
\end{align*}
by Lemma \ref{l:eigen}. Applying the Cauchy-Schwarz inequality,
\begin{align} \label{e:C6}
\begin{split}
& \E\left[\left(\left\|\wh G_{n,u} -G_{0,n,u}\right\|_n+\delta\right)\left|\xi_{j^{\circ}}\right|\right]\\
& = \E\left[\left\|\wh G_{n,u} - G_{0,n,u}\right\|_n
\left|\xi_{j^{\circ}}\right|\right]
+ \delta \E\left[\left|\xi_{j^{\circ}}\right|\right] \\
& \le \sqrt{\E\left[\left\|\wh G_{n,u} -G_{0,n,u}\right\|_n^2\right]} \sqrt{\E\left[\left|\xi_{j^{\circ}}\right|^2\right]} + \delta \sqrt{\E\left[\left|\xi_{j^{\circ}}\right|^2\right]} \\
& = \left(\widehat{R}_n\left(\wh G_{n,u}, G_{0,n,u}\right)^{1/2}+\delta\right)\sqrt{\E\left[\left|\xi_{j^{\circ}}\right|^2\right]} \\
& \leq \left(\widehat{R}_n\left(\wh G_{n,u}, G_{0,n,u}\right)^{1/2}+\delta\right) B_{\varepsilon,e} \sqrt{3 \log \CN_{n,u,\delta}+1},
\end{split}
\end{align}
where the last inequality follows from
Lemma C.1 of \cite{schmidt2020nonparametric}, which gives
\[
\E\left[|\xi_{j^\circ}|^2\right]
\le
\E\left[
\max_{1\le j\le \CN_{n,u,\delta}}|\xi_j|^2
\right]
\le
B_{\varepsilon,e}^2(3\log \CN_{n,u,\delta}+1).
\]
Combining \eqref{e:S2} and \eqref{e:C6} yields the bound~\eqref{e:ESn2}.
\end{proof}

\section{Proof of Proposition \ref{p:rate_3}}
\label{ss:p:rate_3}

Recall that
\begin{align*}
\CR\left(\wh G_{n,u},G_{0,n,u}\right) & := \E\left[\left\|\wh G_{n,u}(X)-G_{0,n,u}(X)\right\|_{\bbL^2}^2\right] 
= \E\left[\ind{\|X\|_{\sup, S_n} \le u} \left\|\wh G_{n,u}(X)-G_0(X)\right\|_{\bbL^2}^2\right],\\
R\left(\wh G_{n,u},G_{0,n,u}\right) & := \E\left[\ind{\|X\|_{\sup,S_n}\le u}\left\|\wh G_{n,u}(X)-G_0(X)\right\|_L^2\right].
\end{align*}
Therefore,
\begin{align}\label{e:CR_R_diff_start}
&\left|
\CR\left(\wh G_{n,u},G_{0,n,u}\right)
-
R\left(\wh G_{n,u},G_{0,n,u}\right)
\right|  \nonumber\\
&\quad \le
\E\left[
\ind{\|X\|_{\sup,S_n}\le u}
\left|
\left\|\wh G_{n,u}(X)-G_0(X)\right\|_{\bbL^2}^2
-
\left\|\wh G_{n,u}(X)-G_0(X)\right\|_{L}^2
\right|
\right].
\end{align}
We first focus on bounding
$$
\left|\left\|\wh G_{n,u}(X)-G_0(X)\right\|_{\bbL^2}^2-
\left\|\wh G_{n,u}(X)-G_0(X)\right\|_L^2\right|.
$$
It follows that
\begin{align} \label{RS_approx_error}
\begin{split}
&\int_0^1 \left(\wh G_{n,u}(X)(t)-G_0(X)(t)\right)^2 dt
- \sum_{\ell=1}^{L} w_\ell\left(\wh G_{n,u}(X)(t_\ell)-G_0(X)(t_{\ell})\right)^2 \\
&= \sum_{\ell=1}^L \int_{m_{\ell-1}}^{m_\ell} \left[\left(\wh G_{n,u}(X)(t)-G_0(X)(t)\right)^2 -\left(\wh G_{n,u}(X)(t_\ell)-G_0(X)(t_\ell)\right)^2 \right] dt,
\end{split}
\end{align}
where $w_1=\frac{t_1+t_2}2$, $w_\ell=\frac{t_{\ell+1}-t_{\ell-1}}2$ and $w_L=1-\frac{t_{L-1}+t_L}2$, and $m_0=0$, $m_\ell=\frac{t_\ell+t_{\ell+1}}2$ and $m_L=1$.
Equation~(\ref{RS_approx_error}) expresses the quadrature error incurred by approximating the integral by the grid-based Riemann sum.
Define
\[
M_n(X):=
\esssup_{t\in [0,1]}
\left|
\frac{d}{dt}
\left(\wh G_{n,u}(X)(t)-G_0(X)(t)\right)^2
\right|.
\]
The approximation error in
\eqref{RS_approx_error} is bounded in absolute value by
\[
M_n(X) \sum_{\ell=1}^L \int_{m_{\ell-1}}^{m_\ell} |t-t_\ell| \, dt 
\le M_n(X) \left(\sum_{\ell=1}^L (m_\ell - m_{\ell-1}) \right)\gamma_n 
= M_n(X) \gamma_n .
\]
Combining this with \eqref{e:CR_R_diff_start} gives
\begin{equation}
\label{e:quad_Mn}
\left|
\CR\left(\wh G_{n,u},G_{0,n,u}\right)
-
R\left(\wh G_{n,u},G_{0,n,u}\right)
\right|
\le
\gamma_n
\E\left[
\ind{\|X\|_{\sup,S_n}\le u}M_n(X)
\right] 
\le \gamma_n \E M_n(X).
\end{equation}
It remains to bound $\E M_n(X)$. Since
\[
\frac{d}{dt}\left(\wh G_{n,u}(X)(t)-G_0(X)(t)\right)^2
=
2 \left(\wh G_{n,u}(X)(t)-G_0(X)(t)\right)
\left(\wh G_{n,u}(X)'(t)-G_0(X)'(t)\right),
\]
we have
\begin{align} \label{e:G0_diff}
\begin{split}
M_n(X)
&\le 2\esssup_{t\in [0,1]}\left|\wh G_{n,u}(X)(t)\wh G_{n,u}(X)'(t)\right|
+ 2\esssup_{t\in [0,1]}\left|\wh G_{n,u}(X)(t)G_0(X)'(t)\right| \\
&\quad + 2\esssup_{t\in [0,1]}\left|G_0(X)(t)\wh G_{n,u}(X)'(t)\right|
+ 2\esssup_{t\in [0,1]}\left|G_0(X)(t)G_0(X)'(t)\right|.
\end{split}
\end{align}
To bound the derivative uniformly over the class $\CG_{n,u}$, let $\wh G$ denote an arbitrary element of $\CG_{n,u}$. Then $\wh G(x)(t)$ is of the form
\begin {align*}
& \ind{\|x\|_{\sup,S_n}\le u}\sum_{k=1}^{p_n} \sum_{i=1}^{q_n} c_i^k \sigma\left(\sum_{j=1}^{|S_n|} \xi_{i j}^k x_j+\theta_i^k\right)
\sigma\left(w_k  t+\zeta_k\right), 
\end{align*}
where $c_i^k, \xi_{i j}^k, \theta_i^k, \zeta_k, w_k \in [-b_n,b_n]$ and $\sum_{j=1}^{|S_n|} \ind{\xi_{i j}^k\not=0} \le r_n$. Since $\wh G\in \CG_{n,u}$,
\[
\sup_{\|x\|_{\sup}\le u}\, \sup_{t\in[0,1]} |\wh G(x)(t)|= \|\wh G\|_{u,\infty} \le B u^{\nu}.
\]
Moreover, since $\wh G(x)$ depends on $x$ only through the sampled values
$\{x(s_{n,j})\}_{j=1}^{J_n}$, the same bound also applies on the discrete
truncation set $\{\|x\|_{\sup,S_n}\le u\}$.
On the set $\{\|x\|_{\sup,S_n}\le u\}$, the sparsity condition
$\sum_{j=1}^{|S_n|}\ind{\xi_{ij}^k\ne0}\le r_n$ implies
\[
\left|
\sum_{j=1}^{|S_n|}\xi_{ij}^k x(s_{n,j})+\theta_i^k
\right|
\le b_n(r_nu+1).
\]
Hence,
\begin{align}
\label{e:Ghat_diff}
\sup_{\|x\|_{\sup,S_n}\le u}\esssup_{t\in[0,1]}
\left|
{d\over dt}\wh G(x)(t)
\right|
&\le C_\sigma
\sum_{k=1}^{p_n}\sum_{i=1}^{q_n}
|c_i^k|\,|w_k|\,
\sigma_{b_n(r_nu+1)}
\le
C_\sigma p_nq_nb_n^2\sigma_{b_n(r_nu+1)},
\end{align}
where $C_\sigma:=\esssup_{t\in{\mathbb R}}|\sigma'(t)|$. For notational convenience, we write 
\[\|G_0(X)'\|_{\sup} := \esssup_{t\in(0,1)} |G_0(X)'(t)|.\]
Combining \eqref{e:G0_diff} and \eqref{e:Ghat_diff}, we have
\begin{align*}
\E M_n(X)
&\le 2B u^{\nu}C_\sigma p_nq_nb_n^2 \sigma_{b_n(r_n u +1)}
+ 2 B u^{\nu} \E\left[\|G_0(X)'\|_{\sup}\right] \\
&\quad + 2C_\sigma p_nq_nb_n^2 \sigma_{b_n(r_n u +1)} \E\left[\|G_0(X)\|_{\sup}\right] 
+ 2\E\left[\|G_0(X)\|_{\sup}\|G_0(X)'\|_{\sup}\right] \\
&\le 2B u^{\nu}C_\sigma p_nq_nb_n^2 \sigma_{b_n(r_n u +1)}
+ 2B u^{\nu}\E^{1/2}\left[\|G_0(X)'\|_{\sup}^2\right] \\
&\quad + 2C_\sigma p_nq_nb_n^2 \sigma_{b_n(r_n u +1)} \E\left[\|G_0(X)\|_{\sup}\right] 
+ 2\E^{1/2}\left[\|G_0(X)\|_{\sup}^2\right]
\E^{1/2}\left[\|G_0(X)'\|_{\sup}^2\right].
\end{align*}
Thus, by Assumption~\ref{a:G_0}, there exists a constant $C_u<\infty$,
depending on the fixed truncation level $u$ but independent of $n$, such that
\[ \E M_n(X) \le C_u\left\{ 1+p_nq_nb_n^2\sigma_{b_n(r_nu+1)} \right\}.\]
Since $p_n,q_n,b_n\to\infty$, the constant term can be absorbed into the second
term for all sufficiently large $n$. Hence, after possibly enlarging $C_u$,
\[\E M_n(X)\le C_u p_nq_nb_n^2\sigma_{b_n(r_nu+1)}.\]
Combining this bound with \eqref{e:quad_Mn}, we get
\begin{align*}
& \left|\CR\left(\wh G_{n,u},G_{0,n,u}\right) - R\left(\wh G_{n,u},G_{0,n,u}\right) \right|
\le C_u p_nq_nb_n^2\gamma_n \sigma_{b_n(r_n u +1)}.
\end{align*} 
This completes the proof. 
\qed

\section{Proof of Proposition \ref{p:rate_2}}
\label{ss:p:rate_2}

We follow ideas from step (I) of the proof of Lemma 4 in \citet[Supplement]{schmidt2020nonparametric}.
As before, let $\CN_{n,u,\delta}:=\CN(\delta,\CG_{n,u},\|\cdot\|_{u,\infty})$ be the $\delta$-covering number of $\CG_{n,u}$ in $\|\cdot\|_{u,\infty}$, and assume that $\CG_{n,u}\subset \cup_{j=1}^{\CN_{n,u,\delta}} B(G_j,\delta)$. Let $j^{\circ}$ be a random index such that $\|\wh G_{n,u}-G_{j^{\circ}} \|_{u,\infty} \leq \delta$. Since both $\wh G_{n,u}$ and $G_{j^\circ}$ belong to $\CG_{n,u}$ and hence contain the
factor $\ind{\|x\|_{\sup,S_n}\le u}$, the covering bound implies that, for any $x$,
\[
\|\wh G_{n,u}(x)-G_{j^\circ}(x)\|_L
\le
\|\wh G_{n,u}-G_{j^\circ}\|_{u,\infty}
\le \delta .
\]
Let $\{X_i^{\prime}\}_{i=1}^n$ be an iid. sample from the distribution of $X$, independent of $\{X_i\}_{i=1}^n$. For each $j=1,\ldots,\CN_{n,u,\delta}$, define
$$
g_{j}\left(X_i, X_i^{\prime}\right):= \left\|G_{j}(X_i')-G_{0,n,u}(X_i')\right\|_L^2 -\left\|G_{j}(X_i)-G_{0,n,u}(X_i)\right\|_L^2.
$$
Then,
\begin{align*}
& \left|R\left(\wh G_{n,u}, G_{0,n,u}\right)-\widehat{R}_n\left(\widehat{G}_{n,u}, G_{0,n,u}\right)\right| \\
& =\left|\E\left[\frac{1}{n} \sum_{i=1}^n \left(\left\|\wh G_{n,u}(X_i')-G_{0,n,u}(X_i')\right\|_L^2-\left\|\wh G_{n,u}(X_i)-G_{0,n,u}(X_i)\right\|_L^2\right)\right]\right| \\
& \leq \E\left[\left|\frac{1}{n} \sum_{i=1}^n g_{j^\circ}\left(X_i, X_i^{\prime}\right)\right|\right]\\
& \quad + \left|\E\left[\frac{1}{n} \sum_{i=1}^n \left(\left\|\wh G_{n,u}(X_i')-G_{0,n,u}(X_i')\right\|_L^2- \left\|G_{j^\circ}(X_i')-G_{0,n,u}(X_i')\right\|_L^2\right)\right]\right| \\
& \quad + \left|\E\left[\frac{1}{n} \sum_{i=1}^n \left(\left\|\wh G_{n,u}(X_i)-G_{0,n,u}(X_i)\right\|_L^2-\left\|G_{j^\circ}(X_i)-G_{0,n,u}(X_i)\right\|_L^2\right)\right]\right| .
\end{align*}
Applying the elementary inequality
\begin{align} \label{e:triangle}
\left|\|x+y\|^2-\|x\|^2\right|=\left|2\langle x,y\rangle + \|y\|^2\right|\le 2\|x\|\cdot \|y\| + \|y\|^2,
\end{align}
with $x = G_{j^\circ}(X_i)-G_{0,n,u}(X_i)$ and 
$y = \widehat G_{n,u}(X_i)-G_{j^\circ}(X_i)$,
we obtain
\begin{align*}
& \left|\big\|\wh G_{n,u}(X_i)-G_{0,n,u}(X_i)\big\|_L^2-\left\|G_{j^\circ}(X_i)-G_{0,n,u}(X_i)\right\|_L^2\right| \\
& \le 2 \left\|G_{j^\circ}(X_i)-G_{0,n,u}(X_i)\right\|_L \cdot \big\|\wh G_{n,u}(X_i)-G_{j^\circ}(X_i)\big\|_L + \big\|\wh G_{n,u}(X_i)-G_{j^\circ}(X_i)\big\|_L^2 \\
&\le 2\delta \left\|G_{j^\circ}(X_i)-G_{0,n,u}(X_i)\right\|_L + \delta^2 \\
& \le 2\delta \left\{B u^{\nu}+C_0u^{\nu}\right\} + \delta^2 
  \le B_u\delta+\delta^2,
\end{align*}
where $B_u = 2 (B+C_0) u^{\nu}$. The last inequality follows from the facts that, for any $G\in\CG_{n,u}$,
$\|G\|_{u,\infty}\le B u^{\nu}$,
and that
$\|G_{0,n,u}\|_{u,\infty}\le C_0u^{\nu}$ by~\eqref{e:G0u_sup}. 
Thus,
\begin{align} \label{e:RRhat_diff_2}
\begin{split}
\left|R\left(\wh G_{n,u}, G_{0,n,u}\right)-\widehat{R}_n\left(\widehat{G}_{n,u}, G_{0,n,u}\right)\right| \le \E\left[\Big|\frac{1}{n} \sum_{i=1}^n g_{j^\circ}\left(X_i, X_i^{\prime}\right)\Big|\right] + 2B_u\delta + 2\delta^2.
\end{split}
\end{align}
Next, define the local radius
\begin{align} \label{e:r_j}
r_j:=\sqrt{n^{-1} \log \CN_{n,u,\delta}} \vee \E^{1 / 2}\left[\left\|G_j(X)-G_{0,n,u}(X)\right\|_L^2\right], \qquad j=1,\ldots,\CN_{n,u,\delta},
\end{align}
where $X$ is an independent copy of the predictor and independent of the data $\{(X_i,Y_i)\}_{i=1}^n$.
Note that $r_j$, which is only used in this proof, is unrelated to the constant $r_n$ used to define $\CG_{n,u}$. 
For the data-dependent random index $j^{\circ}$, define
\begin{align*}
r_{j^{\circ}} & :=\sqrt{n^{-1} \log \CN_{n,u,\delta}} \vee \E^{1 / 2}\left[\left\|G_{j^\circ}(X)-G_{0,n,u}(X)\right\|_L^2 \big|\left(X_i, Y_i\right), i \in [n]\right] \\
& \leq \sqrt{n^{-1} \log \CN_{n,u,\delta}}+\E^{1 / 2}\left[\left\|\wh G_{n,u}(X)-G_{0,n,u}(X)\right\|_L^2 \Big|\left(X_i, Y_i\right), i\in [n]\right]+\delta.
\end{align*}
Now consider
$|\sum_{i=1}^n g_{j^\circ}(X_i, X_i')|$.
By inserting the normalizing factor $r_{j^\circ}B_u$, we obtain
\begin{align*}
\left|\sum_{i=1}^n g_{j^\circ}(X_i,X_i')\right|
&=
\frac{\left|\sum_{i=1}^n g_{j^\circ}(X_i,X_i')\right|}{r_{j^\circ}B_u}\,r_{j^\circ}B_u \\
&\le
\max_{1\le j\le \CN_{n,u,\delta}}
\frac{\left|\sum_{i=1}^n g_j(X_i,X_i')\right|}{r_jB_u}
\left(\sqrt{\frac{\log \CN_{n,u,\delta}}{n}}+U+\delta \right)B_u \\
&=
T\left(\sqrt{\frac{\log \CN_{n,u,\delta}}{n}}+U+\delta \right)B_u,
\end{align*}
where 
\begin{align*}
T& :=\max_{1\le j\le \CN_{n,u,\delta}} \left|\sum_{i=1}^n g_j\left(X_i, X_i^{\prime}\right) /\left(r_j B_u\right)\right|, \\
U & :=\E^{1 / 2}\left[\left\|\wh G_{n,u}(X)-G_{0,n,u}(X)\right\|_L^2 \Big|\left(X_i, Y_i\right), i \in [n]\right].
\end{align*}
By \eqref{e:RRhat_diff_2}, the preceding bound on
$\left|\sum_{i=1}^n g_{j^\circ}(X_i,X_i')\right|$, and the
Cauchy-Schwarz inequality,  together with 
$\E\left[U^2\right]=R\left(\wh G_{n,u}, G_{0,n,u}\right)$,
we obtain
\begin{align} \label{e:C2}
\begin{split}
& \left|R\left(\wh G_{n,u}, G_{0,n,u}\right)-\widehat{R}_n\left(\wh G_{n,u}, G_{0,n,u}\right)\right| \\
& \leq \frac{B_u}{n}\, \E\left[TU\right]+\frac{B_u}{n}\left(\sqrt{\frac{\log \CN_{n,u,\delta}}{n}}+\delta\right) \E[T] 
  + 2B_u\delta + 2\delta^2 \\
& \leq \frac{B_u}{n} R\left(\wh G_{n,u}, G_{0,n,u}\right)^{1 / 2} \E^{1 / 2}\left[T^2\right]+\frac{B_u}{n}\left(\sqrt{\frac{\log \CN_{n,u,\delta}}{n}}+\delta\right) \E[T] 
  + 2B_u\delta + 2\delta^2.
\end{split}
\end{align}
Observe that $\E\left[g_j\left(X_i, X_i^{\prime}\right)\right]=0$. Also, for any $\|x\|_{\sup}\le u$,
\[\left\|G_j\left(x\right)-G_{0,n,u}\left(x\right)\right\|_{\sup}
\le B u^{\nu}+C_0u^{\nu} \le (B+C_0)u^{\nu} =\frac{B_u}{2}.\] 
Thus,
\begin{align} \label{e:T_max}
\frac{\left|g_j\left(X_i, X_i^{\prime}\right)\right|}{r_j B_u} 
\le\frac{\left\|G_j(X_i')-G_{0,n,u}(X_i')\right\|_L^2+
\left\|G_j(X_i)-G_{0,n,u}(X_i)\right\|_L^2}{r_jB_u} 
\leq \frac{B_u^2/2}{r_jB_u} = \frac{B_u}{2r_j},
\end{align}
and
\begin{align} \label{e:T_var}
\begin{split}
\var\left(\frac{g_j\left(X_i, X_i^{\prime}\right)}{r_jB_u}\right) 
& =\frac{2}{(r_jB_u)^2} \var\left(\left\|G_j\left(X_i\right)-G_{0,n,u}\left(X_i\right)\right\|_L^2\right) \\
& \leq\frac{2}{(r_jB_u)^2} \E\left[\left\|G_j\left(X_i\right)-G_{0,n,u}\left(X_i\right)\right\|_L^4\right] \\
& \leq \frac{2}{(r_jB_u)^2}\cdot \frac{B_u^2}{4} \, \E\left[\left\|G_j\left(X_i\right)-G_{0,n,u}\left(X_i\right)\right\|_L^2\right]\le \frac{1}{2},
\end{split}
\end{align}
where the last step follows from the definition of $r_j$ in~\eqref{e:r_j}.
Applying a union bound argument together with Bernstein's inequality using \eqref{e:T_max} and \eqref{e:T_var}, we obtain
\begin{align*}
\pr(T \geq t) & \leq 1 \wedge 2 \CN_{n,u,\delta} \max_{1\le j\le \CN_{n,u,\delta}} \exp\left(-\frac{t^2}{n+B_ut/(3r_j)}\right)  \\
& = 1 \wedge 2 \CN_{n,u,\delta} \exp\left(-\frac{t^2}{n+B_ut/(3r_\star)}\right),
\end{align*}
where $r_\star$ is the minimal local radius defined by
\begin{equation}\label{e:rstar}
\begin{split}
r_\star
&:= \min_{1\le j\le \CN_{n,u,\delta}} r_j 
   = \min_{1\le j\le \CN_{n,u,\delta}} \!\left\{
       \sqrt{n^{-1}\log \CN_{n,u,\delta}}
       \;\vee\;
       \E^{1/2}\!\left[
         \|G_j(X)-G_{0,n,u}(X)\|_L^2
       \right]
     \right\} \\
&~= \sqrt{n^{-1}\log \CN_{n,u,\delta}}
   \;\vee\;
   \min_{1\le j\le \CN_{n,u,\delta}} \E^{1/2}\!\left[
     \|G_j(X)-G_{0,n,u}(X)\|_L^2
   \right].
\end{split}
\end{equation}
Without loss of generality, we may choose the finite covering set so that it contains the approximating
operator $G_{m_n,u,\epsilon}$; this increases the covering number by at most one and does not affect the entropy bound. 
In that case, by \eqref{e:ESn1} together with the identity \eqref{e:ESn1_1} from Lemma~\ref{l:rate_ESn},
\begin{align*}
\E\left[\left\|G_{m_n,u,\epsilon}(X)-G_{0,n,u}(X)\right\|_L^2\right] 
\le
\E\left[\left\|G_{m_n,u,\epsilon}(X)-G_{0,n,u}(X)\right\|_{\sup}^2\right]
\le \Delta_{n,u,\epsilon},
\end{align*}
where $\Delta_{n,u,\epsilon}$ is given by~\eqref{e:Delta}.
Thus, we have
\begin{align} \label{e:rstar_bound} 
\sqrt{n^{-1} \log \CN_{n,u,\delta}} \le r_\star \le \sqrt{n^{-1} \log \CN_{n,u,\delta}} \vee \sqrt{\Delta_{n,u,\epsilon}}
=:\Omega_{n,u,\delta,\epsilon}.
\end{align}
Focus on the exponential bound:
\begin{align*}
\pr(T \geq t) & \leq  1 \wedge 2 \CN_{n,u,\delta}\,
\exp\left(-\frac{t^2}{n+B_ut/(3r_\star)}\right).
\end{align*}
It follows that if $t\ge r_\star n$,
\[
\frac{t^2}{n+B_ut/(3r_\star)}
\ge
\left(1+\frac{B_u}{3}\right)^{-1}r_\star t
=: c r_\star t,
\]
where $c=(1+B_u/3)^{-1} < 1$.
Letting $t_0=c^{-1} r_\star n$, 
we have
\begin{align*}
\E[T] & =\int_0^{\infty} \pr(T \geq t) d t 
   \le t_0 +\int_{t_0}^{\infty} \pr(T \geq t) d t \\
& \leq c^{-1} r_\star n +\int_{c^{-1}r_\star n}^\infty 2 \CN_{n,u,\delta} \exp \left(-cr_\star t\right) d t \\
&= c^{-1} r_\star n + {2 \CN_{n,u,\delta}\over cr_\star}\int_{r_\star^2n}^\infty  \exp(-s) ds  \\
&= c^{-1} r_\star n + {2 \CN_{n,u,\delta}\over cr_\star} \exp(-r_\star^2n)   \\
&\le c^{-1} r_\star n + {2 \CN_{n,u,\delta}\over cr_\star} \exp(-\log \CN_{n,u,\delta}) \quad \mbox{by \eqref{e:r_j}} \\
&= c^{-1} \left(r_\star n + {2 \over r_\star}\right)  
   \le c^{-1} r_\star n \left(1 + {2 \over \log \CN_{n,u,\delta}}\right) \\
& = (1+ B_u/3) r_\star n \left(1 + {2 \over \log \CN_{n,u,\delta}}\right) 
  \le C_1 B_u r_\star n
\end{align*}
for some $C_1<\infty$ using $\log \CN_{n,u,\delta}\ge 1$ and the fact that
$B_u$ is bounded away from zero for fixed $u$.
Applying this bound to the term that involves $\E [T]$ on the right-hand side of \eqref{e:C2}, we obtain
\begin{align} \label{e:T1}
\frac{B_u}{n}\left(\sqrt{\frac{\log \CN_{n,u,\delta}}{n}}+\delta\right) \E[T] \le 
C_1B_u^2\left(\sqrt{\frac{\log \CN_{n,u,\delta}}{n}}+\delta\right) r_\star.
\end{align}
To bound the second moment, we proceed in a similar manner:
\begin{align*}
\E\left[T^2\right] & =\int_0^{\infty} \pr\left(T^2 \geq s\right) d s=\int_0^{\infty} \pr(T \geq \sqrt{s}) d s \\
& \le t_0^2 + \int_{t_0^2}^\infty
2 \CN_{n,u,\delta} \exp \left(-cr_\star \sqrt{s}\right) d s \\
&= (c^{-1}r_\star n)^2 + 4 \CN_{n,u,\delta}  
(t_0 cr_\star+1) e^{-t_0 cr_\star} / (cr_\star)^2 \\
&= (c^{-1}r_\star n)^2 + 4 \CN_{n,u,\delta}  
(r_\star^2 n+1) e^{-r_\star^2 n} / (cr_\star)^2 \\
&\le (c^{-1}r_\star n)^2 + 4   
(r_\star^2 n+1) / (cr_\star)^2 \\
& = (c^{-1}r_\star n)^2 + 4 c^{-2} n + 4c^{-2} r_\star^{-2} \\
& \le  (c^{-1}r_\star n)^2 + 4 c^{-2} n + 4c^{-2} {n\over \log \CN_{n,u,\delta}} \\
& \le C_2 B_u^2 \left((r_\star n)^2+n\right)
\end{align*}
for some $C_2<\infty$.
Applying this to the term that involves $\E [T^2]$ on the right-hand side of \eqref{e:C2}, we get
\begin{align} \label{e:T2} 
\frac{B_u}{n} R\left(\wh G_{n,u}, G_{0,n,u}\right)^{1 / 2} \E^{1 / 2}\left[T^2\right] 
\le C_2^{1/2} \frac{B_u^2}{n} R\left(\wh G_{n,u}, G_{0,n,u}\right)^{1 / 2} \left((r_\star n)^2+n\right)^{1/2}.
\end{align}
By \eqref{e:C2}, \eqref{e:T1} and \eqref{e:T2},
\begin{align*} 
& \left|R\left(\wh G_{n,u}, G_{0,n,u}\right)-\widehat{R}_n\left(\wh G_{n,u}, G_{0,n,u}\right)\right| \\
& \le \frac{C_2^{1/2}B_u^2}{n} R\left(\wh G_{n,u}, G_{0,n,u} \right)^{1 / 2} \left((r_\star n)^2+n\right)^{1/2} + C_1B_u^2\left(\sqrt{\frac{\log \CN_{n,u,\delta}}{n}}+\delta\right) r_\star \\
& \quad + 2B_u\delta + 2\delta^2.
\end{align*}
Express this inequality as $|a-b| \leq 2 \sqrt{a} c+d$, with
\begin{align*}
a &:= R\left(\wh G_{n,u}, G_{0,n,u}\right) \\
b &:= \widehat{R}_n\left(\wh G_{n,u}, G_{0,n,u}\right) \\
c &:= \frac{C_2^{1/2}B_u^2}{2n} \left((r_\star n)^2+n\right)^{1/2} \\
d &:= C_1B_u^2\left(\sqrt{\frac{\log \CN_{n,u,\delta}}{n}}+\delta\right) r_\star + 2B_u\delta + 2\delta^2.
\end{align*}
The preceding bound implies that
\begin{align*} 
a \le b + 2 \sqrt{a} c+d,~
\mbox{ and hence }
(\sqrt{a}-c)^2 = a - 2 \sqrt{a} c + c^2 \le b+d + c^2.
\end{align*}
Therefore, 
$$
\sqrt{a} \le \sqrt{b+d + c^2} + c
\quad\mbox{or}\quad a \le 2(b+d+2c^2).
$$
Substituting the definitions of $a,b,c,d$ gives
\begin{align*}
R\left(\wh G_{n,u}, G_{0,n,u}\right) \le & 2 \widehat{R}_n\left(\wh G_{n,u}, G_{0,n,u}\right)  + \frac{C_2B_u^4}{n^2} \left((r_\star n)^2+n\right) \\
& + 2 C_1B_u^2\left(\sqrt{\frac{\log \CN_{n,u,\delta}}{n}}+\delta\right) r_\star
  + 4(B_u\delta + \delta^2).
\end{align*}
Using \eqref{e:rstar_bound}, we have $r_\star \le \Omega_{n,u,\delta,\epsilon}$, and therefore the bound~\eqref{e:R-Rhat_1} follows. 
\qed

\section{Proof of Proposition \ref{p:hatRn}}
\label{ss:p:hatRn}


By Lemma~\ref{l:basic_ineq}, we have
\begin{align} \label{e:errors}
\wh R_n(\wh G_{n,u},G_{0,n,u}) \le \E S_{n,1} + |\E S_{n,2}|.
\end{align}
Lemma~\ref{l:rate_ESn} deals with the two expectations on the right-hand side of \eqref{e:errors}.
By \eqref{e:errors} and \eqref{e:ESn2}, 
\begin{align*}
\widehat{R}_n\left(\wh G_{n,u}, G_{0,n,u}\right)
& \le \E S_{n,1}
+ 2\delta \sigma_{\varepsilon,e} + 2 B_{\varepsilon,e}\left(\widehat{R}_n\left(\wh G_{n,u}, G_{0,n,u}\right)^{1/2}+\delta\right) \sqrt{3 \log \CN_{n,u,\delta}+1\over n}.
\end{align*}
Denote this inequality as
$a^2 \le b + (a+c)d$,
with
$$
a:=\sqrt{\widehat{R}_n\left(\wh G_{n,u}, G_{0,n,u}\right)}, \ b:=\E S_{n,1}+2\delta \sigma_{\varepsilon,e},\  c:=\delta,\ d= 2 B_{\varepsilon,e}\sqrt{3 \log \CN_{n,u,\delta}+1\over n}.
$$
After completing the square, this implies
$(a-d/2)^2 \le b+ cd + d^2/4$.
Thus,
$$
a \le d/2 + \sqrt{b+ cd + d^2/4},\quad {\rm then}~
a^2 \le d^2/2 + 2(b+ cd + d^2/4) = 2(b+cd) + d^2.
$$
In other words,
\begin{align*} 
& \widehat{R}_n\left(\wh G_{n,u}, G_{0,n,u}\right) \\
&\le 2\left(\E S_{n,1}+2\delta \sigma_{\varepsilon,e}+ 2\delta  
B_{\varepsilon,e}\sqrt{3 \log \CN_{n,u,\delta}+1\over n}\right) + \left(2 B_{\varepsilon,e}\sqrt{3 \log \CN_{n,u,\delta}+1\over n}\right)^2.
\end{align*}
Finally, applying the bound for $\E S_{n,1}$ in~\eqref{e:ESn1} gives the asserted bound in Proposition~\ref{p:hatRn}.
\qed

%
%
\section{Supplementary Details for the Simulation Study}
\label{s:FFBNN_supp}

\subsection{FFBNN architecture}\label{ss:ffbnn_architecture}

We compared the performance of SNO and FFDNN in Section \ref{s:simulation}. Here, we include the same comparison of SNO and FFBNN.

The FFBNN benchmark uses the same preprocessing and network architecture as FFDNN.
Thus, the predictor and response curves are linearly interpolated onto regular grids, yielding
$x\in\mathbb{R}^{m_x}$ and $\widehat{y}\in\mathbb{R}^{m_y}$ with $m_x=100$ and $m_y=75$.
The network has two hidden layers, each consisting of four parallel channels of width $S=30$.
For $k=1,\ldots,4$,
\[
h_1^{(k)}=\operatorname{ReLU}\bigl(m_x^{-1}W_1^{(k)}x+b_1^{(k)}\bigr),
\qquad
h_2^{(k)}=\operatorname{ReLU}\bigl(S^{-1}\sum_{j=1}^4W_2^{(j,k)}h_1^{(j)}+b_2^{(k)}\bigr),
\]
and the output is
\[
\widehat{y}=b_3+S^{-1}\sum_{k=1}^4W_3^{(k)}h_2^{(k)}.
\]
The only difference from FFDNN is the representation of the weight matrices.
In FFBNN, each weight matrix is represented using tensor-product cubic B-spline bases as
\[
W_1^{(k)}=B_{1,\mathrm{out}}C_1^{(k)}B_{1,\mathrm{in}}^\top,\qquad
W_2^{(j,k)}=B_{2,\mathrm{out}}C_2^{(j,k)}B_{2,\mathrm{in}}^\top,\qquad
W_3^{(k)}=B_{3,\mathrm{out}}C_3^{(k)}B_{3,\mathrm{in}}^\top.
\]
Here, $B_{\ell,\mathrm{in}}$ and $B_{\ell,\mathrm{out}}$ are fixed B-spline design matrices, while $C_\ell$ denotes the trainable coefficient matrix.
The input- and output-side numbers of B-spline basis functions are $(5,6)$, $(7,8)$, and $(9,10)$ for the first hidden, second hidden, and output layers, respectively.
Thus, FFBNN can be viewed as a lower-dimensional parameterization of the weight matrices used in FFDNN.
(The bias terms are defined in the same way as in FFDNN.)

\subsection{Comparison with FFBNN}\label{ss:ffbnn_results}

Table~\ref{tab:simulation_irmse_ffbnn} compares the prediction performance of FFBNN and SNO across the eight data-generating models, reporting the mean and standard deviation of the run-wise iRMSE values over the 50 simulation runs. 
Figures~\ref{fig:simulation_histogram_ffbnn} and \ref{fig:simulation_boxplot_ffbnn} show the corresponding histograms and boxplots. As in the comparison with FFDNN in the main article, SNO generally yields smaller and less variable iRMSE values than FFBNN across the eight settings.

Table~\ref{tab:simulation_time_ffbnn} reports the corresponding computational costs. 
As in the comparison with FFDNN in the main article, SNO generally requires fewer training epochs and a shorter total training time than FFBNN. FFBNN also requires more time per epoch in the present implementation. These timing results should be regarded as approximate because wall-clock time can depend on the computational environment, system load, and implementation details.

\begin{table}[h!]
\centering
\caption{Comparison of prediction performance for FFBNN and SNO across eight data-generating models. Mean and Std denote the mean and standard deviation of the iRMSE values over the 50 simulation runs, respectively.}
\label{tab:simulation_irmse_ffbnn}

{\small
\begin{tabular}{lcccc}
\toprule
\multirow{2}{*}{Data-generating model}
& \multicolumn{2}{c}{FFBNN}
& \multicolumn{2}{c}{SNO} \\
\cmidrule(lr){2-3} \cmidrule(lr){4-5}
& Mean & Std & Mean & Std \\
\midrule
Linear            & 0.1615 & 0.1035 & 0.0899 & 0.0090 \\
CAM               & 0.3424 & 0.2213 & 0.1393 & 0.0281 \\
Single-index      & 1.4220 & 0.1269 & 0.1574 & 0.0351 \\
Multiple-index    & 1.3490 & 0.3622 & 0.6164 & 0.2675 \\
Quadratic         & 0.6338 & 0.0637 & 0.1769 & 0.0380 \\
Complex quadratic & 0.8114 & 0.3273 & 0.4782 & 0.2857 \\
Dynamical 1       & 0.2815 & 0.0503 & 0.1156 & 0.0074 \\
Dynamical 2       & 0.3292 & 0.0978 & 0.0984 & 0.0052 \\
\bottomrule
\end{tabular}
}
\end{table}

\begin{table}[h!]
\centering
\caption{Comparison of computational costs for FFBNN and SNO across eight data-generating models. The reported values are averages over the 50 simulation runs.}
\label{tab:simulation_time_ffbnn}

{\small
\begin{tabular}{lcccccc}
\toprule
\multirow{2}{*}{Data-generating model}
& \multicolumn{2}{c}{Epochs}
& \multicolumn{2}{c}{Total time}
& \multicolumn{2}{c}{Time/epoch} \\
\cmidrule(lr){2-3} \cmidrule(lr){4-5} \cmidrule(lr){6-7}
& FFBNN & SNO & FFBNN & SNO & FFBNN & SNO \\
\midrule
Linear            & 500.0 & 110.3 & 68.1 & 10.3 & 0.136 & 0.094 \\
CAM               & 484.4 & 128.0 & 66.2 & 11.9 & 0.137 & 0.093 \\
Single-index      & 190.9 & 150.2 & 26.4 & 13.8 & 0.139 & 0.092 \\
Multiple-index    & 321.1 & 188.1 & 43.8 & 17.2 & 0.137 & 0.092 \\
Quadratic         & 494.2 & 170.1 & 67.7 & 15.7 & 0.137 & 0.092 \\
Complex quadratic & 462.7 & 140.5 & 63.4 & 13.0 & 0.137 & 0.093 \\
Dynamical 1       & 499.3 & 112.4 & 67.9 & 10.4 & 0.136 & 0.093 \\
Dynamical 2       & 500.0 & 97.8  & 68.3 & 9.2  & 0.137 & 0.094 \\
\bottomrule
\end{tabular}
}
\end{table}

\begin{figure}[h!]
\centering
\includegraphics[width=\textwidth]{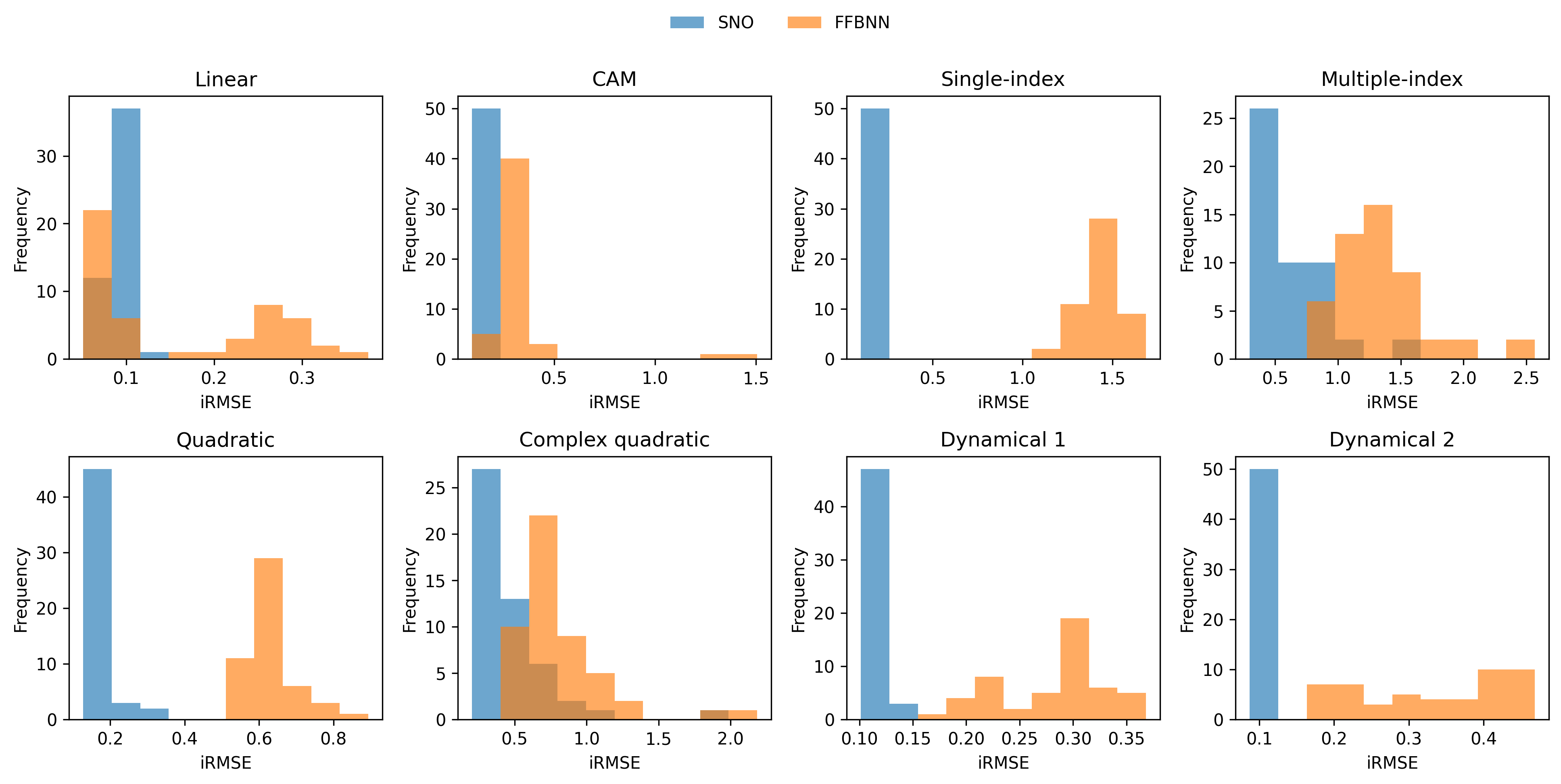}
\caption{Histograms comparing the run-wise iRMSE values of SNO and FFBNN over the 50 replicate runs.}
\label{fig:simulation_histogram_ffbnn}
\end{figure}

\begin{figure}[h!]
\centering
\includegraphics[width=\textwidth]{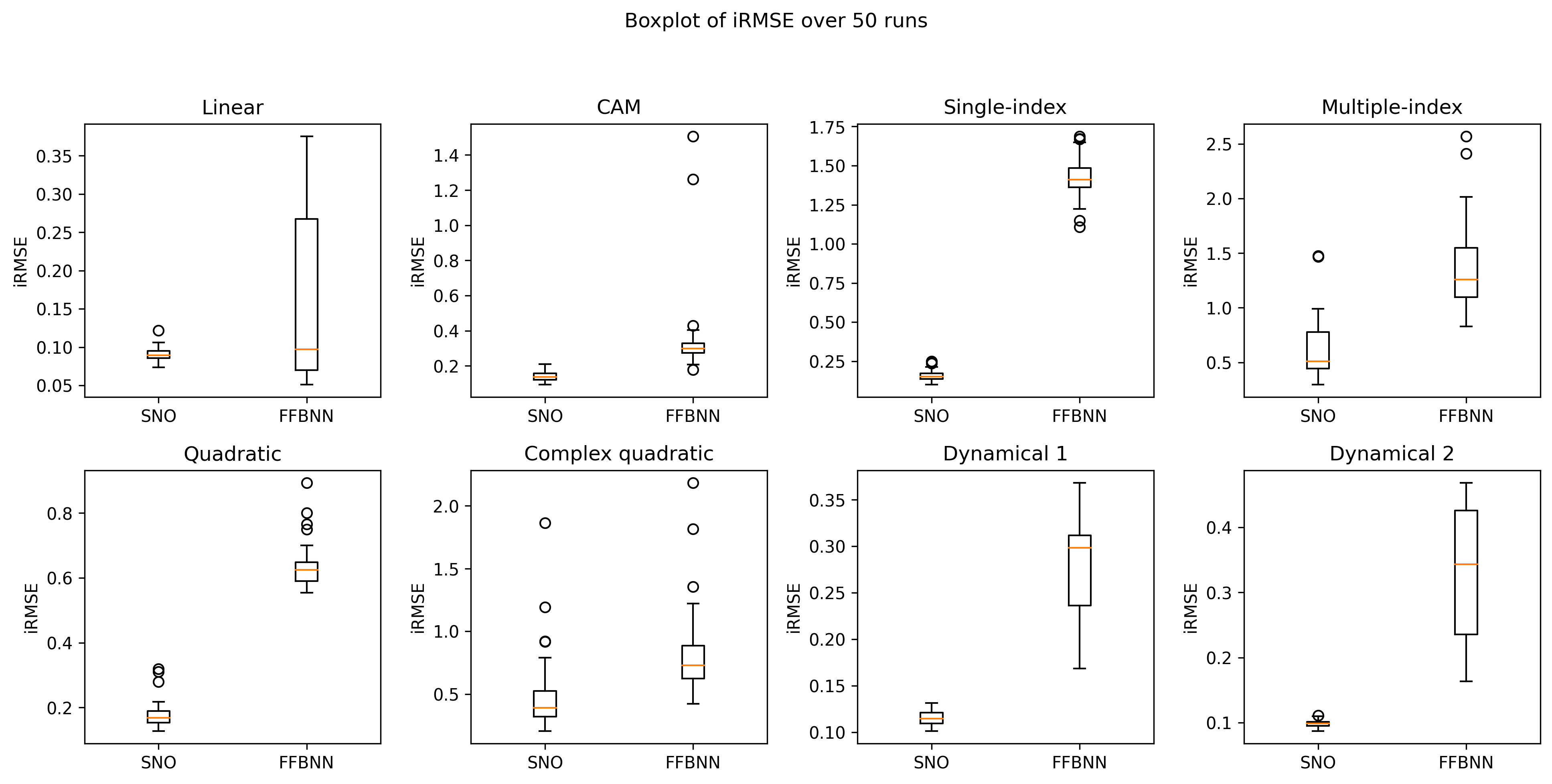}
\caption{Boxplots comparing the run-wise iRMSE values of SNO and FFBNN over the 50 replicate runs.}
\label{fig:simulation_boxplot_ffbnn}
\end{figure}

\subsection{Optimization settings}

Table~\ref{tab:simulation_training_settings} summarizes the optimization settings used for training.
The same optimization settings were used for all three methods, except for the training loss. 

\begin{table}[ht]
\centering
\caption{Optimization settings for the simulation study.}
\label{tab:simulation_training_settings}
\begin{tabular}{lcc}
\hline
Setting & SNO & FFDNN/FFBNN \\
\hline
Optimizer & AdamW & AdamW \\
Learning rate & $5\times 10^{-4}$ & $5\times 10^{-4}$ \\
Weight decay & $10^{-6}$ & $10^{-6}$ \\
Mini-batch size & $32$ & $32$ \\
Maximum number of epochs & $500$ & $500$ \\
Early-stopping patience & $50$ & $50$ \\
Early-stopping criterion & Validation loss & Validation loss \\
Training loss
& Weighted integrated MSE
& MSE on interpolated grid \\
Model retained
& Minimum validation loss
& Minimum validation loss \\
\hline
\end{tabular}
\end{table}

%
%
\section{Supplementary Details for Analysis of Argo Data}
\label{s:Argo_supp}

Table~\ref{tab:argo_training_settings} summarizes the optimization settings used for the BGC-Argo data analysis.

\begin{table}[ht]
\centering
\caption{Optimization settings for the BGC-Argo data analysis.}
\label{tab:argo_training_settings}
\begin{tabular}{ll}
\hline
Setting & Value \\
\hline
Optimizer & AdamW \\
Learning rate & $10^{-3}$ \\
Weight decay & $10^{-5}$ \\
Mini-batch size & $32$ \\
Maximum number of epochs & $300$ \\
Early-stopping patience & $30$ \\
Early-stopping criterion & Validation loss \\
Training loss & Weighted integrated mean squared error \\
Gradient clipping & Maximum norm $5$ \\
\hline
\end{tabular}
\end{table}
\pagebreak


\clearpage

\section{Summary of notation} \label{s:notation}

This section summarizes the main notation used throughout the paper. The notation is grouped according to the model, the separable neural-operator representation, the theoretical sieve class, the sampling design, and the numerical implementation. 

\begin{longtable}{p{0.22\textwidth}p{0.72\textwidth}}
\hline
\textbf{Notation} & \textbf{Meaning} \\
\hline
\endfirsthead

\hline
\textbf{Notation} & \textbf{Meaning} \\
\hline
\endhead

\hline
\endfoot

\multicolumn{2}{l}{\textbf{Function spaces, norms, and basic objects}} \\[1mm]

$[0,1]$ 
& The common domain on which the input and output functions are defined after normalization. \\

$C[0,1]$ 
& The space of continuous real-valued functions on $[0,1]$. \\

$\bbL^2[0,1]$ 
& The Hilbert space of square-integrable real-valued functions on $[0,1]$. \\

$\|x\|_{\sup,S}$ 
& The sup norm of a function $x$ over a set $S\subset[0,1]$:
$
\|x\|_{\sup,S}=\sup_{s\in S}|x(s)|.
$ \\

$\|x\|_{\sup}$ 
& The full sup norm on $[0,1]$:
$
\|x\|_{\sup}=\sup_{s\in[0,1]}|x(s)|.
$ \\

$\|f\|_{\bbL^2}$
& The standard $L^2$ norm of $\bbL^2[0,1]$:
$
\|f\|_{\bbL^2}^2=\int_0^1 f(t)^2 \,dt.
$ \\

$\|G\|_{u,\infty}$ 
& The operator sup norm restricted to the input ball $\{x:\|x\|_{\sup}\le u\}$:
$
\|G\|_{u,\infty}
=
\sup_{\|x\|_{\sup}\le u,\;t\in[0,1]} |G(x)(t)|.
$ \\[3mm]

\multicolumn{2}{l}{\textbf{Function-on-function regression model}} \\[1mm]

$X$ 
& Functional predictor. In the main theoretical development, $X$ is treated as a random element of $C[0,1]$. \\

$Y$ 
& Functional response. \\

$\varepsilon$ 
& Functional error process in the regression model. \\

$G_0$ 
& The true regression operator in the model
$
Y(t)=G_0(X)(t)+\varepsilon(t),\quad t\in[0,1].
$ \\

$C_X(s,t)$ 
& Covariance kernel of the predictor process:
$
C_X(s,t)=\E\{X(s)X(t)\}
$ (wlog assuming $E(X)\equiv 0$).\\

$C_\varepsilon(s,t)$ 
& Covariance kernel of the error process:
$
C_\varepsilon(s,t)=\E\{\varepsilon(s)\varepsilon(t)\}.
$ \\

$B^2_\varepsilon$
& $B^2_\varepsilon:=\sup_{t\in[0,1]} C_\varepsilon(t,t)<\infty$.\\

$d_X(s,t)$ 
& Canonical metric associated with $X$:
$
d_X(s,t)=\E^{1/2}\{(X(s)-X(t))^2\}.
$ \\

$\lambda,\beta$ 
& Constants appearing in the canonical-metric regularity condition:
$
d_X(s,t)\le \lambda |s-t|^\beta.
$ \\

$\alpha,\nu,C_0$ 
& Constants in the H\"older continuity and polynomial growth conditions on $G_0$:
$
\|G_0(x_1)-G_0(x_2)\|_{\sup}
\le C_0 u^{\nu} \|x_1-x_2\|_{\sup}^{\alpha}$
for $\|x_1\|_{\sup},\|x_2\|_{\sup}\le u$, and
$\|G_0(x)\|_{\sup}
\le C_0(1\vee\|x\|_{\sup}^{\nu})$ for $x\in C[0,1]$. \\[3mm]

\multicolumn{2}{l}{\textbf{Separable neural-operator representation}} \\[1mm]

$G_{\theta,\eta}$ 
& A separable neural operator with trainable parameters $\theta$ and $\eta$. \\

$p$ 
& Rank or number of separable terms in the SNO representation. \\

$c_k(x;\theta)$ 
& The $k$-th coefficient component, depending on the input function $x$ and parameterized by the coefficient network. \\

$\phi_k(t;\eta)$ 
& The $k$-th basis component, depending on the output argument $t$ and parameterized by the basis network. \\

$G_{\theta,\eta}(x)(t)$ 
& The SNO representation:
$
G_{\theta,\eta}(x)(t)
=
\sum_{k=1}^{p} c_k(x;\theta)\phi_k(t;\eta).
$ \\

$\theta$ 
& Parameters of the coefficient network. \\

$\eta$ 
& Parameters of the basis network. \\

$\sigma$ and $C_\sigma$
& $\sigma$ is an activation function. In the theoretical sieve analysis, $\sigma$ is assumed to be a Tauber--Wiener activation with an essentially uniformly bounded first derivative  $C_\sigma:=\esssup_{t\in\mathbb{R}} |\sigma'(t)|$. \vskip.1cm
\\ [3mm]

\multicolumn{2}{l}{\textbf{Sampling design for the theoretical analysis}} \\[1mm]

$n$ 
& Number of independent training pairs $(X_i,Y_i)$. \\

$S_n$ 
& Input sampling grid:
$
S_n=\{s_{n,j}:j\in[J_n]\}.
$ \\

$T_n$ 
& Output sampling grid:
$
T_n=\{t_{n,\ell}:\ell\in[L_n]\}.
$ \\

$J_n$ 
& Number of input-grid points in the theoretical sampling design. \\

$L_n$ 
& Number of output-grid points in the theoretical sampling design. \\

$s_{n,j}$ 
& The $j$-th input-grid point. \\

$t_{n,\ell}$ 
& The $\ell$-th output-grid point. \\

$\rho_n$ 
& Maximal input-grid spacing:
$
\rho_n=\max_{0\le j\le J_n}(s_{n,j+1}-s_{n,j}),
$
with $s_{n,0}=0$ and $s_{n,J_n+1}=1$. \\

$\gamma_n$ 
& Maximal output-grid spacing:
$
\gamma_n=\max_{0\le \ell\le L_n}(t_{n,\ell+1}-t_{n,\ell}),
$
with $t_{n,0}=0$ and $t_{n,L_n+1}=1$. \\

$X_{i,j}$ 
& Observed value of the $i$-th predictor curve at the $j$-th input-grid point:
$
X_{i,j}=X_i(s_{n,j}).
$ \\

$Y_{i,\ell}$ 
& Observed value of the $i$-th response curve at the $\ell$-th output-grid point with measurement error: $ Y_{i,\ell} := Y_i(t_{n,\ell}) + e_{i,\ell}$. \\

$e_{i,\ell}$ 
& Scalar measurement error in the discretely observed response. \\

$\sigma_e^2$
& Uniform bound for the measurement error variances.\\

$\widetilde X_n(s)$
& Kriging predictor, or equivalently the $L^2$ projection, of $X(s)$ based on the sampled predictor values on the grid $S_n$, $\{X(s):s\in S_n\}$. \\

$R_n(s)$
& Kriging residual process of the predictor:
$R_n(s):=X(s)-\widetilde X_n(s)$. \\ [3mm]

\multicolumn{2}{l}{\textbf{Truncation and truncated targets}} \\[1mm]
$u$ 
& Truncation radius used to restrict the input-function domain. \\

$V_{i,n,u}$ 
& Truncation indicator:
$
V_{i,n,u}
=
\mathbf 1\{\|X_i\|_{\sup,S_n}\le u\}.
$ \\

$G_{0,u}$ 
& Ideal truncated target based on the full sup norm:
$
G_{0,u}(x)
=
\mathbf 1\{\|x\|_{\sup}\le u\}G_0(x).
$ \\

$G_{0,n,u}$ 
& Grid-based truncated target:
$
G_{0,n,u}(x)
=
\mathbf 1\{\|x\|_{\sup,S_n}\le u\}G_0(x).
$ \\[3mm]

\multicolumn{2}{l}{\textbf{Theoretical sieve, approximation, and complexity quantities}} \\[1mm]
$\mathcal G_{n,u}$ 
& The shallow separable sieve class used in the theoretical analysis. It serves as a tractable surrogate for the practical SNO estimator. \\

$\widehat G_{n,u}$ 
& Empirical risk minimizer over $\mathcal G_{n,u}$. \\

$p_n$ 
& Number of separable terms in the theoretical sieve representation. \\

$q_n$ 
& Number of hidden units in each coefficient component of the theoretical sieve representation. \\

$r_n$
& Sparsity level controlling the number of input locations with nonzero weights in each hidden unit of the coefficient components of $G$ in $\CG_{n,u}$. \\

$b_n$ 
& Uniform bound on the parameters in the theoretical sieve class. \\

$B$ 
& The uniform bound constant used in the sieve network class
$
\|G\|_{u,\infty}\le B u^\nu.
$ \\

$B_u$
& A $u$-dependent constant:
$B_u=2(B+C_0)u^\nu$. \\

$\Delta_{n,u,\epsilon}$ 
& Combined input-grid approximation and truncation error, with tolerance $\epsilon$. \\

$\Omega_{n,u,\delta,\epsilon}$
& Combined sieve-complexity and approximation-error:
$
\Omega_{n,u,\delta,\epsilon}
=
\sqrt{n^{-1}\log\CN_{n,u,\delta}}
\vee
\sqrt{\Delta_{n,u,\epsilon}}.
$ \\

$G_{m,u,\epsilon}$
& $G_{m,u,\epsilon}(x):=\ind{\|x\|_{\sup,S_m}\le u}\,G(x)$, where $G(x)$ is defined in \eqref{e:phin0}. \\ 

$m_n$, or $m_{n,u,\epsilon}$ in full notation
& $\max\{m\le n: G_{m,u,\epsilon}\in\CG_{n,u}\}$. \\

$\sigma_s$
& Bound on the activation function over the range $[-s,s]$: $\sigma_s :=
\sup_{|t|\le s}|\sigma(t)|$.
\\

$\mathcal N_{n,u,\delta}$ 
& Covering number of $\mathcal G_{n,u}$ with radius $\delta$ under the truncated operator norm $\|\cdot\|_{u,\infty}$. \\

$\delta$, $\delta_n$ 
& Covering radius and covering radius sequence. \\[3mm]

\multicolumn{2}{l}{\textbf{Risk and loss functions in theoretical analysis}} \\[1mm]

$\|G(x)\|_{\bbL^2}$
& Standard $\bbL^2[0,1]$ norm of $G(x)$:
$
\|G(x)\|_{\bbL^2}^2
=
\int_0^1 \{G(x)(t)\}^2\,dt.
$ \\

$\|G(x)\|_L$ 
& Discrete weighted norm on the output grid:
$
\|G(x)\|_L^2
=
\sum_{\ell=1}^L
w_\ell \{G(x)(t_\ell)\}^2.
$ \\

$\|G\|_n$ 
& Empirical norm over the training sample and output grid:\\
&$
\|G\|_n^2
=
\frac{1}{n}\sum_{i=1}^n
\sum_{\ell=1}^{L_n}
w_\ell \{G(X_i)(t_{n,\ell})\}^2.
$ \\

$w_\ell, w_{i,\ell}$ 
& Quadrature weights associated with the common output-grid points $t_\ell$ in the theoretical analysis, and those with the curve-specific output-grid points $t_{i,\ell}$ in implementation, respectively. \\

$\CR(\widehat G_{n,u},G_{0,n,u})$ 
& Prediction risk for estimating the truncated target:\\
&$
\CR(\widehat G_{n,u},G_{0,n,u})
=
\E\left[
\|\widehat G_{n,u}(X)-G_{0,n,u}(X)\|_{\bbL^2}^2
\right],
$
where $X$ is an independent copy of the predictor, and the expectation is over the training data and $X$. \\

$\widehat R_n(\widehat G_{n,u},G_{0,n,u})$
& Intermediate risk based on the empirical norm $\|\cdot\|_n$:\\
&$
\widehat R_n(\widehat G_{n,u},G_{0,n,u})
=
\E\left[
\|\widehat G_{n,u}-G_{0,n,u}\|_n^2
\right].
$ \\

$R(\widehat G_{n,u},G_{0,n,u})$
& Intermediate risk based on the discrete output-grid norm $\|\cdot\|_L$:\\
&$
R(\widehat G_{n,u},G_{0,n,u})
=
\E\left[
\|\widehat G_{n,u}(X)-G_{0,n,u}(X)\|_L^2
\right],
$
where $X$ is an independent copy of the predictor, and the expectation is over the training data and $X$. \vskip.1cm
\\[3mm]

\multicolumn{2}{l}{\textbf{Implementation notation}} \\[1mm]

$\{s_{i,j}\}_{j=1}^{J_i}$ 
& Subject-specific irregular grid on which the $i$-th input function is observed. \\

$\{t_{i,\ell}\}_{\ell=1}^{L_i}$ 
& Subject-specific irregular grid on which the $i$-th response function is observed. \\

$J_i$ 
& Number of observed input locations for subject $i$ in the implementation.  \\

$L_i$ 
& Number of observed output locations for subject $i$ in the implementation. \\

notation remark
& Here $s_{i,j}$, $t_{i,\ell}$, $J_i$, and $L_i$ denote subject-specific
sampling grids and their corresponding numbers of observation points.
They should not be confused with the common theoretical grids
$\{s_{n,j}\}_{j=1}^{J_n}$, $\{t_{n,\ell}\}_{\ell=1}^{L_n}$ and their
sizes $J_n$ and $L_n$, used in the theoretical analysis. \\

$B_k$ 
& The $k$-th cubic B-spline basis function used to encode the input curve. \\

$n_B$ 
& Number of B-spline basis functions used for input encoding. \\

$c^{(x)}_{i,k}$ 
& Estimated coefficient of the $k$-th B-spline basis function for the $i$-th input curve. \\

$x_i^{\mathrm{feat}}$ 
& Fixed-dimensional feature vector of the $i$-th input curve:
$
x_i^{\mathrm{feat}}
=
(c^{(x)}_{i,1},\ldots,c^{(x)}_{i,n_B})^\top.
$ \\

$\mathrm{CoefNet}$ 
& Coefficient network mapping $x_i^{\mathrm{feat}}$ to the vector of input-dependent coefficients. \\

$\mathrm{BasisNet}$ 
& Basis network mapping an output argument $t$ to the vector of basis functions. \\

$\gamma(t)$ 
& Optional Fourier-feature representation of the output argument:\\
&$
\gamma(t)
=
(t,\sin(2\pi t),\ldots,\sin(2\pi n_{\mathrm{freq}}t),
\cos(2\pi t),\ldots,\cos(2\pi n_{\mathrm{freq}}t))^\top.
$ \\

$n_{\mathrm{freq}}$ 
& Number of Fourier frequencies used in the basis-network input. \\


$\widehat Y_i(t)$ 
& Predicted response value for subject $i$ at location $t$. \\


$\CL_n(\theta,\eta)$ 
& Empirical integrated squared loss used in the practical SNO training procedure, $\CL_n(\theta,\eta)
= \frac{1}{n} \sum_{i=1}^n \sum_{\ell=1}^{L_i} w_{i,\ell}\,
  \left(Y_{i,\ell} - \widehat Y_{i,\ell}\right)^2$. \\

$\mathrm{iRMSE}$ 
& Root mean integrated squared error used for test evaluation:\\&
$
\mathrm{iRMSE}
=
\left[
\E_X\int_0^1
\{\widehat Y(X)(t)-G_0(X)(t)\}^2\,dt
\right]^{1/2}.
$ \\[3mm]

\multicolumn{2}{l}{\textbf{Argo-data notation}} \\[1mm]

$X_i^{\mathrm{temp}}(s)$ 
& Temperature profile for the $i$-th Argo profile. \\

$X_i^{\mathrm{psal}}(s)$ 
& Salinity profile for the $i$-th Argo profile. \\

$X_i^{\mathrm{doxy}}(s)$ 
& Dissolved-oxygen profile for the $i$-th Argo profile. \\

$s$ 
& Pressure or normalized pressure/depth coordinate in the Argo application. \\

\hline
\end{longtable}

\section*{Code Availability}

The code for implementing computations in this paper is available at \url{https://github.com/senyuan-juncheng/SNO}.

\section*{Acknowledgments}
The authors sincerely thank Aniruddha Rao and Matthew Reimherr for their kind assistance on implementing the code used in \cite{rao2023modern}.

\bibliographystyle{abbrvnat}
\footnotesize
\bibliography{reference}
    
\end{document}